\RequirePackage{amsmath}
\RequirePackage{fix-cm}
\documentclass[smallextended]{svjour3hack}       % onecolumn (second format)
\usepackage{cases}
\usepackage{graphicx}
\usepackage[dvipsnames]{xcolor}
\usepackage{amsfonts,amssymb,amsmath}
\usepackage{mathtools}
\usepackage{booktabs}
\usepackage{placeins}

\usepackage{subcaption}
\usepackage{mwe}

\usepackage{array}
\newcolumntype{L}[1]{>{\raggedright\let\newline\\\arraybackslash\hspace{0pt}}m{#1}}
\newcolumntype{C}[1]{>{\centering\let\newline\\\arraybackslash\hspace{0pt}}m{#1}}
\newcolumntype{R}[1]{>{\raggedleft\let\newline\\\arraybackslash\hspace{0pt}}m{#1}}
\usepackage[ruled,vlined,linesnumbered,norelsize]{algorithm2e}

\usepackage[hypertexnames=false,colorlinks=true,breaklinks=true,bookmarks=true,urlcolor=blue,citecolor=blue,linkcolor=blue,bookmarksopen=false,draft=false]{hyperref}

\smartqed  % flush right qed marks, e.g. at end of proof

\newtheorem{thm}{Theorem}{\bf}{\rm}
\newtheorem{cor}[thm]{Corollary}{\bf}{\rm}
\newtheorem{lem}[thm]{Lemma}{\bf}{\rm}
\newtheorem{prop}[thm]{Proposition}{\bf}{\rm}
\newtheorem{defn}[thm]{Definition}{\bf}{\rm}
{\bf}{\rm}
{\bf}{\rm}

\DeclareMathOperator{\Diag}{Diag}
\DeclareMathOperator{\diag}{diag}
\DeclareMathOperator{\ldet}{ldet}

\DeclareMathOperator{\Trace}{tr}
\DeclareMathOperator{\rank}{rank}
\DeclareMathOperator{\dom}{dom}

\usepackage[mathscr]{euscript}
\makeatletter
\renewcommand*{\top}{%
  {\mathpalette\@transpose{}}%
}
\newcommand*{\@transpose}[2]{%
  \raisebox{\depth}{$\m@th#1\mathsf{T}$}%
}
\makeatother
\journalname{XXX}

\allowdisplaybreaks

\begin{document}

\title{Generalized Scaling for the Constrained Maximum-Entropy Sampling Problem\footnote{A short preliminary version of this paper appeared in ACDA 2023, \cite{ACDA2023}.}}%
\author{\hbox{Zhongzhu Chen  \and Marcia Fampa  \and Jon Lee}} %

\titlerunning{Generalized scaling for CMESP}

\institute{Z. Chen \at
              University of Michigan, Ann Arbor, MI, USA \\
              \email{zhongzhc@umich.edu}
              \and
           M. Fampa \at
              Universidade Federal do Rio de Janeiro, Brazil \\
              \email{fampa@cos.ufrj.br}           %  \\
           \and
            J. Lee \at
              University of Michigan, Ann Arbor, MI, USA \\
              \email{jonxlee@umich.edu}        %  \\
}

\authorrunning{Chen,   Fampa \& Lee}

%\date{\today}%
\date{\today}

\maketitle

\begin{abstract}
The best practical techniques for exact solution of instances of the constrained maximum-entropy sampling problem, a discrete-optimization problem arising in the design of experiments, are via a branch-and-bound
framework, working with a variety of concave continuous relaxations of the objective function. A standard and computationally-important  bound-enhance\-ment technique in this context is \emph{(ordinary) scaling},
via a single positive parameter. Scaling adjusts the shape of continuous relaxations to reduce the gaps between the upper bounds and the optimal value. We extend this technique to \emph{generalized scaling}, employing a positive vector of parameters, which allows much more flexibility and thus potentially reduces the gaps further. We give mathematical results aimed at supporting algorithmic methods for computing optimal generalized scalings, and we give computational results demonstrating the performance of generalized scaling on benchmark problem instances.
\keywords{ordinary scaling \and generalized scaling \and maximum-entropy sampling problem \and convex optimization}
\subclass{90C25 \and 90C27 \and 90C51 \and 62K99 \and 62H11}
\end{abstract}

%%%%%%%%%%%%%%%%%%%%%%%%%%%%%%%%%%%%%%%%%%%%%%%%%%%%%%%%%%%%%%%%%%%%%%%%%%%%%%%%%%%%%%%%%%%%%%%%%%%%%%%%%%%%%%

\section{Introduction}\label{sec:intro}
Let $C$ be a symmetric positive semidefinite matrix with rows/columns
indexed from $N:=\{1,2,\ldots,n\}$, with $n >1$.
Let $A\in\mathbb{R}^{m\times n}$ and $b\in\mathbb{R}^m$.
For $0< s < n$,
we define the \emph{constrained maximum-entropy sampling problem}
\begin{equation}\tag{CMESP}\label{CMESP}
\begin{array}{l}
z(C,s,A, b):=\\
\qquad \max \left\{\strut\textstyle  \ldet C[S(x),S(x)] ~:~ \mathbf{e}^\top x =s,~ x\in\{0,1\}^n,~Ax\leq b\strut \right\}
\end{array}
\end{equation}
where $S(x)$ is the support of $x\in\{0,1\}^n$,
 $C[S(x),S(x)]$ is the principal submatrix of $C$ indexed by $S(x)$, and
$\ldet(\cdot)$ is the natural logarithm of the determinant.
See \cite{FLbook} for an extensive and recent compilation of research on \ref{CMESP}.

When there are no
constraints $Ax\leq b$ present,
we refer to \ref{CMESP} as \hypertarget{MESP}{MESP},
which was introduced in the ``design of experiments'' literature by \cite{SW}, and was first approached for global optimization by \cite{KLQ}.
\hyperlink{MESP}{MESP} corresponds to the fundamental problem of choosing an $s$-subvector of
a Gaussian random $n$-vector, so as to maximize the ``differential entropy'', which is a measure of information (see \cite{Shannon}). We assume that $r:=\rank(C)\geq s$, so that
\hyperlink{MESP}{MESP}  always has a feasible solution with finite objective value.
\hyperlink{MESP}{MESP} has been applied extensively in the field of
environmental monitoring; see \cite[Chapter 4]{FLbook}, and the many references therein.
Important for applications, the constraints $Ax\leq b$ of \ref{CMESP}
can model budget limitations that take into account the differing cost of observing random variables, geographical considerations (such as
having lower and upper bounds on the number of observation in a region), and logical dependencies (such as not observing both of a given pair of random variables, due to technology constraints), for example.
More broadly, \hyperlink{MESP}{MESP} sits within the field of ``optimal design of experiments'', where
the handling of side constraints is a central topic (see \cite{DOE1,DOE2,DOE3}, for example).

\ref{CMESP} was first algorithmically approached by \cite{LeeConstrained}
and then by \cite{FampaPhD,WilliamsPhD,AFLW_IPCO,AFLW_Using}.
It is implicitly considered in all of the mathematical-programming approaches to \hyperlink{MESP}{MESP}, and it is explicitly
mentioned in \cite{Anstreicher_BQP_entropy,Kurt_linx}, for example.
\ref{CMESP} serves as a nice example of a ``non-factorable''
mixed-integer nonlinear program.
When $C$ is a diagonal matrix,
\ref{CMESP} reduces to a general cardinality-constrained
binary linear program.  \cite{AlThaniLeeTridiag,AlThaniLeeTridiag_journal} established that when $C$ is tridiagonal (or even when the support graph of $C$ is a spider with a bounded number of legs), \hyperlink{MESP}{MESP} is then polynomially solvable
by dynamic programming.

\cite{KLQ} established that \hyperlink{MESP}{MESP} is NP-hard
and introduced a novel B\&B (branch-and-bound) approach
based on a ``spectral bound.'' \cite{LeeConstrained}
extended the spectral approach to \ref{CMESP}.
\cite{AFLW_IPCO}  and \cite{AFLW_Using}
developed a bound, the so-called ``NLP bound'', employing a novel convex relaxation.
\cite{Anstreicher_BQP_entropy} developed the  ``BQP bound'', using an extended formulation based on the Boolean quadric polytope. \cite{Kurt_linx} introduced the ``linx bound'', based on a clever convex relaxation.
\cite{nikolov2015randomized} gave a novel
 ``factorization bound'' based on a subtle convex
 relaxation. This was further developed by \cite{li2020best} and then \cite{chen2023computing} (who coined the term ``factorization bound'').
  \cite{chen_mixing} gave a methodology for combining multiple
 convex-optimization bounds to give improved bounds.
 All of these convex-optimization based bounds admit variable-fixing methodology based on convex duality (see \cite{FLbook}, for example).
In computational  practice, the best bounds appear to be the linx bound, the factorization bound, and the NLP bound.
% \zc{If we do not experiment on NLP bound later, why we emphasize that it is state-of-art here?}
% response by JL: Well these is resonable evidence that the bound is good sometimes. Unfortunatley we don't have a convexity result for o-scaling it, so we didn't pursue it here.
For \hyperlink{MESP}{MESP}, the spectral bound  is dominated by the factorization bound (see \cite{chen2023computing}), and the
BQP bound is generally too time-consuming to compute, at present.
In what follows, we concentrate on the linx, BQP and factorization bounds; these are all within the state-of-the-art, and for all we could obtain substantial and relevant results.

An important general technique for potentially improving some of the
entropy upper bounds is ``scaling'',
based on the simple observation that
for a positive constant $\gamma$, and $S$
with $|S|=s$, we have that
\[
\det (\gamma C)[S,S] = \gamma^s \det C[S,S]~.
\]
With this identity,  we can easily see that
\[
z(C,s, A, b) = z(\gamma C,s, A, b) -s \log \gamma~,
\]
so upper bounds for $z(\gamma C,s, A, b)$
 yield upper bounds for $z(C,s, A, b)$,
shifting by $-s \log \gamma$.
It is important to note that many bounding methods are \emph{not} invariant under
scaling; that is the scaled bound does \emph{not} generally  shift by $-s\log \gamma$ (notable exceptions being the spectral and factorization bounds for \hyperlink{MESP}{MESP}).
Scaling was first introduced in \cite{AFLW_IPCO,AFLW_Using}, and then exploited in \cite{Anstreicher_BQP_entropy,Kurt_linx,Al-ThaniLee1,chen_mixing,chen2022masking,chen2023computing}. Scaling can be seen as a technique aimed at adjusting the shape of  concave continuous relaxations of the objective of \ref{CMESP} in order to decrease the gap between the upper bounds and $z(C,s, A, b)$; see \cite{chen2022masking} for an exploration of this in the context of the linx bound.

% In this work, we generalize the idea of scaling to the vector case where we denote $\Upsilon\in \mathbb{R}^n_{++}:=\{\gamma_1, \cdots, \gamma_n\}^\top $ and $t_1, t_2 \in \mathbb{R}_+$. With $|S| = s$, we have that
% \begin{align}\label{gscale:eqn}
%     \det \left(\Diag(\Upsilon)^{t_1} C \Diag(\Upsilon)^{t_2}\right)[S,S] =\bigg( \prod_{i\in S} \gamma_i^{t_1+t_2}\bigg) \cdot \det C[S,S]~
% \end{align}
% where $\Diag(\Upsilon)$ is the diagonal matrix with $\Upsilon$ as diagonals. With this identity, we have that
% \begin{align*}
%     z(C,s, A, b) = z(\Diag(\Upsilon)^{t_1} C \Diag(\Upsilon)^{t_2},s, A, b) - (t_1+t_2)\cdot\bigg( \sum_{i\in S} \log \gamma_i\bigg) ~.
% \end{align*}
% Therefore, upper bounds for $z(\Diag(\Upsilon)^{t_1} C \Diag(\Upsilon)^{t_2},s, A, b)$
%  yield upper bounds for $z(C,s, A, b)$,
% shifting by $-(t_1+t_2)\cdot\big( \sum_{i\in S} \log \gamma_i\big)$. We can facilitate the theoretical analysis of $\Upsilon$ for each upper bound by selecting distinct values for $t_1$ and $t_2$. Although symmetry is preserved by setting $t_1=t_2$, it may be necessary to designate $t_1\neq t_2$ in cases where it does not compromise the symmetry of the upper bound equation but make analysis easier, such as in the linx bound.

% One can observe that $t$ can be merged into $\Upsilon$ without affecting the results. However, as we will see later that by setting $t$ to different values for different upper bounds, we can make the theoretical analysis on $\Upsilon$ easier.

In this work, we generalize the idea of scaling to the vector case and apply it to three different upper bounds: the BQP bound, as well as the state-of-the-art linx and factorization bounds. Throughout, we let  $\Upsilon:=\left(\gamma_1,\gamma_2,\ldots,\gamma_n\right)^\top \in \mathbb{R}_{++}^n$ be a ``scaling vector.''
We refer to our technique as \emph{g-scaling} (i.e., \emph{general scaling}) and the corresponding bounds as \emph{g-scaled} (i.e., \emph{generalized scaled}), and when all elements
of $\Upsilon$ are equal, we say \emph{o-scaling} (i.e., \emph{ordinary scaling}) and \emph{o-scaled} (i.e., \emph{ordinary scaled}).
If all elements of $\Upsilon$ are equal to 1,
we say \emph{unscaled}. In general, setting all of the elements of $\Upsilon$ to be equal,  g-scaling reduces to o-scaling.
This means that g-scaling can provide an upper bound that is at least as good as o-scaling. Moreover, as we will see later, g-scaling can sometimes provide significantly improved upper bounds compared to o-scaling.

 Additionally in this work, we also leverage another key technique for obtaining bounds, namely  ``complementation'', first utilized by \cite{AFLW_IPCO,AFLW_Using}.
If $C$ is invertible,
 %%%%%%%%%%%%%%%%
we have
\[
z(C,s, A, b)=z(C^{-1},n-s,-A,b-A\mathbf{e}) + \ldet C,
\]
\noindent where $z(C^{-1},n-s,-A,b-A\mathbf{e})$ denotes the
optimal value of \ref{CMESP} with $C,s,A,b$ replaced by
$C^{-1},n-s,-A,b-A\mathbf{e}$, respectively.
So we have a
\emph{complementary} \ref{CMESP} problem
 and \emph{complementary} bounds
 (i.e., bounds for the
 complementary problem plus $\ldet C$)
 immediately give us bounds on
 $z$.
 Some upper bounds on $z$ also shift by  $\ldet C$  under complementing (notably, the spectral and linx bounds),
 in which case there is no additional value in computing the
 complementary bound. But the NLP, BQP and factorization bounds are generally not invariant under complementation.
 Details on all of this can be found in \cite{FLbook}.

% \subsection{Terminology and motivation.}
% We let $\mathbb{S}^n$, $\mathbb{S}^n_+$, and $\mathbb{S}^n_{++}$ to denote the set of symmetric matrices, positive semidefinite matrices, and positive definite matrices of order $n$, respectively. We let $\mathbb{R}^n$, $\mathbb{R}^n_+$, and $\mathbb{R}^n_{++}$ to denote the set of vectors of order $n$ with any elements, nonnegative, and positive elements, respectively.
% We denote an all-ones  vector
% by $\mathbf{e}$ and the $i$-th standard unit vector by $\mathbf{e}_i$\thinspace.
% For matrices $A$ and $B$ with the same shape,
% $A\circ B $ is the Hadamard (i.e., element-wise) product,
% and $A\bullet B:=\Trace(A^\top B)$ is the matrix dot-product.
% For a matrix $A$, we denote row $i$ by $A_{i\cdot}$ and
% column $j$ by $A_{\cdot j}$\thinspace.

% \noindent Throughout, we let  $\Upsilon:=\left(\gamma_1,\gamma_2,\ldots,\gamma_n\right)^\top \in \mathbb{R}_{++}^n$ be a ``scaling vector''.
% We refer to our bounds as \emph{g-scaled} (i.e., \emph{generalized scaled}), and when all elements
% of $\Upsilon$ are equal, we  say  \emph{o-scaled} (i.e., \emph{ordinary scaled}).
% If all elements of $\Upsilon$ are equal to 1,
% we say \emph{un-scaled}.
% o-scaling is very important for computing good upper bound values
% (for the NLP bound, the BQP bound, and the linx bound), in the context of B\&B (see \cite{AFLW_Using,Anstreicher_BQP_entropy,Kurt_linx}). So our motivation in considering generalized scaling is to improve on the success of o-scaling.

\subsection{Organization and contributions.}
In \S\ref{sec:bqp}, we introduce the g-scaled BQP bound via symmetric scaling of $C$, and our substantial new result is its convexity in the log of the scaling vector,
generalizing an important and practically-useful result for o-scaling (see \cite[Theorem 11]{chen_mixing}).
In \S\ref{sec:linx}, we introduce the g-scaled
linx bound  via symmetric scaling of $C\Diag(x)C$,  and our substantial new result is its convexity in the log of the scaling vector,
generalizing another very important and practically-useful
result for o-scaling (see \cite[Theorem 18]{chen_mixing}). We wish to emphasize that
the construction of the g-scaling versions of
the BQP and linx bounds are different (unlike in the case of o-scaling) because this is what is needed to gain the convexity results.
These convexity results are key
for the tractability of globally optimizing the scaling, something that we
do not have for general  bound ``masking''\footnote{This is a related bound-improvement  technique where we preprocess $C$ by taking its Hadamard product with a correlation matrix.} (see \cite{AnstreicherLee_Masked,BurerLee} for this, in the context of the spectral bound).
In \S\ref{sec:fact}, we introduce the ``g-scaled
factorization bound.''
We note that there is no o-scaling for the factorization bound, as that bound is invariant under such a scaling (see \cite[Theorem 2.1]{chen2023computing}),
so scaling in this context is completely novel.
In this section, which is one of our major contributions, we  establish ``generalized differentiability'' results for the factorization bound (for the first time, even without the context of scaling), which are essential for the fast and stable calculation of the factorization bound and for globally optimizing the scaling vector, \emph{even using general-purpose nonlinear-optimization software.}
The prior literature on this type of bound focused on characterizing subdifferentials
(see \cite{nikolov2015randomized,li2020best}),
so our generalized differentiability results
were completely unanticipated.
In fact, \cite{chen2023computing})  empirically observed that general-purpose nonlinear-optimization software
works remarkably well, and we can see our generalized differentiability results as an explanation.
% \zc{Shall we say more here? For example, the "general differentiability" on the boundary.}
% J.L.: let's not --- we will get to into details for the introduction
 We are also able to establish that
for \hyperlink{MESP}{MESP}, the all-ones vector is a stationary
point for the factorization bound as a function of the scaling vector. Therefore,
in contrast to the BQP bound and the linx bound, g-scaling is unlikely to help the factorization bound for \hyperlink{MESP}{MESP}.
Despite this, through numerical experiments, we observe that g-scaling can significantly improve the factorization bound for \ref{CMESP},
while o-scaling cannot help it (as we mentioned above).
%
% In \S\ref{sec:fact}, we introduce the g-scaled
% factorization bound, and we establish that
% for \hyperlink{MESP}{MESP}, the all-ones vector is a stationary
% point for the bound as a function of the scaling vector, and therefore
% g-scaling is unlikely to be helpful for \hyperlink{MESP}{MESP}.
% Interestingly, o-scaling cannot help the factorization bound  for \ref{CMESP} (see \cite[Theorem 2.1]{Fact}), but we establish that g-scaling can significantly improve it.
In \S\ref{sec:num}, we present results of computational experiments, demonstrating the improvements on upper bounds
and on the number of variables that can be fixed (using convex duality)
due to g-scaling.
Further, we conducted experiments to
demonstrate the relevance of our generalized differentiability results
for the factorization bound.
In \S\ref{sec:conc}, we make some brief concluding remarks.

\subsection{Notation.}
$\Diag(x)\in\mathbb{R}^{n\times n}$ makes a diagonal matrix from
$x\in\mathbb{R}^n$, and $\diag(X)\in\mathbb{R}^n$ extracts the diagonal of $X\in\mathbb{R}^{n\times n}$ as an $n$-vector.
We let $\mathbb{R}^n_+$ (resp., $\mathbb{R}^n_{++}$) be the points in $\mathbb{R}^n$ having
 nonnegative (resp., positive) components.
We let $\mathbb{S}^n_+$ (resp., $\mathbb{S}^n_{++}$) be the set of
 positive semidefinite (resp., definite) symmetric matrices of order $n$.
 We let $\lambda_\ell(M)$ be the $\ell$-th greatest eigenvalue of
$M\in \mathbb{S}^n_+$\,. We denote by $\mathbf{e}^n$ an all-ones vector in $\mathbb{R}^n$, $e^n_i$ the
$i^{\text{th}}$ standard unit vector in $\mathbb{R}^n$, and $I_n$ the order-$n$ identity matrix.
For matrices $A$ and $B$ with the same shape,
$A\circ B $ is the Hadamard (i.e., element-wise) product,
and $A\bullet B:=\Trace(A^\top B)$ is the matrix dot-product.
For a matrix $A$, we denote row $i$ by $A_{i\cdot}$ and
column $j$ by $A_{\cdot j}$. We let $A^\dagger$  be the Moore-Penrose (generalized) inverse of $A$. We denote natural logarithm by $\log(\cdot)$,
and we apply it component-wise to vectors. We let $E^n_{ij}$ denote the order-$n$ square matrix with
the only nonzero component being a one in the
$(i,j)$ position. We frequently drop the $n$ (for all of these) when it is clear from the context. We use $\|\cdot\|$ for the vector 2-norm when applied to a vector and the 2-norm induced by the vector 2-norm when applied to a matrix. Given a convex function $f$, we use $\partial f(\cdot)$ for the subdifferential of $f$. Given a function $f$ and a direction $d$, we write $\partial f(\cdot; d)$ for the directional derivative of $f$ in the direction $d$.

The notations associated with each \ref{CMESP} upper bound will be introduced in the corresponding section. For simplicity, we will exclude $(C, s, A, b)$ from notations whenever it is clear from the context.

% {\color{blue}
% {\bf In case we need some of this:}
% $\Diag(X):=\Diag(\diag(X))$.
%  $\mathbf{e}_i$ is the $i$-th standard unit vector.
% $A\bullet B:=\Trace(A^\top B)$ is the matrix dot-product.
%  $A_{i\cdot}$ is row $i$ and $A_{\cdot j}$ is
% column $j$ of matrix $A$.
% }

\section{BQP bound}\label{sec:bqp}

The Boolean-Quadratic-Polytope (\ref{BQP}) bound was first suggested in 1995 by Christoph
Helmberg (in private communication to K. Anstreicher and J.Lee),
% suggested (essentially) 5 (see [Lee12, FL00]) to Anstreicher and Lee,
but no one developed it at that
time.
Finally, it was analyzed and developed in \cite{Anstreicher_BQP_entropy}
(see \cite[Section 3.6, p. 110 ]{FLbook} for more details). We
lift to matrix space, by defining  the convex  set %$P(n,s):=$
\begin{align*}
P(n,s):=\!\left\{
(x,X)\in\mathbb{R}^n\! \times \! \mathbb{S}^n \,:\,
 X \! - \!xx^\top\succeq 0,\, \diag(X)=x,\,
\mathbf{e}^\top x=s,\, X\mathbf{e}=sx
\right\}.
\end{align*}
For $\Upsilon \in \mathbb{R}_{++}^n$ and $(x, X)\in P(n,s)$, we now define
%$f_{{\tiny\mbox{BQP}}}(x,X;\Upsilon):=$
   \begin{align*}
   f_{{\tiny\mbox{BQP}}}(x,X;\Upsilon):=&
\textstyle \ldet \left(\strut \left(\Diag(\Upsilon)C\Diag(\Upsilon)\right)\circ X  + \Diag(\mathbf{e}-x)\right)\\
&\quad - \textstyle 2\sum_{i=1}^n  x_i\log \gamma_i \,,
	\end{align*}
\noindent with domain %$\dom\left(f_{{\tiny\mbox{BQP}}};\Upsilon\right):=$
\begin{align*}
\dom\left(f_{{\tiny\mbox{BQP}}};\Upsilon\right):=
  &\left\{\strut(x,X) \in\mathbb{R}^n\times \mathbb{S}^n~:~ \right.\\
 & \quad \left.\left(\strut\Diag(\Upsilon)C\Diag(\Upsilon)\right)\circ X  + \Diag(\mathbf{e}-x) \succ 0\right\}.
\end{align*}
The \emph{g-scaled BQP bound} is defined as
\[
\begin{array}{ll}
	& z_{{\tiny\mbox{BQP}}}(\Upsilon):=\max \left\{f_{{\tiny\mbox{BQP}}}(x,X;\Upsilon) ~:~ (x,X)\in P(n,s),~  Ax\leq b \right\}.\tag{BQP}\label{BQP}
	\end{array}
 \]
We say $x$ is feasible to \ref{BQP} if $x$ satisfies all the constraints in \ref{BQP}.

% Note that this is applying the BQP bound to the right-hand-side of the identity \eqref{gscale:eqn} where $t_1=t_2=1$, i.e.,
Note that we can interpret \ref{BQP} as applying the
unscaled BQP bound to the symmetri\-cally-scaled
matrix $\Diag(\Upsilon)C\Diag(\Upsilon)$, and then
correcting by $- 2\sum_{i=1}^n  \!x_i\log \gamma_i$~.

\begin{thm}\label{thm:bqp}
% \phantom{.}
% \vspace{-10pt}
For all $\Upsilon \in \mathbb{R}_{++}^n$\,, %in \ref{BQP},
the following holds:
\begin{itemize}
\item[\ref{thm:bqp}.i.] \label{bqp.i}  $z_{{\tiny\mbox{BQP}}}(\Upsilon)$ is a valid upper bound for the optimal value of \ref{CMESP}, i.e., $z(C,s,A,b)\leq z_{{\tiny\mbox{BQP}}}(\Upsilon)$;
\item[\ref{thm:bqp}.ii.] \label{bqp.ii} the function $f_{{\tiny\mbox{BQP}}}(x,X;\Upsilon)$ is concave in $(x,X)$ on $\dom\left(f_{{\tiny\mbox{BQP}}};\Upsilon\right)$ and continuously differentiable in $(x,X, \Upsilon)$ on
$\dom\left(f_{{\tiny\mbox{BQP}}};\Upsilon\right) \times \mathbb{R}_{++}^n$\,;
\item[\ref{thm:bqp}.iii.] \label{bqp.iii} for fixed $(x,X)\in \dom\left(f_{{\tiny\mbox{BQP}}};\Upsilon\right)$, $f_{{\tiny\mbox{BQP}}}(x,X;\Upsilon)$ is convex in $\log \Upsilon$, and thus $z_{{\tiny\mbox{BQP}}}(\Upsilon)$ is convex in $\log \Upsilon$.
\end{itemize}
\end{thm}

\begin{remark}
\noindent \cite{Anstreicher_BQP_entropy} established Theorem \ref{thm:bqp}.\emph{i} for $\Upsilon:=\gamma\mathbf{e}$, with $\gamma\in \mathbb{R}_{++}$\,. We generalize this result to $\Upsilon\in \mathbb{R}^n_{++}$\,. The concavity in Theorem \ref{thm:bqp}.\emph{ii}
is a result of \cite{Anstreicher_BQP_entropy}, with details
filled in by \cite[Section 3.6.1]{FLbook}. Theorem \ref{thm:bqp}.\emph{iii}
significantly generalizes a result of \cite{chen_mixing}, where it is
established only for o-scaling: i.e., on $\{ \Upsilon:= \gamma \mathbf{e} ~:~ \gamma\in \mathbb{R}_{++}\}$. The proof of Theorem \ref{thm:bqp}.\emph{iii} requires new techniques (see the proof below). Additionally, the result is quite important as it enables the use of readily available quasi-Newton methods (like BFGS) for finding the globally-optimal g-scaling vector for the \ref{BQP} bound.
\end{remark}

\begin{proof}[Theorem \ref{thm:bqp}]
% \phantom{.}
% \vspace{-7pt}

\begin{itemize}
    \item[\ref{thm:bqp}.i:]
    It is enough to prove that there is a feasible solution to \ref{BQP} with  objective value equal to the optimal value of \ref{CMESP}. In fact, let $x^* \in \{0,1\}^n$ be an optimal solution to \ref{CMESP} with support $S\left(x^*\right)$, and define $X^* := x^*\left(x^*\right)^\top$. Without loss of generality, we assume that $S\left(x^*\right)=\{1, \ldots, s\}$, i.e., $x^* \!= \!\left(\!\begin{array}{c}
         \mathbf{e}^{s}  \\
         0
    \end{array}\!\right)$ and $X^*\! = \!\left(\!\begin{array}{cc}
          I_{s} & 0 \\
          0 & 0
    \end{array}\!\right)$. Clearly, $(x^*, X^*)\in \dom\left(f_{{\tiny\mbox{BQP}}};\Upsilon\right)$ and is feasible to \ref{BQP}. Let $\Upsilon_{S\left(x^*\right)}$ be the sub-vector of $\Upsilon$ indexed by $S\left(x^*\right)$, then
    \begin{align*}
        &f_{{\tiny\mbox{BQP}}}\left(x^*,X^*;\Upsilon\right) \\
        &~ =  \textstyle \ldet \left( \strut\left(\strut\Diag(\Upsilon)C\Diag(\Upsilon)\right)\circ X^* + \Diag(\mathbf{e}-x^*)\right) - \textstyle 2\sum_{i=1}^n  x^*_i\log \gamma_i\\
& ~=\ldet \left( \begin{array}{cc}
    \Upsilon_{S\left(x^*\right)} C \left(S(x^*), S(x^*)\right) \Upsilon_{S\left(x^*\right)} &  0 \\
    0  &~ I_{n-s}
\end{array}\right)-\textstyle 2\sum_{i\in S(x^*)}  x^*_i\log \gamma_i\\
         &~=\ldet C \left(S(x^*), S(x^*)\right).
    \end{align*}
    % Thus $z_{{\tiny\mbox{BQP}}}(\Upsilon) \ge f_{{\tiny\mbox{BQP}}}\left(x^*,X^*;\Upsilon\right) = z$.
    \item[\ref{thm:bqp}.ii:] The concavity is essentially a  result of \cite{Anstreicher_BQP_entropy}, with details filled in by \cite[Section 3.6.1]{FLbook}. The continuous differentiability comes from the analyticity of $f_{{\tiny\mbox{BQP}}}\left(x,X;\Upsilon\right)$ in $(x,X,\Upsilon)\in \dom\left(f_{{\tiny\mbox{BQP}}};\Upsilon\right) \times \mathbb{R}_{++}^n$\,.
    \item[\ref{thm:bqp}.iii:] We sketch the proof first:
    \begin{enumerate}
        \item for fixed $(x,X)\in\dom\left(f_{{\tiny\mbox{BQP}}};\Upsilon\right)$, we derive the Hessian of $f_{{\tiny\mbox{BQP}}}\left(x,X;\Upsilon\right)$ with respect to $\log \Upsilon$ and show that it is positive-semidefinite, which implies the convexity of $f_{{\tiny\mbox{BQP}}}\left(x,X;\Upsilon\right)$ in $\log \Upsilon$;
        \item
        % $z_{{\tiny\mbox{BQP}}}(\Upsilon)$ is the point-wise maximum of $f_{{\tiny\mbox{BQP}}}\left(x,X;\Upsilon\right)$ over feasible $(x,X)$ for \ref{BQP} in domain $\dom\left(f_{{\tiny\mbox{BQP}}};\Upsilon\right)$. Thus,   the convexity of $f_{{\tiny\mbox{BQP}}}\left(x,X;\Upsilon\right)$ in $\log \Upsilon$ for all fixed $(x,X)$ implies
        The convexity of $z_{{\tiny\mbox{BQP}}}(\Upsilon)$  in $\log \Upsilon$ then follows because $z_{{\tiny\mbox{BQP}}}(\Upsilon)$ is the point-wise maximum of $f_{{\tiny\mbox{BQP}}}\left(x,X;\Upsilon\right)$ over feasible $(x,X)$ for \ref{BQP} in domain $\dom\left(f_{{\tiny\mbox{BQP}}};\Upsilon\right)$.
    \end{enumerate}

    \smallskip

    The detailed proof is as follows. For convenience, let
    \begin{align*}
    &F_{{\tiny\mbox{BQP}}}(x,X;\Upsilon):= \left(\strut\Diag(\Upsilon)C\Diag(\Upsilon)\right)\circ X+ \Diag(\mathbf{e}-x), \mbox{ and }\\
    &A_{{\tiny\mbox{BQP}}}(X;\Upsilon):=\left(\strut\Diag(\Upsilon)C\Diag(\Upsilon)\right)\circ X.
    \end{align*}
    In the following derivation, we will consider $(x,X)\in \dom\left(\strut f_{{\tiny\mbox{BQP}}};\Upsilon\right)$ fixed, and regard $\Upsilon$ as a variable. Thus, for simplicity, we  write $F_{{\tiny\mbox{BQP}}}(x,X;\Upsilon)$ and  $A_{{\tiny\mbox{BQP}}}(X;\Upsilon)$ as $F_{{\tiny\mbox{BQP}}}(\Upsilon)$ and $A_{{\tiny\mbox{BQP}}}(\Upsilon)$, respectively.

    Let $\check{x}:=x-\mathbf{e}$, and we  use the identities
    % $A_{{\tiny\mbox{BQP}}}(\Upsilon)F_{{\tiny\mbox{BQP}}}(\Upsilon)^{-1} = I +\Diag(\check{x}) F_{{\tiny\mbox{BQP}}}(\Upsilon)^{-1}$ and
      \begin{align}
     &F_{{\tiny\mbox{BQP}}}(\Upsilon)^{-1}A_{{\tiny\mbox{BQP}}}(\Upsilon) = I + F_{{\tiny\mbox{BQP}}}(\Upsilon)^{-1}\Diag(\check{x}), \label{FAbqp}\\ &A_{{\tiny\mbox{BQP}}}(\Upsilon)F_{{\tiny\mbox{BQP}}}(\Upsilon)^{-1} = I + \Diag(\check{x})F_{{\tiny\mbox{BQP}}}(\Upsilon)^{-1}.\label{AFbqp}
    \end{align}

  We first derive the gradient of $f_{{\tiny\mbox{BQP}}}\left(x,X;\Upsilon\right)$ with respect to $\Upsilon$.
    \begin{align*}
         &\textstyle\frac{\partial f_{{\tiny\mbox{BQP}}}\left(x,X;\Upsilon\right)}{\partial \gamma_i}
         = ~ F_{{\tiny\mbox{BQP}}}(\Upsilon)^{-1}\bullet \frac{\partial A_{{\tiny\mbox{BQP}}}(\Upsilon)}{\partial \gamma_i} -2\frac{x_i}{\gamma_i}\\
        & \qquad \qquad= ~F_{{\tiny\mbox{BQP}}}(\Upsilon)^{-1}\bullet \textstyle\frac{1}{\gamma_i}\left(E_{ii}A_{{\tiny\mbox{BQP}}}(\Upsilon)+A_{{\tiny\mbox{BQP}}}(\Upsilon) E_{ii}\right)-2\textstyle\frac{x_i}{\gamma_i}\\
        % =& ~F_{{\tiny\mbox{BQP}}}(\Upsilon)^{-1}\bullet \frac{1}{\gamma_i}\left(\mathbf{e}_i \mathbf{e}_i^\top A_{{\tiny\mbox{BQP}}}(\Upsilon)+A_{{\tiny\mbox{BQP}}}(\Upsilon) \mathbf{e}_i \mathbf{e}_i^\top\right)-2\frac{x_i}{\gamma_i}\\
    &\qquad\qquad = ~\textstyle\frac{1}{\gamma_i}\left(A_{{\tiny\mbox{BQP}}}(\Upsilon)_{i\cdot}F_{{\tiny\mbox{BQP}}}(\Upsilon)^{-1}_{\cdot i}+F_{{\tiny\mbox{BQP}}}(\Upsilon)^{-1}_{i\cdot}A_{{\tiny\mbox{BQP}}}(\Upsilon)_{\cdot i}-2x_i\right)\\
    &\qquad\qquad =~ \textstyle \frac{1}{\gamma_i}\left(2F_{{\tiny\mbox{BQP}}}(\Upsilon)^{-1}_{i\cdot}A_{{\tiny\mbox{BQP}}}(\Upsilon)_{\cdot i}-2x_i\right),
    \end{align*}
    where the last identity follows from the symmetry of  $F_{{\tiny\mbox{BQP}}}(\Upsilon)$ and $A_{{\tiny\mbox{BQP}}}(\Upsilon)$. Then, applying \eqref{FAbqp}, we obtain
    \begin{align*}
        &\textstyle\frac{\partial f_{{\tiny\mbox{BQP}}}\left(x,X;\Upsilon\right)}{\partial \Upsilon}=
         % = 2\Diag(\Upsilon)^{-1}\left(\diag\left(F_{{\tiny\mbox{BQP}}}(\Upsilon)^{-1}A_{{\tiny\mbox{BQP}}}(\Upsilon)\right)-x\right)\\
        %&\qquad \qquad= 2\Diag(\Upsilon)^{-1}\left(\diag\left(F_{{\tiny\mbox{BQP}}}(\Upsilon)^{-1}\left(F_{{\tiny\mbox{BQP}}}(\Upsilon)+\Diag(x-\mathbf{e}) \right)\right)-x\right)\\
    %&\qquad\qquad=
2\Diag(\Upsilon)^{-1}\left(\diag\left(F_{{\tiny\mbox{BQP}}}(\Upsilon)^{-1}\Diag(\check{x})\right)-\check{x}\right).
    \end{align*}
   Next, we derive the Hessian of $f_{{\tiny\mbox{BQP}}}\left(x,X;\Upsilon\right)$ with respect to $\Upsilon$. Note that
    \begin{align*}
    &\textstyle\frac{1}{2}\frac{\partial^2 f_{{\tiny\mbox{BQP}}}\left(x,X;\Upsilon\right)}{\partial \Upsilon \partial \gamma_i}
    = \frac{1}{2}\frac{\partial}{\partial\gamma_i}\left(\frac{\partial f_{{\tiny\mbox{BQP}}}\left(x,X;\Upsilon\right)}{\partial \Upsilon}\right)\\
    &~ = \textstyle \frac{\partial \Diag(\Upsilon)^{-1}}{\partial\gamma_i} \left(\diag\left(F_{{\tiny\mbox{BQP}}}(\Upsilon)^{-1}\Diag(\check{x})\right)-\check{x}\right)
    \\
    &~~~~  +\Diag(\Upsilon)^{-1}  \textstyle\frac{\partial \left(\diag\left(F_{{\tiny\mbox{BQP}}}(\Upsilon)^{-1}\Diag(\check{x})\right)-\check{x} \right)}{\partial\gamma_i} \\
    &=  \textstyle\frac{\partial \Diag(\Upsilon)^{-1}}{\partial\gamma_i} \left(\diag\left(F_{{\tiny\mbox{BQP}}}(\Upsilon)^{-1}\Diag(\check{x})\right)-\check{x}\right)\\
    &~~~~ -\Diag(\Upsilon)^{-1}  \diag\left(F_{{\tiny\mbox{BQP}}}(\Upsilon)^{-1} \textstyle\frac{\partial F_{{\tiny\mbox{BQP}}}(\Upsilon)}{\partial\gamma_i}F_{{\tiny\mbox{BQP}}}(\Upsilon)^{-1}\Diag(\check{x})\right) \\
    &= -\textstyle\frac{1}{\gamma_i^2}E_{ii}\left(\diag\left(F_{{\tiny\mbox{BQP}}}(\Upsilon)^{-1}\Diag(\check{x})\right)-\check{x}\right)\\
    &~~~~  - \Diag(\Upsilon)^{-1}\diag\left(\!F_{{\tiny\mbox{BQP}}}(\Upsilon)^{-1}\textstyle\frac{\left(E_{ii}A_{{\tiny\mbox{BQP}}}(\Upsilon)+A_{{\tiny\mbox{BQP}}}(\Upsilon) E_{ii}\right)}{\gamma_i}
    % &\qquad\left.\vphantom{\frac{\left(E_{ii}A+A E_{ii}\right)}{\gamma_i}} F_{{\tiny\mbox{BQP}}}(\Upsilon)^{-1}\Diag(\check{x})\!\right)\\
  F_{{\tiny\mbox{BQP}}}(\Upsilon)^{-1}\Diag(\check{x})\!\right).
%     &= -\textstyle\frac{1}{\gamma_i^2}E_{ii}\left(\diag\left(F_{{\tiny\mbox{BQP}}}(\Upsilon)^{-1}\Diag(\check{x})\right)-\check{x}\right)\\
%     &~~ -\textstyle\frac{1}{\gamma_i}\Diag(\Upsilon)^{-1}\diag\left(F_{{\tiny\mbox{BQP}}}(\Upsilon)^{-1}\left(E_{ii}A_{{\tiny\mbox{BQP}}}(\Upsilon)+A_{{\tiny\mbox{BQP}}}(\Upsilon) E_{ii}\right)
% %    &\qquad\left.\vphantom{\left(E_{ii}A+A E_{ii}\right)} F_{{\tiny\mbox{BQP}}}(\Upsilon)^{-1}\Diag(\check{x})\right).
%      F_{{\tiny\mbox{BQP}}}(\Upsilon)^{-1}\Diag(\check{x})\right).
\end{align*}
The first term in this last expression can be reformulated as
 \begin{align*}
    &-\textstyle\frac{1}{\gamma_i^2}E_{ii}\left(\diag\left(F_{{\tiny\mbox{BQP}}}(\Upsilon)^{-1}\Diag(\check{x})\right)-\check{x}\right)\\
     &~~= -  \Diag\left(\Upsilon\right)^{-1} \Diag\left(\diag\left(F_{{\tiny\mbox{BQP}}}(\Upsilon)^{-1}\Diag(\check{x})\right)-\check{x}\right)
     %&\quad
     \Diag\left(\Upsilon\right)^{-1} \mathbf{e}_i\,,
 \end{align*}
 while for the second term, we use \eqref{FAbqp} and \eqref{AFbqp}, and we obtain
 \begin{align*}
     &\diag\left(F_{{\tiny\mbox{BQP}}}(\Upsilon)^{-1}\left(E_{ii}A_{{\tiny\mbox{BQP}}}(\Upsilon)+A_{{\tiny\mbox{BQP}}}(\Upsilon) E_{ii}\right) F_{{\tiny\mbox{BQP}}}(\Upsilon)^{-1}\Diag(\check{x})\right)\\
     % &~=\diag\left(F_{{\tiny\mbox{BQP}}}(\Upsilon)^{-1}E_{ii}A_{{\tiny\mbox{BQP}}}(\Upsilon) F_{{\tiny\mbox{BQP}}}(\Upsilon)^{-1}\Diag(\check{x})\right)\\
    % & \quad +\diag\left(F_{{\tiny\mbox{BQP}}}(\Upsilon)^{-1}A_{{\tiny\mbox{BQP}}}(\Upsilon)E_{ii} F_{{\tiny\mbox{BQP}}}(\Upsilon)^{-1}\Diag(\check{x})\right)\\
     % & ~=\diag\left(F_{{\tiny\mbox{BQP}}}(\Upsilon)^{-1}E_{ii}\left(I + \Diag(\check{x}) F_{{\tiny\mbox{BQP}}}(\Upsilon)^{-1}\right)\Diag(\check{x})\right) \\
     % &  \quad +\diag\left(\left(I+ F_{{\tiny\mbox{BQP}}}(\Upsilon)^{-1}\Diag(\check{x})\right)E_{ii} F_{{\tiny\mbox{BQP}}}(\Upsilon)^{-1}\Diag(\check{x})\right)\\
     &~= \diag\left(F_{{\tiny\mbox{BQP}}}(\Upsilon)^{-1}E_{ii}\Diag(\check{x})\right)+ \diag\left( E_{ii}F_{{\tiny\mbox{BQP}}}(\Upsilon)^{-1}\Diag(\check{x})\right)\\
     &~~ \quad +\diag\left(F_{{\tiny\mbox{BQP}}}(\Upsilon)^{-1}E_{ii} \Diag(\check{x}) F_{{\tiny\mbox{BQP}}}(\Upsilon)^{-1}\Diag(\check{x})\right)\\
     &~~ \quad +\diag\left( F_{{\tiny\mbox{BQP}}}(\Upsilon)^{-1}\Diag(\check{x})E_{ii} F_{{\tiny\mbox{BQP}}}(\Upsilon)^{-1}\Diag(\check{x})\right)\\
     & ~=2\left(x_i-1\right)\left( \left(F_{{\tiny\mbox{BQP}}}(\Upsilon)^{-1}\right)_{ii} \mathbf{e}_i\right)\\
     &~~ \quad + 2\left(x_i-1\right)\left(\diag\left(F_{{\tiny\mbox{BQP}}}(\Upsilon)^{-1}E_{ii} F_{{\tiny\mbox{BQP}}}(\Upsilon)^{-1}\Diag(\check{x})\right)\right)\\
     % & ~=2 \Diag\left(\check{x}\right) \Diag\left(\diag\left(F_{{\tiny\mbox{BQP}}}(\Upsilon)^{-1}\right)\right) \mathbf{e}_i\\
     % &  \quad +2 \Diag\left(\check{x}\right) \diag\left(  \left(F_{{\tiny\mbox{BQP}}}(\Upsilon)^{-1}\circ F_{{\tiny\mbox{BQP}}}(\Upsilon)^{-1}\right)_{\cdot i}\right) (x_i-1)\\
     & ~=2 \Diag\left(\check{x}\right) \Diag\left(\diag\left(F_{{\tiny\mbox{BQP}}}(\Upsilon)^{-1}\right)\right) \mathbf{e}_i\\
     &~~  \quad +2 \Diag\left(\check{x}\right)   \left(F_{{\tiny\mbox{BQP}}}(\Upsilon)^{-1}\circ F_{{\tiny\mbox{BQP}}}(\Upsilon)^{-1}\right) \Diag(\check{x}) \mathbf{e}_i\,,
 \end{align*}
    which implies that
    \begin{align*}
    & - \Diag(\Upsilon)^{-1}\diag\left(\!F_{{\tiny\mbox{BQP}}}(\Upsilon)^{-1}\textstyle\frac{\left(E_{ii}A_{{\tiny\mbox{BQP}}}(\Upsilon)+A_{{\tiny\mbox{BQP}}}(\Upsilon) E_{ii}\right)}{\gamma_i}
  F_{{\tiny\mbox{BQP}}}(\Upsilon)^{-1}\Diag(\check{x})\!\right)\\
        % &\textstyle\frac{1}{\gamma_i}\Diag(\Upsilon)^{-1}\diag\left(F_{{\tiny\mbox{BQP}}}(\Upsilon)^{-1}\left(E_{ii}A_{{\tiny\mbox{BQP}}}(\Upsilon)+A_{{\tiny\mbox{BQP}}}(\Upsilon) E_{ii}\right) F_{{\tiny\mbox{BQP}}}(\Upsilon)^{-1}\Diag(\check{x})\right)\\
    & ~~~=-2 \Diag(\Upsilon)^{-1} \Diag\left(\check{x}\right) \Diag\left(\diag\left(F_{{\tiny\mbox{BQP}}}(\Upsilon)^{-1}\right)\right) \Diag(\Upsilon)^{-1}\mathbf{e}_i \\
    & ~~ \quad - 2 \Diag(\Upsilon)^{-1} \Diag\left(\check{x}\right)   \left(F_{{\tiny\mbox{BQP}}}(\Upsilon)^{-1}\circ F_{{\tiny\mbox{BQP}}}(\Upsilon)^{-1}\right) \Diag(\check{x})\Diag(\Upsilon)^{-1}\mathbf{e}_i\,.
    \end{align*}
   Then, we obtain
    \begin{align*}
        &\textstyle\frac{\partial^2 f_{{\tiny\mbox{BQP}}}\left(x,X;\Upsilon\right)}{\partial \Upsilon ^2 }\\
        &  ~=- 2 \Diag\left(\Upsilon\right)^{-1} \Diag\left(\diag\left(F_{{\tiny\mbox{BQP}}}(\Upsilon)^{-1}\Diag(\check{x})\right)-\check{x}\right) \Diag\left(\Upsilon\right)^{-1}   \\
        &~~~\quad -4 \Diag(\Upsilon)^{-1} \Diag\left(\check{x}\right) \Diag\left(\diag\left(F_{{\tiny\mbox{BQP}}}(\Upsilon)^{-1}\right)\right) \Diag(\Upsilon)^{-1} \\
        &~~~~~~\quad -4 \Diag(\Upsilon)^{-1} \Diag\left(\check{x}\right)   \left(F_{{\tiny\mbox{BQP}}}(\Upsilon)^{-1}\circ F_{{\tiny\mbox{BQP}}}(\Upsilon)^{-1}\right)\Diag(\check{x}) \Diag(\Upsilon)^{-1}.
    \end{align*}
   % With the formula of $\frac{\partial^2 f_{{\tiny\mbox{BQP}}}\left(x,X;\Upsilon\right)}{\partial \Upsilon ^2 }$, we can derive $\frac{\partial^2 f_{{\tiny\mbox{BQP}}}\left(x,X;\Upsilon\right)}{\partial \left(\log \Upsilon\right) ^2}$ by
   Finally, we have
    \begin{align*}
        &\textstyle\frac{\partial^2 f_{{\tiny\mbox{BQP}}}\left(x,X;\Upsilon\right)}{\partial \left(\log \Upsilon\right) ^2} =\Diag(\Upsilon) \textstyle\frac{\partial f_{{\tiny\mbox{BQP}}}\left(x,X;\Upsilon\right)}{\partial \Upsilon } + \Diag(\Upsilon)\frac{\partial^2 f_{{\tiny\mbox{BQP}}}\left(x,X;\Upsilon\right)}{\partial \Upsilon ^2 } \Diag(\Upsilon)\\
        & ~= -4 \Diag\left(\check{x}\right) \Diag\left(\diag\left(F_{{\tiny\mbox{BQP}}}(\Upsilon)^{-1}\right)\right) \\
        &~~~~\quad - 4  \Diag\left(\check{x}\right)   \left(F_{{\tiny\mbox{BQP}}}(\Upsilon)^{-1}\circ F_{{\tiny\mbox{BQP}}}(\Upsilon)^{-1}\right) \Diag(\check{x})\\
        &~ =4 \Diag\left(\mathbf{e}-x\right) \Diag\left(\diag\left(F_{{\tiny\mbox{BQP}}}(\Upsilon)^{-1}\right)\right) \\
        &~~~~\quad -4  \Diag\left(\mathbf{e}-x\right)   \left(F_{{\tiny\mbox{BQP}}}(\Upsilon)^{-1}\circ F_{{\tiny\mbox{BQP}}}(\Upsilon)^{-1}\right) \Diag(\mathbf{e}-x).
    \end{align*}
\smallskip
    Next, we will show the positive semidefiniteness of $\frac{\partial^2 f_{{\tiny\mbox{BQP}}}\left(x,X;\Upsilon\right)}{\partial \left(\log \Upsilon\right) ^2} $ for all $ 0\le x \le \mathbf{e}, X \succeq 0$ such that $(x,X) \in \dom\left(f_{{\tiny\mbox{BQP}}};\Upsilon\right)$. Note that we will not require $(x,X)$ to be feasible for \ref{BQP}. We analyse  two cases.

    \smallskip
    \underline{Case $1$}: when $0\le x<\mathbf{e}$ and $X\succeq 0$,  let $D_{{\tiny\mbox{BQP}}}(x):=\left(\Diag(\mathbf{e}-x)\right)^{1/2} \succ 0$, and let $H_{{\tiny\mbox{BQP}}}(x,X;\Upsilon):=$
    $\left(D_{{\tiny\mbox{BQP}}}(x)\right)^{-1} A_{{\tiny\mbox{BQP}}}(\Upsilon) \left(D_{{\tiny\mbox{BQP}}}(x)\right)^{-1}\succeq 0$. Again, for simplicity, we write $D_{{\tiny\mbox{BQP}}}(x)$ and $H_{{\tiny\mbox{BQP}}}(x,X;\Upsilon)$ as $D_{{\tiny\mbox{BQP}}}$ and $H_{{\tiny\mbox{BQP}}}(\Upsilon)$, respectively. First, we note that
    \begin{align*}
        D_{{\tiny\mbox{BQP}}}  F_{{\tiny\mbox{BQP}}}(\Upsilon)^{-1}D_{{\tiny\mbox{BQP}}} =\left( D_{{\tiny\mbox{BQP}}}^{-1} A_{{\tiny\mbox{BQP}}}(\Upsilon) D_{{\tiny\mbox{BQP}}}^{-1} + I\right)^{-1}.
    \end{align*}
    Then, we have
    \begin{align*}
        & \textstyle{\frac{1}{4}} \frac{\partial^2 f_{{\tiny\mbox{BQP}}}\left(x,X;\Upsilon\right)}{\partial \left(\log \Upsilon\right) ^2} \\
        &~= \Diag\left(\diag\left( D_{{\tiny\mbox{BQP}}}  F_{{\tiny\mbox{BQP}}}(\Upsilon)^{-1}D_{{\tiny\mbox{BQP}}}\right)\right) \\
        & ~~\quad -\Diag\left( D_{{\tiny\mbox{BQP}}}  F_{{\tiny\mbox{BQP}}}(\Upsilon)^{-1}D_{{\tiny\mbox{BQP}}}\right) \circ \Diag\left( D_{{\tiny\mbox{BQP}}}  F_{{\tiny\mbox{BQP}}}(\Upsilon)^{-1}D_{{\tiny\mbox{BQP}}}\right)\\
        & ~=\left( D_{{\tiny\mbox{BQP}}}  F_{{\tiny\mbox{BQP}}}(\Upsilon)^{-1}D_{{\tiny\mbox{BQP}}}\right)\circ I \\
        & ~~~\quad -\Diag\left( D_{{\tiny\mbox{BQP}}}  F_{{\tiny\mbox{BQP}}}(\Upsilon)^{-1}D_{{\tiny\mbox{BQP}}}\right) \circ \Diag\left( D_{{\tiny\mbox{BQP}}}  F_{{\tiny\mbox{BQP}}}(\Upsilon)^{-1}D_{{\tiny\mbox{BQP}}}\right)\\
        & ~=\left(H_{{\tiny\mbox{BQP}}}(\Upsilon)+I \right)^{-1}\circ I - \left(H_{{\tiny\mbox{BQP}}}(\Upsilon)+I \right)^{-1} \circ \left(H_{{\tiny\mbox{BQP}}}(\Upsilon)+I \right)^{-1}\\
        & ~=\left(H_{{\tiny\mbox{BQP}}}(\Upsilon)+I \right)^{-1} \circ \left(I - \left(H_{{\tiny\mbox{BQP}}}(\Upsilon)+I \right)^{-1} \right)
        \succeq~ 0.
    \end{align*}
    The last inequality holds because $H_{{\tiny\mbox{BQP}}}(\Upsilon)+I  \! \succ \! 0$ and
    the Schur Product Theorem (\cite[p. 15, Theorem VII]{Schur}).
    % because the Hadamard product of two positive semidefinite matrices is positive semidefinite.

    \smallskip
    \underline{Case $2$}: now, we discuss the general case $0\le x \le \mathbf{e}, X\succeq 0$. Note that for $\Upsilon \in \mathbb{R}^n_{++}$, $\frac{\partial f^2_{{\tiny\mbox{BQP}}}\left(x,X;\Upsilon\right)}{\partial \left(\log \Upsilon\right)^2}$ is analytic in $0\le x \le \mathbf{e}, X \succeq 0$ such that $(x,X) \in \dom\left(f_{{\tiny\mbox{BQP}}};\Upsilon\right)$.  Therefore, given $0\le x \le \mathbf{e}, X \succeq 0$, assume that $\frac{\partial f^2_{{\tiny\mbox{BQP}}}\left(x,X;\Upsilon\right)}{\partial \left(\log \Upsilon\right)^2}\not \succeq 0$. Then by the analyticity (continuity) of $\frac{\partial f^2_{{\tiny\mbox{BQP}}}\left(x,X;\Upsilon\right)}{\partial \left(\log \Upsilon\right)^2}$, there exists small enough $\epsilon>0$ such that {for any $0\le x' \le \mathbf{e}, X' \succeq 0$ in the intersection of the neighbourhood
    \[
    \mathcal{N}_\epsilon(x,X):= \left\{(x',X'): \|x-x'\|_{\infty}+ \|X-X'\|_F\le \epsilon\right\},
    \]
    (where $\|\cdot\|_{\infty}$ is the vector infinity-norm, and $\|\cdot\|_F$ is the Frobenius norm) and $\{(x',X'): 0\le x' \le \mathbf{e}, X' \succeq 0, (x',X') \in \dom\left(f_{{\tiny\mbox{BQP}}};\Upsilon\right)\}$, we have $\frac{\partial f^2_{{\tiny\mbox{BQP}}}\left(x',X';\Upsilon\right)}{\partial \left(\log \Upsilon\right)^2} \not \succeq 0$.} On the other hand, this intersection contains some $(x',X')$ such that $0\le x'<\mathbf{e}$, $X'\succeq 0$, e.g. $(x',X')=(x-\sum_{i:x_i=1} \epsilon e_i, X)$. This is a contradiction to Case $1$.   \smallskip

   In conclusion, for each fixed $(x,X)\in \{(x,X): 0\le x \le \mathbf{e}, X \succeq 0, (x,X) \in \dom\left(f_{{\tiny\mbox{BQP}}};\Upsilon\right)\}$, we have that $f_{{\tiny\mbox{BQP}}}\left(x,X;\Upsilon\right)$ is convex in $\log \Upsilon$. In particular, for $(x,X)\in \dom\left(f_{{\tiny\mbox{BQP}}};\Upsilon\right)$ and feasible to \ref{BQP}, $f_{{\tiny\mbox{BQP}}}\left(x,X;\Upsilon\right)$ is convex in $\log \Upsilon$. Finally, as $z_{{\tiny\mbox{BQP}}}(\Upsilon)$ is the point-wise maximum of $f_{{\tiny\mbox{BQP}}}\left(x,X;\Upsilon\right)$ over all such $(x,X)$, then $z_{{\tiny\mbox{BQP}}}(\Upsilon)$ is convex in $\log \Upsilon$. $\hfill \square$

\section{linx bound}\label{sec:linx}
The linx bound was first analyzed and developed in \cite{Kurt_linx} (see \cite[Section 3.3]{FLbook} for more details). For  $\Upsilon \in \mathbb{R}_{++}^n$ and  $x\in[0,1]^n$,
%and $\gamma>0$
 we now define
 % $f_{{\tiny\mbox{linx}}}(x;\Upsilon) :=$
\begin{align*}
f_{{\tiny\mbox{linx}}}(x;\Upsilon) :=& \textstyle\frac{1}{2}\left(\strut\ldet \left( \Diag(\Upsilon) C\Diag(x)C \Diag(\Upsilon)+\Diag(\mathbf{e}-x) \right)\right) \\ &\quad \textstyle -\sum_{i=1}^n x_i\log \gamma_i
\end{align*}
with %$\dom\left(f_{{\tiny\mbox{linx}}};\Upsilon\right):=$
\begin{align*}
 \dom\left(f_{{\tiny\mbox{linx}}};\Upsilon\right)\!:=\!  \left\{\strut x \in \mathbb{R}^n \!\,:\!\, \Diag(\Upsilon) C\Diag(x)C \Diag(\Upsilon)+\Diag(\mathbf{e}-x) \succ 0 \right\}.
\end{align*}
We then define the \emph{g-scaled linx  bound}  %is
\begin{equation}\tag{linx}\label{linx}
\begin{array}{ll}
	z_{{\tiny\mbox{linx}}}(\Upsilon):=\max \left\{\strut f_{{\tiny\mbox{linx}}}(x;\Upsilon)
	~:~ \mathbf{e}^{\top}x=s,~ 0\leq x\leq \mathbf{e},~ Ax\leq b\strut \right\}.
\end{array}
\end{equation}
We say that $x$ is feasible to \ref{linx} if $x$ satisfies all the constraints in \ref{linx}.

% Note that this is applying the BQP bound to the right-hand-side of the identity \eqref{gscale:eqn} where $t_1=1, t_2=0$, we \emph{cannot} interpret this bound as

% Note we can interpret \ref{linx} as applying the
% un-scaled linx bound to the row-scaled
% matrix $\Diag(\Upsilon)C$, because we would lose symmetry we can interpret \ref{linx} as applying the
% un-scaled linx bound to the row-scaled
% matrix $\Diag(\Upsilon)C$,  because although we lose symmetry in the matrix, we do not lose symmetry in the upper bound formula $f_{{\tiny\mbox{linx}}}(x;\Upsilon)$.

\medskip

It is very important to note, in contrast to g-scaling for the \ref{BQP} bound, that we are \emph{not} applying the ordinary \ref{linx} bound to a symmetric scaling of $C$.
In this way, g-scaling for the \ref{linx} bound is more subtle.
Rather, we are symmetrically scaling $\Diag(\Upsilon) C\Diag(x)C \Diag(\Upsilon)$.
This point would not apply to o-scaling, as scalars commute through matrix multiplication.

\vbox{
\begin{thm}\label{thm:linx}
For all $\Upsilon \in \mathbb{R}_{++}^n$ in \ref{linx}, the following hold:
\begin{itemize}
\item[\ref{thm:linx}.i.] $z_{{\tiny\mbox{linx}}}(\Upsilon)$ is a valid upper bound for the optimal value of \ref{CMESP}, i.e.,\newline $z(C,s,A,b)\leq z_{{\tiny\mbox{linx}}}(\Upsilon)$;
\item[\ref{thm:linx}.ii.] the function $f_{{\tiny\mbox{linx}}}(x;\Upsilon)$ is concave in $x$ on $\dom\left(f_{{\tiny\mbox{linx}}};\Upsilon\right)$ and continuously differentiable in $(x,\Upsilon)$ on $\dom\left(f_{{\tiny\mbox{linx}}};\Upsilon\right)\times \mathbb{R}^n_{++}$\,;
\item[\ref{thm:linx}.iii.] \label{linx.iii} for fixed $x\in \dom\left(f_{{\tiny\mbox{linx}}};\Upsilon\right)$, $f_{{\tiny\mbox{linx}}}(x;\Upsilon)$ is convex in $\log \Upsilon$, and thus $z_{{\tiny\mbox{linx}}}(\Upsilon)$ is convex in $\log \Upsilon$.
\end{itemize}
\end{thm}
}

\end{itemize}
\end{proof}

\begin{remark}
    \cite{Kurt_linx} established Theorem \ref{thm:linx}.\emph{i} for $\Upsilon:=\gamma \mathbf{e}$, with $\gamma\in\mathbb{R}_{++}$\,. We generalize this result to $\Upsilon\in \mathbb{R}^n_{++}$.
    The concavity in Theorem \ref{thm:linx}.\emph{ii}
is a result of \cite{Kurt_linx}, with details
filled in by \cite{FLbook}. Theorem \ref{thm:linx}.\emph{iii}
generalizes a result of \cite{chen_mixing}, where it is
established only for o-scaling: i.e., on $\{ \Upsilon= \gamma \mathbf{e} ~:~ \gamma\in \mathbb{R}_{++}\}$. The proof of Theorem \ref{thm:linx}.\emph{iii} requires new techniques (see the below). As in the case of the  \ref{BQP} bound, the result is quite important as it enables the use of readily available quasi-Newton methods (like BFGS) for finding the globally optimal g-scaling for the \ref{linx} bound.
\end{remark}

A small example of how g-scaling can improve upon o-scaling for the \ref{linx} bound can be found in the Appendix.

\begin{proof}[Theorem \ref{thm:linx}]
\begin{itemize}
    \item[\ref{thm:linx}.i:] It is enough to prove that there is a feasible solution to \ref{linx} with  objective value equal to the optimal value of \ref{CMESP}. In fact, let $x^* \in \{0,1\}^n$ be one optimal solution to \ref{CMESP} with support $S\left(x^*\right)$, and define $X^* := x^*\left(x^*\right)^\top$. Without loss of generality, we assume that $S\left(x^*\right)=\{1,  \ldots, s\}$, i.e., $x^* = \left(\!\begin{array}{c}
         \mathbf{e}_{s}  \\
         0
    \end{array}\!\right)$.
    Let $T(x^*): = N \backslash S(x^*)$ be the complementary set of $S(x^*)$. For convenience, we denote $\tilde{C}:=  \Diag(\Upsilon) C$, $\tilde{C}_{ST}:=\tilde{C}(S(x^*), T(x^*))$, and so on. Note that $\tilde{C}$ is not symmetric. Also, note that  $\tilde{C}$ depends on $\Upsilon$, and  $\tilde{C}_{ST}$ depends on $\Upsilon, x^*, S(x^*)$, and $T(x^*)$. We can write
\begin{align*}
\tilde{C} \operatorname{Diag}(x) \tilde{C}^\top
\!=\!&\left(\begin{array}{ll}
\tilde{C}_{S S
} & \tilde{C}_{S T} \\
\tilde{C}_{T S} & \tilde{C}_{T T}
\end{array}\right)\!
\left(\begin{array}{ll}
I_s & 0 \\
0 & 0
\end{array}\right)\!
\left(\begin{array}{ll}
\tilde{C}_{S S}^\top & \tilde{C}_{T S} ^\top \\
\tilde{C}_{S T}^\top & \tilde{C}_{T T}^\top
\end{array}\right)
\!=\!\left(\begin{array}{cc}
\tilde{C}_{S S} \tilde{C}_{S S}^\top & \tilde{C}_{S S} \tilde{C}_{T S}^\top \\
\tilde{C}_{T S}^{T} \tilde{C}_{S S}^\top & \tilde{C}_{T S} \tilde{C}_{T S}^\top
\end{array}\right),
\end{align*}
and therefore
$$
\tilde{C}\operatorname{Diag}(x) \tilde{C}^\top+\operatorname{Diag}(\mathbf{e}-x)=\left(\begin{array}{cc}
\tilde{C}_{S S} \tilde{C}_{S S}^\top & \tilde{C}_{S S} \tilde{C}_{T S}^\top \\
\tilde{C}_{T S}^{T} \tilde{C}_{S S}^\top & \tilde{C}_{T S} \tilde{C}_{T S}^\top+I_{n-s}
\end{array}\right).
$$
Applying the well-known Schur-complement determinant formula, we then obtain
$$
\begin{aligned}
&\operatorname{ldet}\left(\tilde{C} \operatorname{Diag}(x) \tilde{C}^\top+\operatorname{Diag}(\mathbf{e}-x)\right) \\
&~=2 \operatorname{ldet} \tilde{C}_{S S}+\operatorname{ldet}\left(\tilde{C}_{T S}^\top \tilde{C}_{T S}+I_{n-s}-\tilde{C}_{T S}^\top \tilde{C}_{S S}^\top \tilde{C}_{S S}^{-\top} \tilde{C}_{SS}^{-1} \tilde{C}_{S S} \tilde{C}_{T S}^\top\right) \\
&~=2 \operatorname{ldet} \tilde{C}_{S S}\,.
\end{aligned}
$$
Let $\Upsilon_{S(x^*)}$ be the sub-vector of $\Upsilon$ indexed by $S(x^*)$. Then, we have
\begin{align*}
    &f_{{\tiny\mbox{linx}}}(x^*;\Upsilon)=\textstyle\frac{1}{2} \operatorname{ldet}\left(\tilde{C} \operatorname{Diag}(x) \tilde{C}^\top+\operatorname{Diag}(\mathbf{e}-x)\right)-\sum_{i\in N} x^*_i\log \gamma_i\\
    &~=\operatorname{ldet} \tilde{C}_{S S}-\textstyle\sum_{i\in S(x^*)} \log \gamma_i\\
    &~= \operatorname{ldet} \left(\strut\Diag\left(\Upsilon_{S(x^*)}\right)C\left(S(x^*), S(x^*)\right)\right)
    -\textstyle\sum_{i\in S(x^*)} \log \gamma_i\\
    &~=\operatorname{ldet} C\left(S(x^*), S(x^*)\right).
\end{align*}

    \item[\ref{thm:linx}.ii:] The concavity is essentially a  result of \cite{Kurt_linx}, with details filled in by \cite[Section 3.3.1]{FLbook}. The continuous differentiability comes from the analyticity of $f_{{\tiny\mbox{linx}}}(x;\Upsilon)$ in $(x,\Upsilon) \in \dom\left(f_{{\tiny\mbox{linx}}};\Upsilon\right)\times \mathbb{R}^n_{++}$\,.
    \item[\ref{thm:linx}.iii:] We sketch the proof first:
    \begin{enumerate}
        \item for fixed $x\in \dom\left(f_{{\tiny\mbox{linx}}};\Upsilon\right)$, we derive the Hessian of $f_{{\tiny\mbox{linx}}}\left(x;\Upsilon\right)$ with respect to $\log \Upsilon$ and show that it is positive-semidefinite, which implies the convexity of $f_{{\tiny\mbox{linx}}}\left(x;\Upsilon\right)$ in $\log \Upsilon$;
        \item
     The convexity of $z_{{\tiny\mbox{linx}}}(\Upsilon)$  in $\log \Upsilon$ then follows because $z_{{\tiny\mbox{linx}}}(\Upsilon)$ is the point-wise maximum of $f_{{\tiny\mbox{linx}}}\left(x;\Upsilon\right)$ over feasible $x$ for \ref{linx} in domain $\dom\left(f_{{\tiny\mbox{linx}}};\Upsilon\right)$.
    \end{enumerate}

    \smallskip
    The detailed proof is as follows: for convenience, let
    \begin{align*}
    &F_{{\tiny\mbox{linx}}}(x;\Upsilon):=  \Diag(\Upsilon) C\Diag(x)C \Diag(\Upsilon)+ \Diag(\mathbf{e}-x), \mbox{ and }\\
    &A_{{\tiny\mbox{linx}}}(x;\Upsilon):=\Diag(\Upsilon) C\Diag(x)C \Diag(\Upsilon).
    \end{align*}

    In the following derivation, we will fix $x$ and regard $\Upsilon$ as a variable. Thus, for simplicity, we will write $F_{{\tiny\mbox{linx}}}(x;\Upsilon)$ and  $A_{{\tiny\mbox{linx}}}(x;\Upsilon)$ as $F_{{\tiny\mbox{linx}}}(\Upsilon)$ and $A_{{\tiny\mbox{linx}}}(\Upsilon)$, respectively.
     Let  $\check{x}:=x-\mathbf{e}$, and we note that
    % $A_{{\tiny\mbox{BQP}}}(\Upsilon)F_{{\tiny\mbox{BQP}}}(\Upsilon)^{-1} = I +\Diag(\check{x}) F_{{\tiny\mbox{BQP}}}(\Upsilon)^{-1}$ and
     \begin{align}
     &F_{{\tiny\mbox{linx}}}(\Upsilon)^{-1}A_{{\tiny\mbox{linx}}}(\Upsilon) = I + F_{{\tiny\mbox{linx}}}(\Upsilon)^{-1}\Diag(\check{x}),\label{FAlinx}\\
      &A_{{\tiny\mbox{linx}}}(\Upsilon)F_{{\tiny\mbox{linx}}}(\Upsilon)^{-1} = I + \Diag(\check{x})F_{{\tiny\mbox{linx}}}(\Upsilon)^{-1}.\label{AFlinx}
    \end{align}

    \medskip

    Given $x\in \dom\left(f_{{\tiny\mbox{linx}}};\Upsilon\right)$, we first derive the gradient of $f_{{\tiny\mbox{linx}}}\left(x;\Upsilon\right)$ with respect to $\Upsilon$. We have
    \begin{align*}
         &\textstyle\frac{\partial f_{{\tiny\mbox{linx}}}\left(x;\Upsilon\right)}{\partial \gamma_i}=\frac{1}{2} F_{{\tiny\mbox{linx}}}(\Upsilon)^{-1}\bullet \frac{\partial A_{{\tiny\mbox{linx}}}(\Upsilon)}{\partial \gamma_i} -\frac{x_i}{\gamma_i}\\
         & ~=\textstyle\frac{1}{2}  F_{{\tiny\mbox{linx}}}(\Upsilon)^{-1}\bullet \frac{1}{\gamma_i}\left(E_{ii}A_{{\tiny\mbox{linx}}}(\Upsilon)+A_{{\tiny\mbox{linx}}}(\Upsilon) E_{ii}\right)-\frac{x_i}{\gamma_i}\\
        % & ~ =\textstyle\frac{1}{2}  F_{{\tiny\mbox{linx}}}(\Upsilon)^{-1}\bullet \frac{1}{\gamma_i}\left(\mathbf{e}_i \mathbf{e}_i^\top A_{{\tiny\mbox{linx}}}(\Upsilon)+A_{{\tiny\mbox{linx}}}(\Upsilon) \mathbf{e}_i \mathbf{e}_i^\top\right)-\frac{x_i}{\gamma_i}\\
    &~ =\textstyle\frac{1}{2\gamma_i}\left(\left(A_{{\tiny\mbox{linx}}}(\Upsilon)_{i\cdot}F_{{\tiny\mbox{linx}}}(\Upsilon)^{-1}_{\cdot i}+F_{{\tiny\mbox{linx}}}(\Upsilon)^{-1}_{i\cdot}A_{{\tiny\mbox{linx}}}(\Upsilon)_{\cdot i}\right)-2x_i\right)\\
    & ~=\textstyle\frac{1}{\gamma_i}\left(F_{{\tiny\mbox{linx}}}(\Upsilon)^{-1}_{i\cdot}A_{{\tiny\mbox{linx}}}(\Upsilon)_{\cdot i}-x_i\right),
    \end{align*}
    where the last identity follows from the symmetry of  $F_{{\tiny\mbox{linx}}}(\Upsilon)$ and $A_{{\tiny\mbox{linx}}}(\Upsilon)$. Then, applying \eqref{FAlinx}, we obtain
    \begin{align*}
        &\textstyle\frac{\partial f_{{\tiny\mbox{linx}}}\left(x;\Upsilon\right)}{\partial \Upsilon}
        % = \Diag(\Upsilon)^{-1}\left(\diag\left(F_{{\tiny\mbox{linx}}}(\Upsilon)^{-1}A_{{\tiny\mbox{linx}}}(\Upsilon)\right)-x\right)\\
        %& ~=\Diag(\Upsilon)^{-1}\left(\diag\left(F_{{\tiny\mbox{linx}}}(\Upsilon)^{-1}\left(F_{{\tiny\mbox{linx}}}(\Upsilon)+\Diag(\check{x}) \right)\right)-x\right)\\
    %&~
=\Diag(\Upsilon)^{-1}\left(\diag\left(F_{{\tiny\mbox{linx}}}(\Upsilon)^{-1}\Diag(\check{x})\right)-\check{x}\right).
    \end{align*}
   Next, we derive the Hessian of $f_{{\tiny\mbox{linx}}}\left(x;\Upsilon\right)$ with respect to $\Upsilon$. We have
    \begin{align*}
    &\textstyle\frac{\partial^2 f_{{\tiny\mbox{linx}}}\left(x;\Upsilon\right)}{\partial \Upsilon \partial \gamma_i}=\textstyle\frac{\partial}{\partial\gamma_i}\left(\frac{\partial f_{{\tiny\mbox{linx}}}\left(x;\Upsilon\right)}{\partial \Upsilon}\right)\\
    & ~= \textstyle\frac{\partial \Diag(\Upsilon)^{-1}}{\partial\gamma_i} \left(\diag\left(F_{{\tiny\mbox{linx}}}(\Upsilon)^{-1}\Diag(\check{x})\right)-\check{x}\right)\\
    &~~\quad +\Diag(\Upsilon)^{-1}  \textstyle\frac{\partial \left(\diag\left(F_{{\tiny\mbox{linx}}}(\Upsilon)^{-1}\Diag(\check{x})\right)-\check{x} \right)}{\partial\gamma_i} \\
    & ~ =\textstyle\frac{\partial \Diag(\Upsilon)^{-1}}{\partial\gamma_i}\left(\diag\left(F_{{\tiny\mbox{linx}}}(\Upsilon)^{-1}\Diag(\check{x})\right)-\check{x}\right)\\
    &~~\quad -\Diag(\Upsilon)^{-1}  \diag\left(F_{{\tiny\mbox{linx}}}(\Upsilon)^{-1} \textstyle\frac{\partial F_{{\tiny\mbox{linx}}}(\Upsilon)}{\partial\gamma_i}F_{{\tiny\mbox{linx}}}(\Upsilon)^{-1}\Diag(\check{x})\right) \\
    &~= -\textstyle\frac{1}{\gamma_i^2}E_{ii}\left(\diag\left(F_{{\tiny\mbox{linx}}}(\Upsilon)^{-1}\Diag(\check{x})\right)-\check{x}\right)\\
    &~~\quad -\Diag(\Upsilon)^{-1}\diag\left(F_{{\tiny\mbox{linx}}}(\Upsilon)^{-1}\textstyle\frac{\left(E_{ii}A_{{\tiny\mbox{linx}}}(\Upsilon)+A_{{\tiny\mbox{linx}}}(\Upsilon) E_{ii}\right)}{\gamma_i}
    % \right.\\
    % &\qquad\left.\vphantom{\frac{\left(E_{ii}A+A E_{ii}\right)}{\gamma_i}}
    F_{{\tiny\mbox{linx}}}(\Upsilon)^{-1}\Diag(\check{x})\right).
    % & ~=-\textstyle\frac{1}{\gamma_i^2}E_{ii}\left(\diag\left(F_{{\tiny\mbox{linx}}}(\Upsilon)^{-1}\Diag(\check{x})\right)-\check{x}\right)\\
    % &\quad -\textstyle\frac{1}{\gamma_i}\Diag(\Upsilon)^{-1}\diag\left(F_{{\tiny\mbox{linx}}}(\Upsilon)^{-1}\left(E_{ii}A_{{\tiny\mbox{linx}}}(\Upsilon)+A_{{\tiny\mbox{linx}}}(\Upsilon) E_{ii}\right)
    % % \right.\\
    % % &\qquad\left.\vphantom{\left(E_{ii}A+A E_{ii}\right)}
    % F_{{\tiny\mbox{linx}}}(\Upsilon)^{-1}\Diag(\check{x})\right).
\end{align*}
 For the first term, we can reformulate
 \begin{align*}
    & -\textstyle\frac{1}{\gamma_i^2}E_{ii}\left(\diag\left(F_{{\tiny\mbox{linx}}}(\Upsilon)^{-1}\Diag(\check{x})\right)-\check{x}\right)\\
     & ~=-  \Diag\left(\Upsilon\right)^{-1} \Diag\left(\diag\left(F_{{\tiny\mbox{linx}}}(\Upsilon)^{-1}\Diag(\check{x})\right)-\check{x}\right) \Diag\left(\Upsilon\right)^{-1} \mathbf{e}_i\,,
 \end{align*}
 while for the second term, we can reformulate
 \begin{align*}
     &\diag\left(F_{{\tiny\mbox{linx}}}(\Upsilon)^{-1}\left(E_{ii}A_{{\tiny\mbox{linx}}}(\Upsilon)+A_{{\tiny\mbox{linx}}}(\Upsilon) E_{ii}\right) F_{{\tiny\mbox{linx}}}(\Upsilon)^{-1}\Diag(\check{x})\right)\\
     % &~=\diag\left(F_{{\tiny\mbox{linx}}}(\Upsilon)^{-1}E_{ii}A_{{\tiny\mbox{linx}}}(\Upsilon) F_{{\tiny\mbox{linx}}}(\Upsilon)^{-1}\Diag(\check{x})\right)\\
     % & \quad +\diag\left(F_{{\tiny\mbox{linx}}}(\Upsilon)^{-1}A_{{\tiny\mbox{linx}}}(\Upsilon)E_{ii} F_{{\tiny\mbox{linx}}}(\Upsilon)^{-1}\Diag(\check{x})\right)\\
     &~= \diag\left(F_{{\tiny\mbox{linx}}}(\Upsilon)^{-1}E_{ii}\left(I + \Diag(\check{x}) F_{{\tiny\mbox{linx}}}(\Upsilon)^{-1}\right)\Diag(\check{x})\right) \\
     &~~ \quad + \diag\left(\left(I+ F_{{\tiny\mbox{linx}}}(\Upsilon)^{-1}\Diag(\check{x})\right)E_{ii} F_{{\tiny\mbox{linx}}}(\Upsilon)^{-1}\Diag(\check{x})\right)\\
     &~= \diag\left(F_{{\tiny\mbox{linx}}}(\Upsilon)^{-1}E_{ii}\Diag(\check{x})\right)+ \diag\left( E_{ii}F_{{\tiny\mbox{linx}}}(\Upsilon)^{-1}\Diag(\check{x})\right)\\
     & ~~\quad + \diag\left(F_{{\tiny\mbox{linx}}}(\Upsilon)^{-1}E_{ii} \Diag(\check{x}) F_{{\tiny\mbox{linx}}}(\Upsilon)^{-1}\Diag(\check{x})\right)\\
     &~~\quad + \diag\left( F_{{\tiny\mbox{linx}}}(\Upsilon)^{-1}\Diag(\check{x})E_{ii} F_{{\tiny\mbox{linx}}}(\Upsilon)^{-1}\Diag(\check{x})\right)\\
     &~= 2\left(x_i-1\right)\left( \left(F_{{\tiny\mbox{linx}}}(\Upsilon)^{-1}\right)_{ii} \mathbf{e}_i\right)\\
     &~~\quad + 2\left(x_i-1\right)\left(\diag\left(F_{{\tiny\mbox{linx}}}(\Upsilon)^{-1}E_{ii} F_{{\tiny\mbox{linx}}}(\Upsilon)^{-1}\Diag(\check{x})\right)\right)\\
     % &~= 2 \Diag\left(\check{x}\right) \Diag\left(\diag\left(F_{{\tiny\mbox{linx}}}(\Upsilon)^{-1}\right)\right) \mathbf{e}_i\\
     % & \quad + 2 \Diag\left(\check{x}\right) \diag\left(  \left(F_{{\tiny\mbox{linx}}}(\Upsilon)^{-1}\circ F_{{\tiny\mbox{linx}}}(\Upsilon)^{-1}\right)_{\cdot i}\right) (x_i-1)\\
     &~= 2 \Diag\left(\check{x}\right) \Diag\left(\diag\left(F_{{\tiny\mbox{linx}}}(\Upsilon)^{-1}\right)\right) \mathbf{e}_i\\
     &~~ \quad + 2 \Diag\left(\check{x}\right)   \left(F_{{\tiny\mbox{linx}}}(\Upsilon)^{-1}\circ F_{{\tiny\mbox{linx}}}(\Upsilon)^{-1}\right) \Diag(\check{x}) \mathbf{e}_i\,,
 \end{align*}
    which implies that
    \begin{align*}
        &\textstyle\frac{1}{\gamma_i}\Diag(\Upsilon)^{-1}\diag\left(F_{{\tiny\mbox{linx}}}(\Upsilon)^{-1}\left(E_{ii}A_{{\tiny\mbox{linx}}}(\Upsilon)+A_{{\tiny\mbox{linx}}}(\Upsilon) E_{ii}\right)
     %   \right.\\
    %&\qquad\left.\vphantom{\left(E_{ii}A+A E_{ii}\right)}
    F_{{\tiny\mbox{linx}}}(\Upsilon)^{-1}\Diag(\check{x})\right)\\
    & ~=2 \Diag(\Upsilon)^{-1} \Diag\left(\check{x}\right) \Diag\left(\diag\left(F_{{\tiny\mbox{linx}}}(\Upsilon)^{-1}\right)\right) \Diag(\Upsilon)^{-1}\mathbf{e}_i \\
    & ~~\quad  +2 \Diag(\Upsilon)^{-1} \Diag\left(\check{x}\right)   \left(F_{{\tiny\mbox{linx}}}(\Upsilon)^{-1}\circ F_{{\tiny\mbox{linx}}}(\Upsilon)^{-1}\right) \Diag(\check{x})\Diag(\Upsilon)^{-1}\mathbf{e}_i\,.
    \end{align*}
    Finally, we obtain
    \begin{align*}
        &\textstyle\frac{\partial^2 f_{{\tiny\mbox{linx}}}\left(x;\Upsilon\right)}{\partial \Upsilon ^2 }
         =-  \Diag\left(\Upsilon\right)^{-1} \Diag\left(\diag\left(F_{{\tiny\mbox{linx}}}(\Upsilon)^{-1}\Diag(\check{x})\right)-\check{x}\right) \Diag\left(\Upsilon\right)^{-1}   \\
        & \quad - 2 \Diag(\Upsilon)^{-1} \Diag\left(\check{x}\right) \Diag\left(\diag\left(F_{{\tiny\mbox{linx}}}(\Upsilon)^{-1}\right)\right) \Diag(\Upsilon)^{-1} \\
        &~~~~\quad -2 \Diag(\Upsilon)^{-1} \Diag\left(\check{x}\right)   \left(F_{{\tiny\mbox{linx}}}(\Upsilon)^{-1}\circ F_{{\tiny\mbox{linx}}}(\Upsilon)^{-1}\right)  \Diag(\check{x}) \Diag(\Upsilon)^{-1}.
    \end{align*}
   Then, we have
    \begin{align*}
        &\textstyle\frac{\partial^2 f_{{\tiny\mbox{linx}}}\left(x;\Upsilon\right)}{\partial \left(\log \Upsilon\right) ^2} = \Diag(\Upsilon) \textstyle\frac{\partial f_{{\tiny\mbox{linx}}}\left(x;\Upsilon\right)}{\partial \Upsilon } + \Diag(\Upsilon)\frac{\partial^2 f_{{\tiny\mbox{linx}}}\left(x;\Upsilon\right)}{\partial \Upsilon ^2 } \Diag(\Upsilon)\\
        & ~= -2 \Diag\left(\check{x}\right) \Diag\left(\diag\left(F_{{\tiny\mbox{linx}}}(\Upsilon)^{-1}\right)\right) \\
        &~~~~\quad ~-2  \Diag\left(\check{x}\right)   \left(F_{{\tiny\mbox{linx}}}(\Upsilon)^{-1}\circ F_{{\tiny\mbox{linx}}}(\Upsilon)^{-1}\right) \Diag(\check{x})\\
        &~ =2 \Diag\left(\mathbf{e}-x\right) \Diag\left(\diag\left(F_{{\tiny\mbox{linx}}}(\Upsilon)^{-1}\right)\right) \\
        &~~~~\quad ~ - 2  \Diag\left(\mathbf{e}-x\right)   \left(F_{{\tiny\mbox{linx}}}(\Upsilon)^{-1}\circ F_{{\tiny\mbox{linx}}}(\Upsilon)^{-1}\right) \Diag(\mathbf{e}-x).
    \end{align*}
\smallskip
    Next, we are going to show the positive semidefiniteness of $\frac{\partial^2 f_{{\tiny\mbox{linx}}}\left(x;\Upsilon\right)}{\partial \left(\log \Upsilon\right) ^2} $ for all $0\le x \le \mathbf{e}$ such that $x\in \dom\left(f_{{\tiny\mbox{linx}}};\Upsilon\right)$. Note that we will not require $x$ to be feasible to \ref{linx}. We divide the discussion into two cases.

    \smallskip
    \underline{Case $1$}: when $0\le x<\mathbf{e}$,  let $D_{{\tiny\mbox{linx}}}(x):=\left(\Diag(\mathbf{e}-x)\right)^{1/2} \succ 0$, and  $H_{{\tiny\mbox{linx}}}(x;\Upsilon)$ $:=\left(D_{{\tiny\mbox{linx}}}(x)\right)^{-1} A_{{\tiny\mbox{linx}}}(\Upsilon) \left(D_{{\tiny\mbox{linx}}}(x)\right)^{-1}\succeq 0$. Again for simplicity, we write $D_{{\tiny\mbox{linx}}}(x)$ and $H_{{\tiny\mbox{linx}}}(x;\Upsilon)$ as $D_{{\tiny\mbox{linx}}}$ and $H_{{\tiny\mbox{linx}}}(\Upsilon)$. First, we note
    \begin{align*}
        D_{{\tiny\mbox{linx}}}  F_{{\tiny\mbox{linx}}}(\Upsilon)^{-1}D_{{\tiny\mbox{linx}}} =\left( D_{{\tiny\mbox{linx}}}^{-1} A_{{\tiny\mbox{linx}}}(\Upsilon) D_{{\tiny\mbox{linx}}}^{-1} + I\right)^{-1}.
    \end{align*}
     Then, we have
    \begin{align*}
        & \textstyle\frac{1}{2} \frac{\partial^2 f_{{\tiny\mbox{linx}}}\left(x;\Upsilon\right)}{\partial \left(\log \Upsilon\right) ^2} \\
        &~= \Diag\left(\diag\left( D_{{\tiny\mbox{linx}}}  F_{{\tiny\mbox{linx}}}(\Upsilon)^{-1}D_{{\tiny\mbox{linx}}}\right)\right) \\
        &~~ \quad -\Diag\left( D_{{\tiny\mbox{linx}}}  F_{{\tiny\mbox{linx}}}(\Upsilon)^{-1}D_{{\tiny\mbox{linx}}}\right) \circ \Diag\left( D_{{\tiny\mbox{linx}}}  F_{{\tiny\mbox{linx}}}(\Upsilon)^{-1}D_{{\tiny\mbox{linx}}}\right)\\
        & ~=\left( D_{{\tiny\mbox{linx}}}  F_{{\tiny\mbox{linx}}}(\Upsilon)^{-1}D_{{\tiny\mbox{linx}}}\right)\circ I \\
        &~~ \quad -\Diag\left( D_{{\tiny\mbox{linx}}}  F_{{\tiny\mbox{linx}}}(\Upsilon)^{-1}D_{{\tiny\mbox{linx}}}\right) \circ \Diag\left( D_{{\tiny\mbox{linx}}}  F_{{\tiny\mbox{linx}}}(\Upsilon)^{-1}D_{{\tiny\mbox{linx}}}\right)\\
        & ~=\left(H_{{\tiny\mbox{linx}}}(\Upsilon)+I \right)^{-1}\circ I - \left(H_{{\tiny\mbox{linx}}}(\Upsilon)+I \right)^{-1} \circ \left(H_{{\tiny\mbox{linx}}}(\Upsilon)+I \right)^{-1}\\
        & ~=\left(H_{{\tiny\mbox{linx}}}(\Upsilon)+I \right)^{-1} \circ \left(I - \left(H_{{\tiny\mbox{linx}}}(\Upsilon)+I \right)^{-1} \right)
        \succeq  ~ 0.
    \end{align*}
    The last inequality holds because $H_{{\tiny\mbox{BQP}}}(\Upsilon)+I \! \succ\! 0$ and the Schur Product Theorem.

    \smallskip
    \underline{Case $2$}: We now discuss general $0\le x \le \mathbf{e}$. Note that given $\Upsilon \in \mathbb{R}^n_{++}$\,, $\frac{\partial f^2_{{\tiny\mbox{linx}}}\left(x;\Upsilon\right)}{\partial \left(\log \Upsilon\right)^2}$ is analytical in $0\le x \le \mathbf{e}$ such that $x \in \dom\left(f_{{\tiny\mbox{linx}}};\Upsilon\right)$. Therefore, given $0\le x \le \mathbf{e}$, assume that $\frac{\partial f^2_{{\tiny\mbox{linx}}}\left(x;\Upsilon\right)}{\partial \left(\log \Upsilon\right)^2}\not\succeq 0$. Then by the analyticity (continuity) of $\frac{\partial f^2_{{\tiny\mbox{linx}}}\left(x;\Upsilon\right)}{\partial \left(\log \Upsilon\right)^2}$, there exists small enough $\epsilon>0$ such that for any $0\le x' \le \mathbf{e}$ in the intersection of neighbourhood $\mathcal{N}_\epsilon(x):= \left\{x': \|x-x'\|_{\infty}\le \epsilon\right\}$ (where $\|\cdot \|_{\infty}$ is the vector infinity norm) and $\{x': 0\le x' \le \mathbf{e}, x'\in \dom\left(f_{{\tiny\mbox{linx}}};\Upsilon\right)\}$, we have $\frac{\partial f^2_{{\tiny\mbox{linx}}}\left(x;\Upsilon\right)}{\partial \left(\log \Upsilon\right)^2}\not \succeq 0$. On the other hand, this intersection contains some $x'$ such that $0\le x'<\mathbf{e}$, e.g. $x'=x-\sum_{i:x_i=1} \epsilon e_i$. This contradicts Case $1$.
    \smallskip

   In conclusion, for each fixed $x\in \{(x,X): 0\le x \le \mathbf{e}, x \in \dom\left(f_{{\tiny\mbox{linx}}};\Upsilon\right)\}$, $f_{{\tiny\mbox{linx}}}\left(x;\Upsilon\right)$ is convex in $\log \Upsilon$. In particular, for $x\in \dom\left(f_{{\tiny\mbox{linx}}};\Upsilon\right)$ and feasible to \ref{linx}, $f_{{\tiny\mbox{linx}}}\left(x;\Upsilon\right)$ is convex in $\log \Upsilon$.
    Finally, as $z_{{\tiny\mbox{linx}}}(\Upsilon)$ is the point-wise maximum of $f_{{\tiny\mbox{linx}}}\left(x;\Upsilon\right)$ over all such $x$, we have that  $z_{{\tiny\mbox{linx}}}(\Upsilon)$ is convex in $\log \Upsilon$. $\hfill \square$

\end{itemize}
\end{proof}

\section{Factorization bound}\label{sec:fact}

The factorization bound was first analyzed in \cite{nikolov2015randomized},
and then developed further in \cite{li2020best} and in \cite{chen2023computing} (see \cite[Section 3.4]{FLbook} for more details). The definition of the factorization bound is based on the following key lemma.

\begin{lem}\label{Nik}(see \cite[Lemma 14]{nikolov2015randomized})
 Let $\lambda\in\mathbb{R}^k_+$ with $\lambda_1\geq \lambda_2\geq \cdots\geq \lambda_k$\,, and let $0<s\leq k$. There exists a unique integer $\iota$, with $0\leq \iota< s$, such that
 $
 \lambda_{\iota}>\frac{1}{s-\iota}\sum_{\ell=\iota+1}^k \lambda_{\ell}\geq \lambda_{\iota +1}
 $\,,
 with the convention $\lambda_0:=+\infty$.
\end{lem}
% \begin{remark}
%     Instead of restricting $\lambda \in \mathbb{R}_+^n$ as in \cite[Lemma 14]{nikolov2015randomized}, we allow $\lambda$ to be any vector in $\mathbb{R}^n$. The proof of \cite[Lemma 14]{nikolov2015randomized} does not require $\lambda \in \mathbb{R}_+^n$, so Lemma \ref{Nik} is still valid. As we will see later, this relaxation enables us to establish an open domain for the objective function of the factorization bound.
% \end{remark}

Now, suppose that  $\lambda\in\mathbb{R}^k_+$ with
$\lambda_1\geq\lambda_2\geq\cdots\geq\lambda_k\,.$ Given an integer $s$ with $0<s\leq k$,
let $\iota$ be the unique integer defined by Lemma \ref{Nik}. We define
$$
\textstyle\phi_s(\lambda):=\sum_{\ell=1}^{\iota} \log \lambda_\ell + (s - \iota)\log\left(\frac{1}{s-{\iota}} \sum_{\ell=\iota+1}^{k}
\lambda_\ell\right).
$$
%\end{equation*}
Next, for $X\in\mathbb{S}_{+}^k$\,, we define
$\Gamma_s(X):=  \phi_s(\lambda_1(X),\ldots,\lambda_k(X))$ where $\lambda_1(X)\ge \lambda_2(X)\ge \cdots\ge \lambda_k(X)$ are the eigenvalues of $X$.

Suppose that the rank of $C$ is $r\geq s$. Then we factorize $C=FF^\top$,
with $F\in \mathbb{R}^{n\times k}$, for some $k$ satisfying $r\le k \le n$.
% Remember that for the sake of simplicity, we exclude the data $C,s, A,b$ variables in the parameters when defining terms associated with each upper bound.
It has been established \cite[Theorem 2.2]{chen2023computing} that the value of the factorization bound is independent of the choice of $F$. Consequently, for the sake of simplicity, while certain terms may feature $F$ in their defining equations, it will not be included as a parameter for such terms.

Now, for
$\Upsilon \in \mathbb{R}_{++}^n$ and
$x\in [0,1]^n$, we define
\begin{align*}
    &F_{{\tiny\mbox{DDFact}}}(x;\Upsilon):
=\textstyle \sum_{i=1}^n \gamma_i x_i F_{i\cdot}^\top F_{i\cdot}~, \text{ and}\\
&\textstyle f_{{\tiny\mbox{DDFact}}}(x;\Upsilon):= \strut \Gamma_s(F_{{\tiny\mbox{DDFact}}}(x;\Upsilon)) -\sum_{i=1}^n x_i \log \gamma_i\,.
\end{align*}
\noindent Finally, we define the \emph{g-scaled factorization bound}
% \[
% \begin{array}{ll}
% 	z_{{\tiny\mbox{DDFact}}}(\Upsilon):= 	&\max \left\{  f_{{\tiny\mbox{DDFact}}}(x;\Upsilon)~:~\right. \\
%  &\qquad \left.\mathbf{e}^\top x=s,~
%  0\leq x\leq \mathbf{e}, ~ Ax\leq b\strut \right\}.
% \end{array}
% \tag{DDFact}\label{DDFact}
% \]
\begin{align}
	z_{{\tiny\mbox{DDFact}}}(\Upsilon):= 	\max \left\{  f_{{\tiny\mbox{DDFact}}}(x;\Upsilon ) \,:\,\mathbf{e}^\top x=s,~
 0\leq x\leq \mathbf{e}, ~ Ax\leq b\strut \right\}.
\tag{DDFact}\label{DDFact}
\end{align}
The reason for the nomenclature \ref{DDFact} is because it is obtained from the Lagrangian dual of the Lagrangian dual of a nonconvex continuous relaxation of  \ref{CMESP} (see \cite{chen2023computing}).
Note that \[
F_{{\tiny\mbox{DDFact}}}(x;\Upsilon)
= F^\top \Diag(\sqrt{\Upsilon}) \Diag(x) \Diag(\sqrt{\Upsilon}) F.
\]
So,  we can interpret \ref{DDFact} as applying the
unscaled \ref{DDFact} bound to the symmetrically-scaled  matrix $\Diag(\sqrt{\Upsilon}) F\Diag\big(\sqrt{\Upsilon}\big)F^\top \Diag(\sqrt{\Upsilon}) = C\Diag\big(\sqrt{\Upsilon}\big)$, and then
correcting by $- \sum_{i=1}^n  x_i\log \gamma_i$\,.

In what follows, the following notations will be employed:
\begin{align*}
    &\dom\left(\Gamma_s\right):=  \left\{X : X\succeq 0, \rank(X)\ge s \right\}, \text{ and }\\
    &\dom\left(f_{{\tiny\mbox{DDFact}}}; \Upsilon\right):=  \left\{x : F_{{\tiny\mbox{DDFact}}}(x;\Upsilon)\in \dom \left(\Gamma_s\right)\right\}
\end{align*}
being the domains of $\Gamma_s(X)$ and $f_{{\tiny\mbox{DDFact}}}(x;\Upsilon)$, respectively. Moreover, we denote
\begin{align*}
   \dom\left(f_{{\tiny\mbox{DDFact}}}; \Upsilon\right)_+:=  \left\{x : x\ge 0,  F_{{\tiny\mbox{DDFact}}}(x;\Upsilon)\in \dom \left(\Gamma_s\right)\right\}
\end{align*}
as the intersection of $\dom\left(f_{{\tiny\mbox{DDFact}}}; \Upsilon\right)$ and $\mathbb{R}^n_+$\,. Because the feasible solutions of \ref{DDFact} with finite objective values
% \mf{is it better to say here - Because all feasible solutions to \ref{DDFact}? Also, should we clarify here the relation between $\dom\left(f_{{\tiny\mbox{DDFact}}}; \Upsilon\right)_+$ and the feasible set of \ref{DDFact}? } \zc{agree, modified}
are evidently confined in $\dom\left(f_{{\tiny\mbox{DDFact}}}; \Upsilon\right)_+$\,, it is enough to concentrate on $\dom\left(f_{{\tiny\mbox{DDFact}}}; \Upsilon\right)_+$ instead of $\dom\left(f_{{\tiny\mbox{DDFact}}}; \Upsilon\right)$.%_+$.
We wish to highlight the following important point.

% This is applying the \ref{DDFact} bound to the right-hand-side of the identity \eqref{gscale:eqn} where $t_1= t_2=\frac{1}{2}$, i.e., we can interpret this bound as

\begin{remark}\label{rem:fullrank}
    Generally, we must choose a factorization with $k$ being at least the rank of $C$, but it is natural to choose one with $k$ equal to the rank of $C$;
    for example, via a spectral decomposition of $C$.
    In this case, $F_{{\tiny\mbox{DDFact}}}(x;\Upsilon)$ is full-rank if and only if $x\in\mathbb{R}^n_{++}$\,.
    In light of this, we can fully understand where on the boundary of the feasible region of \ref{DDFact}, we can encounter solutions not in
    % {\color{yellow}$\dom\left(f_{{\tiny\mbox{DDFact}}}; \Upsilon\right)$} \mf{i think should replace the yellow by: "
    the interior of $\dom\left(f_{{\tiny\mbox{DDFact}}}; \Upsilon\right)_+$\,.
    % " } .
\end{remark}

It is commonly assumed in the literature that the function $f_{{\tiny\mbox{DDFact}}}(x;\Upsilon)$ may exhibit non-smooth behavior in $x$, and toward this end, the supdifferential is characterized.  In their work, \cite{li2020best} utilized a Frank-Wolfe algorithm to evaluate \ref{DDFact} for the \hyperlink{MESP}{MESP}  case. Subsequently, \cite{chen2023computing} employed a BFGS-based algorithm of \texttt{Knitro} for \ref{DDFact},
to handle both \hyperlink{MESP}{MESP}  and \ref{CMESP}, wherein they utilized supgradient information to update the Hessian approximation.
This algorithm achieved superior performance in terms of both speed and accuracy, in the spirit of \cite{nonsmoothBFGS} which investigated the excellent performance of BFGS on non-smooth problems.
In the following section, we will establish that $f_{{\tiny\mbox{DDFact}}}(x;\Upsilon)$ is actually  \emph{in a certain generalized sense} ``differentiable'' in $x\in \dom\left(f_{{\tiny\mbox{DDFact}}}; \Upsilon\right)_+$\,.
% contradicting the non-smoothness assumption
% %falsifying the non-smoothness conjecture in
% of the literature.
These findings serve as a theoretical foundation for the efficiency of algorithms (e.g., those employed by  \cite{chen2023computing}) that rely on smoothness for their convergence.
We will introduce two necessary definitions to facilitate the establishment of our  generalized differentiability results.
% \begin{defn}\label{def:fact}
% For any $x \in \dom\left(f_{{\tiny\mbox{DDFact}}}; \Upsilon\right)_+$, suppose that the eigenvalues of $F_{{\tiny\mbox{DDFact}}}(x;\Upsilon)$ are $\lambda_1\ge \cdots\ge \lambda_r > \lambda_{r+1}=\cdots=\lambda_{k}
% =0$~, where $r\in [s,k]$ and $F_{{\tiny\mbox{DDFact}}}(x;\Upsilon) =  Q \Diag( \lambda ) Q $ with an orthonormal matrix $ Q $. Define $ \beta:=(\beta_1,\beta_2,\ldots,\beta_{k})^\top$ such that
% \begin{align*}
%   & \textstyle\beta_{i}:=\frac{1}{\lambda_{i}},~ \forall~ i \in[1, \iota],~\\
%   & \textstyle\beta_{i}:=\frac{s-\iota}{\sum_{i \in[\iota+1, k]} \lambda_{i}},~ \forall~ i \in[\iota+1, k],
% \end{align*}
% where $\iota$ is the unique integer defined in Lemma \ref{Nik}.
% \end{defn}

\begin{defn}\label{def:fact}
For $x \in \dom\left(f_{{\tiny\mbox{DDFact}}}; \Upsilon\right)_+\,$, let the eigenvalues of $F_{{\tiny\mbox{DDFact}}}(x;\Upsilon)$ be $\lambda_1\ge \cdots\ge \lambda_r> \lambda_{r+1}=\cdots= \lambda_{k}=0$, and $F_{{\tiny\mbox{DDFact}}}(x;\Upsilon) =  Q \Diag( \lambda ) Q $ with an orthonormal matrix $ Q $. Define $ \beta:=(\beta_1,\beta_2,\ldots,\beta_{k})^\top$ such that
\begin{align*}
  & \textstyle\beta_{i}:=\frac{1}{\lambda_{i}},~ \forall~ i \in[1, \iota],~\\
  & \textstyle\beta_{i}:=\frac{s-\iota}{\sum_{i \in[\iota+1, k]} \lambda_{i}},~ \forall~ i \in[\iota+1, k],
\end{align*}
where $\iota$ is the unique integer defined in Lemma \ref{Nik}.
\end{defn}

 In cases where an explicit analytic formula is unavailable for a function, such as the objective of \ref{DDFact}, the conventional definition of (Fr\'echet) differentiability only applies to points that exist within the interior of the function domain. This restriction presents challenges when attempting to analyze the properties of a function for points where the conventional definition of (Fr\'echet) differentiability is not defined, e.g., points at the boundary of the function domain, which
 is important for understanding the behavior of algorithms having iterates at such points. For our particular function, when we choose a factorization with $k$ equal to the rank of $C$, such points are precisely the ones with zero components (see Remark \ref{rem:fullrank}), and might well be visited by active-set methods.\footnote{In fact we will see in our computational results
 (Table \ref{table:factexperiments_boundary}) that they are frequently visited by active-set methods.}
 % \mf{i think this last comment is specific to our case} \zc{What you means by saying specific to our case?}
 % % further restricts our ability to design an efficient optimization algorithm.
 % For instance, it is impossible to ensure the convergence of an algorithm that has iterates to the boundary of the function domain, while it may be the fastest.
 To overcome this difficulty, we will extend the definition of (Fr\'echet) differentiability, in a natural way, to include points at the boundary of a (convex) set.

\begin{defn}\label{def:generaldiff}
We define a function $f:\mathbb{R}^n\rightarrow \mathbb{R}$ to be generalized differentiable with respect to a set $\mathcal A\subseteq \dom(f)$ if a linear operator $g(x):\mathbb{R}^n\rightarrow \mathbb{R}$ exists for all $x\in \mathcal{A}$, such that for all $d$ with $x+d\in \mathcal{A}$, we have $f(x+d)-f(x)-g(x)^\top d=o\left(\|d\|\right)$. We refer to $g(x)$ as the generalized gradient with respect to $\mathcal{A}$. We will omit ``with respect to $\mathcal{A}$" when it is clear from the context.
\end{defn}

% \jon{It seems like we should make some assumption about  $\mathcal{A}$
% being convex} \zc{I agree with that for an algorithm converging to optimality, convexity is required. But just for this definition, it is of problem without convexity. We can definitely change it to requiring convexity if you think it is better.
% }

\begin{remark}\label{rmk:generaldiff}
    We would like to highlight that our concept of generalized differentiability is almost as potent as differentiability on $\mathcal{A}$. Specifically, it possesses identical capabilities as differentiability if we use
    feasible-point optimization algorithms.
    % that always confines its iterates within $\mathcal{A}$, e.g., projected gradient algorithms, interior-point algorithms, and active-set algorithms.
    The reasons are as follows:
    \begin{enumerate}
        \item if $x$ lies in the interior of $\mathcal{A}$, then the generalized differentiability and generalized gradient are exactly differentiability and gradient, respectively;
        \item if $x$ lies on the boundary of $\mathcal{A}$, then the Whitney Extension Theorem (\cite[Theorem 1]{Whitney}) guarantees the existence of a compact neighborhood $\mathcal{N}_c(x) \subset \mathcal{A}$ such that the restriction of $f$ on $\mathcal{N}_c(x)$ has a continuously differentiable extension $\hat f$ on $\mathbb{R}^n$, with prescribed derivative information on $\mathcal{N}_c(x)$. In other words, $\hat f(x)=f(x), \frac{\partial \hat f(x)}{\partial x}= g(x)$ for all $x\in \mathcal{N}_c(x)$. Consequently, the generalized differentiability of $f$ is equivalent to its differentiability at the boundary point $x$, as long as we examine a larger open set that contains a local neighborhood of the boundary point;
        \item The equation $f(x+d)-f(x)-g(x)^\top d=o\left(\|d\|\right)$ implies that as $d$ approaches the zero vector, the expression $f(x+d)-f(x)-g(x)^\top d$ approaches zero, regardless of the \emph{path} taken by $d$. This statement is essentially the definition of differentiability, except that $x+d \in \mathcal{A}$. Consequently, if an optimization algorithm that always confines its iterates within $\mathcal{A}$ is utilized, the capabilities of generalized differentiability are identical to those of differentiability. Specifically, if this optimization algorithm converges under differentiability, it should also converge under generalized differentiability.
    \end{enumerate}
\end{remark}

In the subsequent analysis, we aim to establish the continuous generalized differentiability of the objective of \ref{DDFact} concerning its dependence on $\dom\left(f_{{\tiny\mbox{DDFact}}}; \Upsilon\right)_+$\,.
% By doing so, we can refute the conjecture of non-smoothness posited in the relevant literature and provide convergence guarantees for algorithms that have iterates positioned at the boundary of $\dom\left(f{{\tiny\mbox{DDFact}}}; \Upsilon\right)_+$\,.
This will give some theoretical understanding of
the good performance of algorithms that empirically have many iterates at the boundary of the feasible region, where
smoothness was in question.
Such algorithms were observed to outperform interior-point algorithms, which will be shown in the experiments.

\begin{thm}\label{thm:fact}
For all $\Upsilon \in \mathbb{R}_{++}^n$ in \ref{DDFact}, the following hold:
\begin{itemize}
\item[\ref{thm:fact}.i.] $z_{{\tiny\mbox{DDFact}}} (\Upsilon)$ yields a valid upper bound for the optimal value of \ref{CMESP}, i.e., $z(C,s,A,b)\leq z_{{\tiny\mbox{DDFact}}} (\Upsilon)$;
\item[\ref{thm:fact}.ii.] the function $f_{{\tiny\mbox{DDFact}}}(x;\Upsilon)$  is concave in $x$ on $\dom\left(f_{{\tiny\mbox{DDFact}}}; \Upsilon\right)_+$;
\item[\ref{thm:fact}.iii.] the function $f_{{\tiny\mbox{DDFact}}}(x;\Upsilon)$ is generalized differentiable with respect to \break $\dom\left(f_{{\tiny\mbox{DDFact}}}; \Upsilon\right)_+$\,, with generalized gradient
\begin{align*}
g_x(x;\Upsilon):= \Upsilon\circ \diag\left(F Q \Diag\left( \beta\right) Q ^\top F^\top\right) - \log\Upsilon,
\end{align*}
where $C=FF^\top$ is a factorization of $C$ and $Q, \beta$ are defined in Definition \ref{def:fact}. In particular, $g_x(x;\Upsilon)$ is invariant to different choices of $F, Q$ as long as we change $\beta$ accordingly;

\item[\ref{thm:fact}.iv.] given $x\in \dom\left(f_{{\tiny\mbox{DDFact}}}; \Upsilon\right)_+$\,, the function $f_{{\tiny\mbox{DDFact}}}(x;\Upsilon)$ is differentiable in $\Upsilon$ with gradient
\begin{align*}
    g_\Upsilon(x;\Upsilon):=  x\circ\diag\left(F Q \Diag\left( \beta\right) Q ^\top F^\top\right) - \Diag(\Upsilon)^{-1} x,
\end{align*}
where $C=FF^\top$ is a factorization of $C$ and $Q, \beta$ are defined in Definition \ref{def:fact}. In particular, $g_\Upsilon(x;\Upsilon)$ is invariant to different choices of $F, Q$, as long as we change $\beta$ accordingly. Additionally, for \hyperlink{MESP}{MESP}, let $x^*$ be an optimal solution to \ref{DDFact}; then we have
\begin{align*}
    \left.g_\Upsilon(x^*;\Upsilon)\right|_{\Upsilon=\mathbf{e}} = 0
\end{align*}
(which does \emph{not} generally hold for \ref{CMESP}, as we will see in \S\ref{sec:num}).
\item[\ref{thm:fact}.v.]
the function $f_{{\tiny\mbox{DDFact}}}(x;\Upsilon)$ is continuously generalized differentiable in $x$ and continuously differentiable in $\Upsilon$ on $\dom\left(f_{{\tiny\mbox{DDFact}}}; \Upsilon\right)_+$ $\times \mathbb{R}^n_{++}$\,, i.e., $g_x(x;\Upsilon)$ and $g_\Upsilon(x;\Upsilon)$ are continuous on $\dom\left(f_{{\tiny\mbox{DDFact}}}; \Upsilon\right)_+ \times \mathbb{R}^n_{++}$\,.
\end{itemize}

\end{thm}

\begin{remark}
\cite{nikolov2015randomized} established Theorem \ref{thm:fact}.\emph{i} for $\Upsilon:=\mathbf{e}$, and hence only regarded as a function of $x$, which was developed further in \cite{li2020best}. We generalize this result to the situation where $\Upsilon\in \mathbb{R}^n_{++}$ and is varying.
We note that the o-scaled factorization bound for \ref{CMESP} is
invariant under the scale factor (see \cite{chen2023computing}), so the
use of any type of scaling in the context of the \ref{DDFact} bound is completely new.
%%%%%%%%%%
Theorem \ref{thm:fact}.\emph{ii} is a result of \cite{nikolov2015randomized}, with details filled in by \cite[Section 3.4.2]{FLbook}.
%%%%%%%%%%
Theorem \ref{thm:fact}.\emph{iii} is the first differentiablity result of any type for the
\ref{DDFact} bound.
% The proof methods (see below) are quite technical and novel.
These results illuminate the success of  BFGS-based methods for calculating the \ref{DDFact} bound, not fully anticipated by previous works which exposed only supgradients connected to   \ref{DDFact}.
%%%%%%%%%%
Theorem \ref{thm:fact}.\emph{iv} provides the potential for fast algorithms leveraging BFGS-based methods to improve the \ref{DDFact} bound by g-scaling, as we will see in experiments \S\ref{sec:num}.
These observations and Theorem \ref{thm:fact}.\emph{iv} leave open
the interesting question of whether g-scaling can help
the \ref{DDFact} bound for \hyperlink{MESP}{MESP};
we can interpret Theorem \ref{thm:fact}.\emph{iv} as a partial result toward a negative answer. Theorem \ref{thm:fact}.v is a consequence of Theorems \ref{thm:fact}.iii,iv.
\end{remark}

\begin{proof}[Theorem \ref{thm:fact}.i,ii]

% \begin{itemize}
%     \item[\ref{thm:fact}.i,ii:]
    These are essentially  results of \cite{nikolov2015randomized}; see also \cite[Section 3.4]{FLbook}. Intuitively, \ref{DDFact} is the Lagrangian dual of the Lagrangian dual of a nonconvex continuous relaxation of \ref{CMESP} (see \cite{chen2023computing}). Therefore, \ref{DDFact} has a concave objective function, and
    the optimal value $z_{{\tiny\mbox{DDFact}}} (\Upsilon)$ serves as valid upper bound for the optimal value of \ref{CMESP}.
    % \item[\ref{thm:fact}.ii:] This is essentially a  result of \cite{nikolov2015randomized}, with details filled in by \cite[Section 3.4.2]{FLbook}. Intuitively, \ref{DDFact} is the Lagrangian dual of the Lagrangian dual of a nonconvex continuous relaxation of \ref{CMESP} (see \cite{chen2023computing}). As a result, it has a concave objective function.
    $\hfill \square$
%\end{itemize}

\end{proof}

% For clarity, we first introduce and prove several necessary preliminaries before going into the main part of proving Theorem \ref{thm:fact}.iii $\textstyle{\&}$ iv.

% The first two lemmas regards the properties of $\Gamma_s(X)$.

% \begin{lem}\label{FactX:prop1}
% For $X\in \dom\left(\Gamma_s\right)$, $\Gamma_s(X)= \min_{\Lambda \succ 0}\left\{- \ldet_{s}(\Lambda) + \Trace(X\Lambda) \right\}$ where $\ldet_{s}(\Lambda)$ denotes the product of the $s$ smallest eigenvalues of $\Lambda$. Furthermore, $\Gamma_s(X)$ is closed and concave.
% \end{lem}
% \begin{proof}
%     The proof of characterization can be found in \cite[Lemma 3]{li2020best}. For any fixed $\Lambda$, $- \ldet_{s}(\Lambda) + \Trace(X\Lambda)$ is an affine function of $X$ with closed and convex epigraph. The epigraph of $\Gamma_s(X)$ is essentially the intersection of all such epigraphs, thus also closed and convex, which means that $\Gamma_s(X)$ is closed and concave.
% \end{proof}

% \begin{lem}\label{FactX:prop2}
%     For $X\in \dom\left(\Gamma_s\right)$, $\Gamma_s(X)$ can be written as the infimum of affine majorants above it, in other words,
%     \begin{align*}
%     \Gamma_s(X)=\inf \{\alpha(X): \alpha(X)=\Trace\left(A^\top X\right)+b \ge \Gamma_s(X), ~\forall X \in \dom\left(\Gamma_s\right)\}.
% \end{align*}
% \end{lem}
% \begin{proof}
%     From Lemma \ref{FactX:prop1}, we know that $\Gamma_s(X)$ is concave and closed. Then the conclusion follows from  the Frenchel biconjugation (\cite[Theorem 4.2.1]{borwein2006convex}).
% \end{proof}

Toward establishing the generalized differentiability of $f_{{\tiny\mbox{DDFact}}}(x;\Upsilon)$, we begin by characterizing the directional derivatives. Toward this end, our first step is to derive the supdifferential of the objective of $\Gamma_s(X)$ with respect to $X\in \dom(\Gamma_s)$.
% , which is essentially the result of \cite[Proposition 2]{li2020best}.

\begin{prop}{\cite[Proposition 2]{li2020best}}\label{FactX:subg}
Given $X\in \dom(\Gamma_s)$ with rank $r \in[s, k]$, suppose that its eigenvalues are $\lambda_{1} \geq \cdots \geq \lambda_{r}>\lambda_{r+1}=\cdots=\lambda_{k}=0$ and $X=Q \operatorname{Diag}(\lambda) Q^{\top}$ with an orthonormal matrix $Q$. Then the supdifferential of the function $\Gamma_{s}(X)$ at $X$ denoted by $\partial \Gamma_{s}(X)$ is

\begin{align*}
&\partial \Gamma_{s}(X)=\left\{ \vphantom{ \beta_{i}=\textstyle\frac{s-\iota}{\sum_{i \in[\iota+1, k]} \lambda_{i}}}
Q \operatorname{Diag}(\beta) Q^{\top}: X=Q \operatorname{Diag}(\lambda) Q^{\top}, Q \text { is orthonormal, } \right.\\
&\quad \lambda_{1} \geq \cdots \geq \lambda_{r}>\lambda_{r+1}=\cdots=\lambda_{k}=0, \\
&\quad \left.\beta \in \operatorname{conv}\left\{\beta: \beta_{i}= \textstyle\frac{1}{\lambda_{i}}, \forall i \in[\iota], \beta_{i}=\frac{s-\iota}{\sum_{i \in[\iota+1, k]} \lambda_{i}}, \forall i \in[\iota+1, r], \right.\right.\\
& \quad\left. \left. \vphantom{\beta_{i}= \textstyle\frac{1}{\lambda_{i}}, \forall i \in[\iota], \beta_{i}=\frac{s-\iota}{\sum_{i \in[\iota+1, k]} \lambda_{i}}, \forall i \in[\iota+1, r],} \beta_{i} \geq \beta_{r}, \forall i \in[r+1, k]\right\}\right\},
\end{align*}
where $\iota$ is the unique integer defined in Lemma \ref{Nik}.
\end{prop}

\begin{remark}
   If $X \succ 0$, then $\beta$ is uniquely determined, resulting in a singleton supdifferential $\partial \Gamma_{s}(X)$ and differentiability of $\Gamma_s(X)$ at $X$.
   % , owing to the concavity of $\Gamma_s(X)$ with respect to $X$ \mf{the last part of this sentence is not clear to me} \zc{I want to say a concave function with a unique sup-gradient at some point is actually differentiable at that point}.
   This further implies the differentiability of $f_{{\tiny\mbox{DDFact}}}(x;\Upsilon)$ in $x$ by the chain rule when $F_{{\tiny\mbox{DDFact}}}(x;\Upsilon) \succ 0$. However, the supdifferential $\partial \Gamma_{s}(X)$ is not a singleton when $X$ is located on the boundary of the positive-semidefinite cone. This indicates that $\Gamma_{s}(X)$ is \emph{not} differentiable at such points.
   % Remember that
   %  $$ f_{{\tiny\mbox{DDFact}}}(x;\Upsilon)= \strut \Gamma_s(F_{{\tiny\mbox{DDFact}}}(x;\Upsilon)) -\textstyle\sum_{i=1}^n x_i \log \gamma_i.$$
   %  Suppose the rank of $C $ is smaller than $k$, because
   %  \begin{align*}
   %      &\rank\left(F_{{\tiny\mbox{DDFact}}}(x;\Upsilon)\right)\\
   %      = & \rank\left(F^\top \Diag(x\circ \Upsilon) F\right)\\
   %      =& \rank\left( \Diag(x\circ \Upsilon)^{1/2} F F^\top \Diag(x\circ \Upsilon)^{1/2}\right)\\
   %      =& \rank\left( \Diag(x\circ \Upsilon)^{1/2}C \Diag(x\circ \Upsilon)^{1/2}\right)\\
   %      \le & \rank(C),
   %  \end{align*}
   %  we will always have $\rank\left(F_{{\tiny\mbox{DDFact}}}(x;\Upsilon)\right) < k$, i.e., $F_{{\tiny\mbox{DDFact}}}(x;\Upsilon)$ lies on the boundary of the positive semidefinite cone. This means that $f_{{\tiny\mbox{DDFact}}}(x;\Upsilon)$ is always non-differentiable in $F_{{\tiny\mbox{DDFact}}}(x;\Upsilon)$ over the feasible region.
   However, as we will show later, such non-differentiability does not really transfer to $x$. In fact, $f_{{\tiny\mbox{DDFact}}}(x;\Upsilon)$ is generalized differentiable at every $x\in \dom\left(f_{{\tiny\mbox{DDFact}}}; \Upsilon\right)_+$\,. In other words, even if $\Gamma_s(F_{{\tiny\mbox{DDFact}}}(x;\Upsilon))$ is non-differentiable in $F_{{\tiny\mbox{DDFact}}}(x;\Upsilon)$, $f_{{\tiny\mbox{DDFact}}}(x;\Upsilon) = \Gamma_s(F_{{\tiny\mbox{DDFact}}}(x;\Upsilon)) -\sum_{i=1}^n x_i \log \gamma_i$  is still generalized differentiable in $x$.
\end{remark}

The subsequent step involves computing the directional derivative of $\Gamma_s(X)$ at $X$, using the supdifferential characterized in Proposition \ref{FactX:subg}. It is a well-known fact that if $X$ is located in the interior of $\dom(\Gamma_s)$,  then
\begin{align*}
\Gamma'_s(X;D)=\inf_{G\in \partial \Gamma_s( X)} \Trace(G^\top D),
\end{align*}
where $D$ is a feasible direction at $X$ in $\dom(\Gamma_s)$; see e.g. \cite[Theorem 23.4]{rockafellar1997convex}. However, in our current context, $X$ might lie on the boundary of $\dom(\Gamma_s)$. Fortunately, \cite[p. 65]{moreau1966fonctionnelles}  provides a result that ensures that the same formula holds if $\Gamma_s(X)$ is continuous at $X$. Thus, our first step is to establish the continuity of $\Gamma_s(X)$ at $X \in \dom(\Gamma_s)$.
\begin{lem}\label{FactX:continuity}
$\Gamma_{s}(X)$ is continuous on its domain.
%$\dom\left(\Gamma_s\right)$.
\end{lem}
\begin{proof}
Consider $X\in \dom\left(\Gamma_s\right)$ with eigenvalues $\lambda_1\ge\lambda_2\ge \cdots\ge \lambda_k\ge 0$ and $\iota$ defined in Lemma \ref{Nik}. Let $P\in \mathbb{S}^n$ be such that $X+P\in \dom\left(\Gamma_s\right)$ and $\hat \lambda_1\ge\hat \lambda_2\ge \cdots\ge \hat \lambda_k$ be the eigenvalues of $X+P$ with $\hat \iota$ again defined in Lemma \ref{Nik}. We will use the continuity of eigenvalues (with respect to entries of the matrix) to prove the result.

We discuss two sub-cases:
\begin{itemize}
    \item[1.] $\lambda_{\iota}>\frac{1}{s-\iota}\sum_{\ell=\iota+1}^k \lambda_{\ell}> \lambda_{\iota +1}$\,. Then for $\|P\|$ small enough, by the continuity of eigenvalues, we have $\hat \lambda_{\iota}>\frac{1}{s-\iota}\sum_{\ell=\iota+1}^k \hat\lambda_{\ell}> \hat\lambda_{\iota +1}$\,, which implies $\hat \iota=\iota$. Again by the continuity of eigenvalues and $\log(\cdot)$, $\Gamma_s(X)$ is continuous at $X$ on $\dom\left(\Gamma_s\right)$.
    \item[2.] $\lambda_{\iota}>\frac{1}{s-\iota}\sum_{\ell=\iota+1}^k \lambda_{\ell}= \lambda_{\iota +1}$\,. We have to be more careful in this case as any small $\|P\|$ can make $\hat{\iota}$ different from $\iota$. We first characterize a range where $\hat{\iota}$ should lie within. Let $\iota_e:=\max\{i: \lambda_i=\lambda_{\iota+1}, s> i\ge \iota+1\}$. We claim that $\hat \iota\in [\iota+1, \iota_e]$. Before proving this claim, we demonstrate three preliminary results:

    \begin{itemize}
        \item[(a)] For any $i\le \iota$, $\lambda_i > \frac{1}{s-i}\sum_{\ell=i+1}^k \lambda_{\ell}$\,;
        \item[(b)] For any $i\le \iota-1$, $\frac{1}{s-i}\sum_{\ell=i+1}^k \lambda_{\ell} < \lambda_{i+1}$\,;
        \item[(c)] For any $s> i> \iota_e$, $\lambda_i < \frac{1}{s-i}\sum_{\ell=i+1}^k \lambda_{\ell}$\,.
    \end{itemize}

Note that (a) holds for $i=0$. Assume that there exists some $i\le \iota$ such that $\lambda_i \le \frac{1}{s-i}\sum_{\ell=i+1}^k \lambda_{\ell}$\,. Without loss of generality, let $i$ be the minimum integer satisfying this condition. Obviously $i\ge 1$. Furthermore,
\begin{align*}
    \textstyle\frac{1}{s-i+1}\sum_{\ell=i}^k \lambda_{\ell}=  \frac{\left(\sum_{\ell=i+1}^k \lambda_{\ell}\right) + \lambda_i}{s-i+1} \ge \frac{1}{s-i+1} ((s-i)\lambda_i + \lambda_i) = \lambda_i\,.
\end{align*}
By assumption, we also have that $\lambda_{i-1}>\frac{1}{s-i+1}\sum_{\ell=i}^k \lambda_{\ell}$\,, which together with the above deduction, implies that $\iota = i-1\le \iota-1$, a contradiction.
For (b), if there exists $i\le \iota-1$ with $\frac{1}{s-i}\sum_{\ell=i+1}^k \lambda_{\ell} \ge \lambda_{i+1}$\,, together with (a), we have $\iota = i\le \iota-1$, a contradiction. Finally, (c) comes from
\begin{align*}
        (s-i)\lambda_i < (s-i)\lambda_{\iota+1} = &~ (s-\iota)\lambda_{\iota+1}- (i-\iota)\lambda_{\iota+1}\\
        \le &~ \textstyle\sum_{\ell=\iota+1}^k \lambda_{\ell}- (i-\iota)\lambda_{\iota+1}\\
        = &~ \textstyle\sum_{\ell=\iota+1}^i (\lambda_{i}-\lambda_{\iota+1}) + \sum_{\ell=i+1}^k \lambda_\ell\\
        \le & ~\textstyle\sum_{\ell=i+1}^k \lambda_\ell\,.
    \end{align*}
    With (a--c), and the continuity of eigenvalues with respect to the entries of a matrix, for $\|P\|$ small enough, we have:
    \begin{itemize}
        \item[($\hat a$)] For any $i\le \iota$, $\hat \lambda_i > \frac{1}{s-i}\sum_{\ell=i+1}^k \hat\lambda_{\ell}$\,;
        \item[($\hat b$)] For any $i\le \iota-1$, $\frac{1}{s-i}\sum_{\ell=i+1}^k \hat \lambda_{\ell} < \hat \lambda_{i+1}$\,;
        \item[($\hat c$)] For any $s> i> \iota_e$, $\hat \lambda_i < \frac{1}{s-i}\sum_{\ell=i+1}^k\hat \lambda_{\ell}$\,.
    \end{itemize}
    $(\hat a-\hat c)$ suggest that $\hat \iota \in [\iota+1, \iota_e]$, otherwise the condition $
 \hat \lambda_{\hat \iota}>\frac{1}{s-\hat \iota}\sum_{\ell=\hat \iota+1}^k \hat \lambda_{\ell}$ $\geq \hat \lambda_{\hat \iota +1}$ in Lemma \ref{Nik} will be violated.

From $\frac{1}{s-\iota}\sum_{\ell=\iota+1}^k \lambda_{\ell}= \lambda_{\iota +1}$ and the definition of $\iota_e$\,, we also have
 \begin{align}
     \textstyle\lambda_{\iota+1} = \frac{1}{s-\iota-1}\sum_{\ell=\iota+2}^k \lambda_{\ell}\nonumber
     = &~\textstyle\lambda_{\iota+2} =  \frac{1}{s-\iota-2}\sum_{\ell=\iota+3}^k \lambda_{\ell}\nonumber\\
     =& ~\textstyle\cdots\nonumber\\
     =&~  \textstyle\lambda_{\iota_e} = \frac{1}{s-\iota_e}\sum_{\ell=\iota_e+1}^k \lambda_{\ell}\,. \label{FactX:continuity-eqn1}
 \end{align}

 We are now ready to prove the continuity results. Because $\hat\iota\in [\iota+1, \iota_e]$ and \eqref{FactX:continuity-eqn1}, we have
\begin{align*}
    &\textstyle\sum_{\ell=1}^{\iota} \log(\lambda_\ell)+(s-\iota)\log\left(\frac{1}{s-\iota}\sum_{\ell=\iota+1}^k \lambda_{\ell}\right)\\
    &\quad= \textstyle\sum_{\ell=1}^{\iota} \log(\lambda_\ell)+(\hat\iota-\iota)\log\!\left(\!\frac{1}{s-\iota}\sum_{\ell=\iota+1}^k \lambda_{\ell}\!\right) \!+\! (s-\hat \iota)\log\!\left(\!\frac{1}{s-\iota}\sum_{\ell=\iota+1}^k \lambda_{\ell}\!\right)\\
    &\quad= \textstyle\sum_{\ell=1}^{\iota} \log(\lambda_\ell)+\sum_{\ell=\iota+1}^{\hat \iota} \log(\lambda_\ell) + (s-\hat \iota)\log\left(\frac{1}{s-\hat \iota}\sum_{\ell=\hat \iota+1}^k \lambda_{\ell}\right)\\
    &\quad= \textstyle\sum_{\ell=1}^{\hat \iota} \log(\lambda_\ell) + (s-\hat \iota)\log\left(\frac{1}{s-\hat \iota}\sum_{\ell=\hat \iota+1}^k \lambda_{\ell}\right).
\end{align*}
Note that the above equation holds for any $\|P\|$ small enough, and with its corresponding $\hat \iota$ and at $X+P$, we have
\begin{align*}
    \textstyle\Gamma_s(X+P)=\sum_{\ell=1}^{\hat\iota} \log(\hat \lambda_\ell)+(s-\hat \iota)\log\left(\frac{1}{s-\hat\iota}\sum_{\ell=\hat \iota+1}^k \hat \lambda_{\ell}\right).
\end{align*}
Then by the continuity of eigenvalues (with respect to elements of the matrix) and $\log(\cdot)$ function, we conclude that  $\Gamma_s(X)$ is continuous on $\dom(\Gamma_s)$.
\end{itemize}

$\hfill \square$
\end{proof}

\noindent The continuity of $\Gamma_s(X)$ in $X$ together with the continuity of eigenvalues in matrix elements also implies the following.

% continuity of $f_{{\tiny\mbox{DDFact}}} (x;\Upsilon)$ in $(x, \Upsilon)$. \mf{should we write this sentence more formally, saying where the continuity holds in each case?}\zc{modified, please check}
\begin{cor}\label{Factx:continuity}
$f_{{\tiny\mbox{DDFact}}} (x;\Upsilon)$ is continuous in $(x, \Upsilon)$ on $\dom\left(f_{{\tiny\mbox{DDFact}}}; \Upsilon\right)_+$\,.
\end{cor}

\noindent Utilizing the above results, we can characterize the directional derivative of $\Gamma_s(X)$ and further characterize the directional derivative of $f_{{\tiny\mbox{DDFact}}} (x;\Upsilon)$.
\begin{prop}\label{FactX:dirg}
    For $ X\in \dom\left(\Gamma_s\right)$, let $D\in \mathbb{S}^n$ be such that $X+D\in \dom\left(\Gamma_s\right)$; then the directional derivative of $\Gamma_s(X)$ at $X$ in the direction $D$, denoted  $\Gamma'_s(X;D)$, exists and
\begin{align*}
    \Gamma'_s(X;D)=\inf_{G\in \partial \Gamma_s( X)} \Trace(G^\top D).
\end{align*}
\end{prop}
\begin{proof}
   By definition, \cite[Lemma 3]{li2020best} and Lemma \ref{FactX:continuity}, $\Gamma_s(X)$ is convex, finite, and continuous in $X$ over $\dom(\Gamma_s)$. Then the conclusion follows by \cite[p. 65]{moreau1966fonctionnelles}.
   $\hfill \square$
\end{proof}

\begin{prop}\label{Factx:dirg}
For  $x\in \dom\left(f_{{\tiny\mbox{DDFact}}}; \Upsilon\right)_+$\,, let $d\in \mathbb{R}^n$ be such that $x+d\in \dom\left(f_{{\tiny\mbox{DDFact}}}; \Upsilon\right)_+$\,; then the directional derivative of $f_{{\tiny\mbox{DDFact}}} (x;\Upsilon)$ at $x$ in the direction $d$ exists, and
\[
f'_{{\tiny\mbox{DDFact}}} (x;\Upsilon;d) = \left(\Upsilon \circ \diag\left(FQ\Diag\left(\beta\right)QF^\top\right)\right)^\top d - \log(\Upsilon)^\top d,
\]
where $C=FF^\top $ is a factorization of $C$, and $Q, \beta$ are  defined in Definition \ref{def:fact}.
In particular, $FQ\Diag\left(\beta\right)QF^\top$ and thus $f'_{{\tiny\mbox{DDFact}}} (x;\Upsilon;d)$ is invariant to the choice of $F, Q$, as long as we change $\beta$ accordingly.
\end{prop}
\begin{proof}
    By Theorem \ref{thm:fact}.ii, we have that $f_{{\tiny\mbox{DDFact}}} (x;\Upsilon)$ is concave. Then by \cite[Theorem 23.1]{rockafellar1997convex}, the directional derivative $f'_{{\tiny\mbox{DDFact}}} (x;\Upsilon;d)$ exists and
\begin{align*}
    &f'_{{\tiny\mbox{DDFact}}} (x;\Upsilon;d) \\
    &\quad=~\textstyle\lim_{t\rightarrow 0^+} \frac{f_{{\tiny\mbox{DDFact}}} (x+td;\Upsilon)-f_{{\tiny\mbox{DDFact}}} (x;\Upsilon)}{t}\\
    &\quad=~\textstyle\lim_{t\rightarrow 0^+} \frac{\Gamma_s\left(F_{{\tiny\mbox{DDFact}}}(x+td;\Upsilon)\right)-\Gamma_s\left(F_{{\tiny\mbox{DDFact}}}(x;\Upsilon)\right) + t\log(\Upsilon)^\top d}{t}\\
    &\quad=~\textstyle\lim_{t\rightarrow 0^+} \frac{\Gamma_s\left(F_{{\tiny\mbox{DDFact}}}(x;\Upsilon)+ t F_{{\tiny\mbox{DDFact}}}(d;\Upsilon)\right)-\Gamma_s\left(F_{{\tiny\mbox{DDFact}}}(x;\Upsilon)\right)+ t\log(\Upsilon)^\top d}{t}\\
&\quad=~\textstyle\Gamma'_s\left(F_{{\tiny\mbox{DDFact}}}(x;\Upsilon);F_{{\tiny\mbox{DDFact}}}(d;\Upsilon)\right)+ \log(\Upsilon)^\top d\\
    &\quad=~\inf_{G\in \partial \Gamma_s\left(F_{{\tiny\mbox{DDFact}}}(x;\Upsilon)\right)} \Trace\left(G^\top F_{{\tiny\mbox{DDFact}}}(d;\Upsilon)\right)+\log(\Upsilon)^\top d,
\end{align*}
where the last equation is due to Proposition \ref{FactX:dirg}. Let $\Theta(x, \Upsilon)$ denote the set of $(Q,\beta)$ in the characterization of $\partial \Gamma_s\left(F_{{\tiny\mbox{DDFact}}}(x;\Upsilon)\right)$, as described in Proposition \ref{FactX:subg}. In particular,
\begin{align*}
&\Theta(x, \Upsilon)=\left\{ \vphantom{\textstyle\beta_{i}=\frac{1}{\lambda_{i}}, \forall i \in[\iota], \beta_{i}=\frac{s-\iota}{\sum_{i \in[\iota+1, k]} \lambda_{i}}, \forall i \in[\iota+1, r],}
(Q, \beta): F_{{\tiny\mbox{DDFact}}}(x;\Upsilon) =Q \operatorname{Diag}(\lambda) Q^{\top}, Q \text { is orthonormal, }\right. \\
&\quad \lambda_{1} \geq \cdots \geq \lambda_{r}>\lambda_{r+1}=\cdots=\lambda_{k}=0, \\
&\quad \left.\beta \in \operatorname{conv} \left\{\beta: \beta_{i}=\textstyle\frac{1}{\lambda_{i}}, \forall i \in[\iota], \beta_{i}=\frac{s-\iota}{\sum_{i \in[\iota+1, k]} \lambda_{i}}, \forall i \in[\iota+1, r], \right.\right.\\
& \quad \left. \left. \vphantom{\textstyle\beta_{i}=\frac{1}{\lambda_{i}}, \forall i \in[\iota], \beta_{i}=\frac{s-\iota}{\sum_{i \in[\iota+1, k]} \lambda_{i}}, \forall i \in[\iota+1, r],} \beta_{i} \geq \beta_{r}, \forall i \in[r+1, k]\right\}\right\},
\end{align*}
where $\iota$ is the unique integer defined in Lemma \ref{Nik}. According to the former derivation,
\begin{align}
    &f'_{{\tiny\mbox{DDFact}}} (x;\Upsilon;d)\nonumber\\
    &~=~\inf_{G\in \partial \Gamma_s\left(F_{{\tiny\mbox{DDFact}}}(x;\Upsilon)\right)} \Trace\left(G^\top F_{{\tiny\mbox{DDFact}}}(d;\Upsilon)\right)+\log(\Upsilon)^\top d\nonumber\\
    &~=~\inf_{(Q,\beta)\in \Theta(x;\Upsilon)} \Trace\left(FQ\Diag(\beta)Q^\top F^\top \Diag(\Upsilon\circ d)\right)+\log(\Upsilon)^\top d\nonumber\\
    &~=~\inf_{(Q,\beta)\in \Theta(x;\Upsilon)} \left(\Upsilon\circ\diag\left(FQ\Diag(\beta)Q^\top F^\top\right)\right)^\top d+\log(\Upsilon)^\top d.\label{Factx:dirg-eqn1}
\end{align}
We now show that the infimum of \eqref{Factx:dirg-eqn1} is obtained by $(Q,\beta)$ characterized in Definition \ref{def:fact}. For simplicity, we let
\begin{align}\label{def:g}
\begin{array}{rll}
   g(Q,\beta; \Upsilon):= & \Upsilon\circ\diag\left(FQ\Diag(\beta)Q^\top F^\top\right);&\\
    g_j(Q, \beta; \Upsilon):= &\gamma_i  F_{i\cdot} Q\Diag(\beta)Q^{\top} F_{i\cdot}^{\top}\,, & \text{the } i^{\text{th}} \text{ element of } g(Q,\beta; \Upsilon); \\
    q_j: = & Q_{\cdot j}\,, &\text{the } j^{\text{th}} \text{ column of } Q
\end{array}
\end{align}
(where $F_{i\cdot}$ denotes the $i^{\text{th}}$ row of $F$). By the definition of $\Theta(x, \Upsilon)$, we also have that if $j_1>j_2$\,, the eigenvalues $\lambda_{j_1}, \lambda_{j_2}$ associated with $q_{j_1}, q_{j_2}$ satisfy $\lambda_{j_1}\le\lambda_{j_2}$\,.

We claim that for any $1\le i\le n, r< j\le k$ where $x_i>0$, $F_{i\cdot } q_j = 0$. First by the characterization of $Q$ in $\Theta(x, \Upsilon)$, we have that $q_{j_1}^\top q_{j_2}=0$ for all $1\le j_1\le r<j_2\le k$, because $q_{j_1}, q_{j_2}$ lie in eigenspaces corresponding to different eigenvalues. Therefore, it is enough to prove that $F_{i\cdot}$ lies in the space spanned by $\{q_j ~:~ 1\le j\le r\}$. Notice that $F_{{\tiny\mbox{DDFact}}}(x;\Upsilon) = \sum_{i=1}^n \gamma_i x_i  F_{i\cdot}^{\top}F_{i\cdot} = F^\top \Diag(\Upsilon\circ x) F$. Therefore, the column space of $F_{{\tiny\mbox{DDFact}}}(x;\Upsilon)$ is equal to the row space of $\Diag(\Upsilon\circ x)^{\frac{1}{2}} F$, which is in turn equal to the space spanned by $\{F_{i\cdot}^\top ~:~ 1\le i\le n,\, x_i >0\}$. On the other hand, $F_{{\tiny\mbox{DDFact}}}(x;\Upsilon)= Q\Diag(\lambda)Q^\top$, and thus the column space of $F_{{\tiny\mbox{DDFact}}}(x;\Upsilon)$ is equal to the row space of $\Diag(\lambda)^{\frac{1}{2}}Q^\top$, which is in turn equal to the space spanned by $\{q_j ~:~ 1\le j \le r\}$. Therefore, we have proved that
\begin{align*}
    \text{span}\{F_{i\cdot}^\top ~:~ 1\le i\le n,\, x_i >0\} = \text{span}\{q_j ~:~ 1\le j \le r\},
\end{align*}
which implies that for all $1\le i\le n, r< j\le k$ where $x_i>0$, $F_{i\cdot } q_j = 0$. With this result, we have when $x_i>0$,
\begin{align*}
    g_i(Q,\beta; \Upsilon)= &~ \gamma_i F_{i\cdot} Q\Diag(\beta)Q^{\top} F_{i\cdot}^{\top}=\textstyle\gamma_i \sum_{j=1}^r \beta_j \|F_{i\cdot} q_j\|^2,
\end{align*}
which is invariant to $(Q,\beta)\in \Theta(x;\Upsilon)$ because $\beta_j, 1\le j\le r$ are fixed and $Q$ is not contained in the right-hand side formula.

We choose some $(\hat Q, \hat\beta)$ defined in Definition \ref{def:fact}. Note that the choice of $(\hat Q,  \hat\beta)$ is not unique. Then we can write the directional derivative as
\begin{align*}
    &f'_{{\tiny\mbox{DDFact}}} (x;\Upsilon;d)\\
    &=\inf_{(Q, \beta)\in \Theta(x;\Upsilon)} \textstyle\sum_{x_i >0 } g_i(Q, \beta; \Upsilon) d_i +\sum_{x_i = 0 } g_i(Q, \beta; \Upsilon) d_i+\log(\Upsilon)^\top d\\
    &=  {\textstyle\sum_{x_i >0 }} g_i\left(\hat Q, \hat\beta; \Upsilon\right) d_i + \inf_{(Q,\beta)\in \Theta(x;\Upsilon)} {\textstyle\sum_{x_i = 0 } \gamma_i \sum_{j=1}^n \beta_j \|F_{i\cdot} q_j\|^2 d_i+\log(\Upsilon)^\top d}\\
    &= {\textstyle\sum_{x_i >0 }} g_i\left(\hat Q, \hat\beta; \Upsilon\right) d_i + \inf_{(Q,\beta)\in \Theta(x;\Upsilon)} \textstyle\sum_{j=1}^n \beta_j \sum_{x_i = 0 }\gamma_i    \|F_{i\cdot} q_j\|^2 d_i+\log(\Upsilon)^\top d.
\end{align*}
Note that if $x_i=0$, we must have $d_i\ge 0$ to make $x+d\in \dom\left(f_{{\tiny\mbox{DDFact}}}; \Upsilon\right)_+$\,. Therefore, for each $r < j\le k$, $\sum_{i: x_i = 0 } \gamma_i \|F_{i\cdot} q_j\| d_i\ge 0$, and the infimum is achieved if and only if each $\beta_j\,, r < j\le k$ takes the minimum value in $\Theta(x; \Upsilon)$, which is easy to see that is just the value in Definition \ref{def:fact}. In particular, $(\hat Q, \hat\beta)$ is such a choice. Specifically, we have
\begin{align*}
    &f'_{{\tiny\mbox{DDFact}}} (x;\Upsilon;d)\\
    &\quad=\textstyle\sum_{x_i >0 } g_i\left(\hat Q, \hat\beta; \Upsilon\right) d_i +  \gamma_i\sum_{j=1}^n \hat \beta_j \sum_{x_i = 0 }    \|F_{i\cdot} \hat q_j\|^2 d_i+\log(\Upsilon)^\top d\\
    &\quad=\textstyle\sum_{ x_i >0 } g_i\left(\hat Q,\hat\beta; \Upsilon\right) d_i +  \gamma_i\sum_{x_i = 0 } g_i\left(\hat Q, \hat \beta; \Upsilon\right) d_i+\log(\Upsilon)^\top d\\
    &\quad=  g_i\left(\hat Q, \hat\beta; \Upsilon\right)^\top d+\log(\Upsilon)^\top d.
\end{align*}
Note that $\hat Q$ can be any $Q$ defined in Definition \ref{def:fact}, and the value of $g(Q,\beta; \Upsilon)$ is invariant as long as we change $\hat \beta$ accordingly. The invariance relative to $F$ is due to the invariance of $f_{{\tiny\mbox{DDFact}}} (x;\Upsilon)$ relative to $F$ (see \cite[Theorem 2.2]{chen2023computing}). $\hfill \square$
\end{proof}

With the characterization of directional derivative in Proposition \ref{Factx:dirg}, we can prove the general differentiability with respect to $\dom\left(f_{{\tiny\mbox{DDFact}}}; \Upsilon\right)_+$ as defined in Definition \ref{def:generaldiff}.
\begin{proof}[Theorem \ref{thm:fact}.iii,iv,v]
\begin{itemize}
    \item[\ref{thm:fact}.iii:]
    By Proposition \ref{Factx:dirg}, let $g_x(x;\Upsilon):= \Upsilon\circ \diag\left(FQ\Diag(\beta)QF^\top\right)+\log(\Upsilon)$ for any $(Q,\beta)$ defined in Definition \ref{def:fact}. Proposition \ref{Factx:dirg} shows that $(x;\Upsilon)$ is invariant to the choice of $(Q,\beta)$ and $F$. Then the directional derivative of $f_{{\tiny\mbox{DDFact}}}(x;\Upsilon)$ with respect to $x\in \dom\left(f_{{\tiny\mbox{DDFact}}}; \Upsilon\right)_+$ and feasible direction $d\in \mathbb{R}^n$ such that $x+d\in \dom\left(f_{{\tiny\mbox{DDFact}}}; \Upsilon\right)_+$ is $g_x(x;\Upsilon)^\top d$.

    We first demonstrate two preliminary results:
    \begin{itemize}
        \item[(a)] Given $x\in \dom\left(f_{{\tiny\mbox{DDFact}}}; \Upsilon\right)_+$, we define the neighbourhood of $x$ with radius $r$ with respect to $\dom\left(f_{{\tiny\mbox{DDFact}}}; \Upsilon\right)_+$ as
        \begin{align*}
           \mathcal{N}_r(x) := \left\{y ~:~ \|y-x\|\le r, y\in  \dom\left(f_{{\tiny\mbox{DDFact}}}; \Upsilon\right)_+\right\}.
        \end{align*}
        We claim that for $r$ small enough, $\mathcal{N}_r(x)$ is a compact set. Recall that:
        \begin{align*}
            &\dom\left(\Gamma_s\right):=  \left\{X : X\succeq 0, \rank(X)\ge s \right\}, \text{ and }\\
            &\dom\left(f_{{\tiny\mbox{DDFact}}}; \Upsilon\right)_+:=  \left\{x : x\ge 0, F_{{\tiny\mbox{DDFact}}}(x;\Upsilon)\in \dom \left(\Gamma_s\right)\right\}.
        \end{align*}
        By the continuity of eigenvalues, there is some small enough $\tilde{r}>0$ such that when $r\le \tilde{r}$, $F_{{\tiny\mbox{DDFact}}}(y;\Upsilon)$ has at least the same number of nonzero eigenvalues as $F_{{\tiny\mbox{DDFact}}}(x;\Upsilon)$, and so $\rank\left(F_{{\tiny\mbox{DDFact}}}(y;\Upsilon)\right)\ge s$. Moreover, note that the set $\{x~:~ F_{{\tiny\mbox{DDFact}}}(x;\Upsilon)\succeq 0\}$ and the non-negative cone $\mathbb{R}_+^n=\{x~:~ x\ge 0\}$ are  closed. So  $\mathcal{N}_r(x)$ can be seen as the intersection of  $\{x~:~ F_{{\tiny\mbox{DDFact}}}(x;\Upsilon)\succeq 0\}$, $\mathbb{R}_+^n$,  and the sphere $\{y ~:~ \|y-x\|\le r\}$, thus is closed and bounded, and thus compact.
        Furthermore, we have shown in Corollary \ref{Factx:continuity} that $f_{{\tiny\mbox{DDFact}}}(x;\Upsilon)$ is continuous over $\dom\left(f_{{\tiny\mbox{DDFact}}}; \Upsilon\right)_+$, thus uniform continuous over  $\mathcal{N}_r(x)$ for $r\le \tilde{r}$.

        \item[(b)] Given $x\in \dom\left(f_{{\tiny\mbox{DDFact}}}; \Upsilon\right)_+$, we define the circle of  $x$ with radius $r$ with respect to $\dom\left(f_{{\tiny\mbox{DDFact}}}; \Upsilon\right)_+$ as
        \begin{align*}
            \mathcal{C}_r(x) := \left\{y~:~ \|y-x\|= r, y\in  \dom\left(f_{{\tiny\mbox{DDFact}}}; \Upsilon\right)_+\right\}.
        \end{align*}
        With similar logic to the above, when $r\le \tilde{r}$, $\mathcal{C}_{r}(x)$ is closed and bounded. Then by Heine-Borel Theorem, for any $\epsilon>0$, there exists a finite set $F\subset \mathcal{C}_{\tilde{r}}(x)$ such that for any $y\in \mathcal{C}_{\tilde{r}}(x)$, there exists $u\in F$ such that $\|y-u\|<\epsilon$.
    \end{itemize}
    \smallskip
    Now we are ready to establish generalized differentiability of $f_{{\tiny\mbox{DDFact}}}(x;\Upsilon)$ with respect to $\dom\left(f_{{\tiny\mbox{DDFact}}}; \Upsilon\right)_+$\,. In particular, we want to demonstrate that for any $\epsilon>0$, there exists some $\delta>0$ such that whenever $y\in \dom\left(f_{{\tiny\mbox{DDFact}}}; \Upsilon\right)_+$ and $\|y-x\|< \delta$, we have
    \begin{align*}
        \left|f_{{\tiny\mbox{DDFact}}}(y;\Upsilon) - f_{{\tiny\mbox{DDFact}}}(x;\Upsilon)- g_x(x;\Upsilon)^\top (y-x)\right| < \epsilon.
    \end{align*}

    We will assume that $g_x(x;\Upsilon)\neq 0$, because the case where $g_x(x;\Upsilon)= 0$ is implied by the continuity of $f_{{\tiny\mbox{DDFact}}}(x;\Upsilon)$ (see Corollary \ref{Factx:continuity}). We have the following four facts:
    \begin{itemize}
        \item [(1)] from (a), $f_{{\tiny\mbox{DDFact}}}(x;\Upsilon)$ is uniformly continuous on $\mathcal{N}_{\tilde{r}}(x)$. Then given $\epsilon>0$, there is some $\delta_1 > 0$ such that for any $x_1, x_2\in \mathcal{N}_{\tilde{r}}(x)$ such that $\|x_1-x_2\| < \frac{\delta_1}{\|g_x(x;\Upsilon)\|}$, we have $|f_{{\tiny\mbox{DDFact}}}(x_1;\Upsilon)-f_{{\tiny\mbox{DDFact}}}(x_2;\Upsilon)|<\frac{\epsilon}{3}$.
        \item[(2)] given $\delta_1>0$, by (b), there is some finite set $F\subset \mathcal{C}_{\tilde{r}}(x)$ such that for every $y\in \mathcal{C}_{\tilde{r}}(x)$, there exists $u\in F$ such that $\|y-u\|< \frac{\min\{\epsilon, \delta_1\}}{3\cdot \|g_x(x;\Upsilon)\|}$.
        \item[(3)] because of the finiteness of $F$ and the existence of the directional derivative in direction $u-x$, $\forall\, u\in F$, given $\epsilon>0$, there exists some $\delta_2\le 1$ such that for any $u\in F$, when $t< \delta_2$, we have
        \begin{align*}
            \textstyle\left|f_{{\tiny\mbox{DDFact}}}(x+t(u-x);\Upsilon) - f_{{\tiny\mbox{DDFact}}}(x;\Upsilon)-t g_x(x;\Upsilon)^\top (u-x)\right|< \frac{\epsilon}{3}.
        \end{align*}
        Note that $t< \delta_2$ is equivalent to that $ t\cdot \|u-x\| < \delta_3:= \delta_2\cdot \tilde{r}$.
        \item[(4)] for every $y\in \mathcal{N}_{\tilde{r}}(x)$, we have that $x+\frac{y-x}{\|y-x\|}\cdot \tilde{r}\in \mathcal{C}_{\tilde{r}}(x)$. By (2), there is some $u\in F$ such that
        \[
        \textstyle\left\|x+\frac{y-x}{\|y-x\|}\cdot \tilde{r}-u\right\| < \frac{\min\{\epsilon, \delta_1\}}{3\cdot \|g_x(x;\Upsilon)\|},
        \]
        and thus
        \begin{align*}
            \textstyle\left\|y-\left(x+ \frac{u-x}{\tilde{r}}\cdot \|y-x\|\right)\right\|
            =\textstyle\frac{\|y-x\|}{\tilde{r}} \cdot  \left\|x+\frac{y-x}{\|y-x\|}\cdot \tilde{r}-u\right\|
            < \textstyle\frac{\min\{\epsilon, \delta_1\}}{3\cdot \|g_x(x;\Upsilon)\|}.
        \end{align*}
    \end{itemize}

    From (1--4), given $\epsilon>0$, there exists $\delta_3 >0$ and a finite set $F\subset \mathcal{C}_{\tilde{r}}(x)$ such that for any $y\in \dom\left(f_{{\tiny\mbox{DDFact}}}; \Upsilon\right)_+$ and $\|y-x\|< \delta_3$, there exists some $u\in F$ such that
    \begin{align*}
         &\left|f_{{\tiny\mbox{DDFact}}}(y;\Upsilon) - f_{{\tiny\mbox{DDFact}}}(x;\Upsilon)- g_x(x;\Upsilon)^\top (y-x)\right|\\
         &\quad\le  \left|f_{{\tiny\mbox{DDFact}}}(y;\Upsilon) - f_{{\tiny\mbox{DDFact}}}(\hat y;\Upsilon)\right| \\
         &~~\qquad + \left|f_{{\tiny\mbox{DDFact}}}(\hat y
         ;\Upsilon) - f_{{\tiny\mbox{DDFact}}}(x;\Upsilon) - g_x(x;\Upsilon)^\top (\hat y-x)\right|\\
        &~~~~~~\qquad+\left| g_x(x;\Upsilon)^\top(y-\hat y)\right|\\
         &\quad < \frac{\epsilon}{3}+\frac{\epsilon}{3}+ \|g_x(x;\Upsilon)\| \,  \frac{\min\{\epsilon, \delta_1\}}{3  \|g_x(x;\Upsilon)\|}
         ~ < ~ \epsilon,
    \end{align*}
    where $\hat y = x+ \frac{u-x}{\tilde{r}}\, \|y-x\|$. Finally, the invariance of $g_x(x;\Upsilon)$ to $F, Q$ as long as we change $\beta$ accordingly follows from Proposition \ref{Factx:dirg}.
    \item[\ref{thm:fact}.iv:] For the first part, note that $F_{{\tiny\mbox{DDFact}}}(x;\Upsilon) = \sum_{i=1}^n \gamma_i x_i F_{i\cdot}^\top F_{i\cdot}$, and for every $x\in \dom\left(f_{{\tiny\mbox{DDFact}}}; \Upsilon\right)_+$ and $\Upsilon\in \mathbb{R}^n_{++}$, $F_{{\tiny\mbox{DDFact}}}(x;\Upsilon) \in \dom(\Gamma_s)$, and is thus well defined. On the other hand, by switching the value of $x$ and $\Upsilon$, we find that $F_{{\tiny\mbox{DDFact}}}(\Upsilon;x) = F_{{\tiny\mbox{DDFact}}}(x;\Upsilon)$. We can use the same method which we used to derive the generalized gradient with respect to $x$ to derive the generalized gradient with respect to $\Upsilon$. This means that $f_{{\tiny\mbox{DDFact}}}(x;\Upsilon)$ is generalized differentiable at $\Upsilon$ with generalized gradient
    \begin{align*}
    g_\Upsilon(x;\Upsilon)=&~x\circ\diag\left(F Q \Diag\left( \beta\right) Q ^\top F^\top\right) - \Diag(\Upsilon)^{-1}x,
    \end{align*}
    where $C=FF^\top$ is a factorization of $C$ and $(Q, \beta)$ are defined in Definition \ref{def:fact}. In particular, $g_\Upsilon(x;\Upsilon)$ is invariant to different choices of $F, Q$ and thus well defined. Moreover, because $\Upsilon \in \mathbb{R}^n_{++}$ lies in the interior of $\mathbb{R}^n_{++}$\,, the generalized differentiability reduces to differentiability. Moreover, the invariance of $g_\Upsilon(x;\Upsilon)$ to $F, Q$ as long as we change $\beta$ accordingly follows the same logic as Theorem \ref{thm:fact}.iii.
\smallskip

For the second part, note that $\left.g_\Upsilon(x^*;\Upsilon)\right|_{\Upsilon=\mathbf{e}} = 0$ is equivalent to  $x^*\circ \left(g^*(Q;\beta; \mathbf{e})-\mathbf{e}\right)=0$, where
$g^*(Q;\beta;\mathbf{e}) = \diag\left(F Q \Diag\left( \beta\right) Q ^\top F^\top\right)$
 as defined in \eqref{def:g}, specifically for $x^*$. This is further equivalent to $$
g^*_i(Q;\beta; \mathbf{e})=1, ~\forall x^*_i>0.$$
In the following proof, we will leverage the KKT conditions of \ref{DDFact} which we present here: for any optimal solution $x^*$ to \ref{DDFact}, there is some $\upsilon^* \in \mathbb{R}^n, \nu ^*\in \mathbb{R}^n, \pi^*\in \mathbb{R}^m$ such that
\begin{align*}
    \begin{array}{l}
 \mathbf{e}^\top x^*=s,~ Ax^*\leq b,~
 0\leq x^*\leq \mathbf{e},\\
\upsilon^*\geq 0, ~\nu^*\geq 0, ~\pi^*\geq 0,\\
    g^*(Q, \beta; \Upsilon) + \upsilon^* - \nu^*  - A^\top \pi^* - \tau^*\mathbf{e}=0,\\
    \pi^*\circ \left(b-Ax^*\right)=0, ~ \upsilon^*\circ x^*=0, ~ \nu^*\circ \left(\mathbf{e}-x^*\right)=0,
\end{array}\tag{DDFact-KKT}\label{DDFact:KKT}
\end{align*}
where $g^*(Q, \beta; \Upsilon)=\Upsilon\circ \diag\left(F Q \Diag\left( \beta\right) Q ^\top F^\top\right)$ as defined in \eqref{def:g}, specifically for $x^*$.
The existence of $\upsilon^* \in \mathbb{R}^n, \nu ^*\in \mathbb{R}^n, \tau^*\in \mathbb{R}, \pi^*\in \mathbb{R}^m$ is due to that: (1) \ref{DDFact} is a generalized differentiable convex-optimization problem; (2) Slater's condition holds because of the affine constraints describing the feasible region of \ref{DDFact}.
\smallskip

By \cite[Section 3.1]{li2020best}, when $\Upsilon =\mathbf{e}$ and there are not linear contraints $Ax\le b$, then there is a closed-form solution $(\upsilon^*, \nu^*, \tau^*)$ to \ref{DDFact:KKT} given $x^*$. Suppose that $\sigma$ is a permutation of $\{1,2, \cdots, n\}$ such that
\[
\left(g^*(Q;\beta; \mathbf{e})\right)_{\sigma(1)}\ge \left(g^*(Q;\beta; \mathbf{e})\right)_{\sigma(2)} \ge \cdots \ge \left(g^*(Q;\beta; \mathbf{e})\right)_{\sigma(n)}\,,
\]
where $\left(g^*(Q;\beta; \mathbf{e})\right)_i$ denotes the $i^{\text{th}}$ element of $g^*(Q;\beta; \mathbf{e})$. Then
\begin{align*}
    &\tau^* = \left(g^*(Q;\beta; \mathbf{e})\right)_{\sigma(s)},\\
    &\nu^*_{\sigma(i)}= \begin{cases}
    \left(g^*(Q;\beta; \mathbf{e})\right)_{\sigma(i)}-\tau^*, & \forall~ 1\le i \le s; \\
    0, & \forall~ s+1\le i \le n,
    \end{cases}\\
    & \upsilon^* = \nu^* + \tau^* \mathbf{e} - g^*(Q;\beta; \mathbf{e}).
\end{align*}

We claim that
\[\textstyle
\sum_{i\in \{1,2,\ldots,n\}}
x^*_{\sigma(i)} \left(g^*(Q;\beta; \mathbf{e})\right)_{\sigma(i)}= \sum_{i\in  \{1,2,\ldots,s\}}
\left(g^*(Q;\beta; \mathbf{e})\right)_{\sigma(i)}=s.
\]
In fact, by \ref{DDFact:KKT}, we have
\begin{align*}
    0= &~x^*\circ \left(g^*(Q;\beta; \mathbf{e})+\upsilon^*-\nu^* -\tau^*\mathbf{e}\right)\\
    =& ~x^* \circ g^*(Q;\beta; \mathbf{e}) + x^* \circ \upsilon^* - x^* \circ \nu^* - \tau^* x^*\\
    =& ~x^* \circ g^*(Q;\beta; \mathbf{e})  - x^* \circ \nu^* - \tau^* x^*\\
 = & ~x^* \circ g^*(Q;\beta; \mathbf{e})  - \nu^* - \tau^* x^* ,
\end{align*}
and further
\begin{align*}
    0= &~\mathbf{e}^\top \left(x^* \circ g^*(Q;\beta; \mathbf{e})  - \nu^* - \tau^* x^*\right)\\
    =&~\textstyle\sum_{i\in \{1,2,\ldots,n\}}
x^*_{\sigma(i)} \left(g^*(Q;\beta; \mathbf{e})\right)_{\sigma(i)} + \sum_{i\in \{1,2,\ldots,s\}} \nu^*_{\sigma(i)} + \tau^* s\\
 =&~\textstyle\sum_{i\in \{1,2,\ldots,n\}}
x^*_{\sigma(i)} \left(g^*(Q;\beta; \mathbf{e})\right)_{\sigma(i)} - \sum_{i\in \{1,2,\ldots,s\}} \left(g^*(Q;\beta; \mathbf{e})\right)_{\sigma(i)}.
\end{align*}
On the other hand, by \cite{chen2023computing}, the duality gap of \ref{DDFact} is $\mathbf{e}^\top \nu^* + \tau^* s -s=0$ and thus
\begin{align}
    \sum_{i\in \{1,2,\ldots,n\}}
x^*_{\sigma(i)} \left(g^*(Q;\beta; \mathbf{e})\right)_{\sigma(i)}= \sum_{i\in  \{1,2,\ldots,s\}}
\left(g^*(Q;\beta; \mathbf{e})\right)_{\sigma(i)}=s.\label{thm:Fact-4-eqn1}
\end{align}
Furthermore, we claim that if $x_{\sigma(i)}^*=1$, then $g_{\sigma(i)}(x^*)\le 1$. Note that by the proof of Proposition \ref{Factx:dirg}, letting  $q_j$ be the $j^{\text{th}}$ column of $Q$, we have
\begin{align}
    \left(g(Q;\beta; \mathbf{e})\right)_{\sigma(i)} = &~F_{\sigma(i)\cdot} Q \Diag\left( \beta\right) Q ^\top F_{\sigma(i)\cdot}^{\top}\nonumber\\
    =&  ~\textstyle\sum_{j=1}^r \beta_j F_{\sigma(i)\cdot} q_j q_j^\top F_{\sigma(i)\cdot}\nonumber\\
    \le & ~\textstyle\sum_{j=1}^r \frac{1}{\lambda_j} F_{\sigma(i)\cdot} q_j q_j^\top F_{\sigma(i)\cdot} \nonumber\\
    = & ~\textstyle F_{\sigma(i)\cdot} \left(F_{{\tiny\mbox{DDFact}}}(x;\mathbf{e})\right)^{\dagger} F_{\sigma(i)\cdot}^{\top}\nonumber\\
    = &~ \textstyle  F_{\sigma(i)\cdot} \left( F_{\sigma(i)\cdot}^\top F_{\sigma(i)\cdot}+ \sum_{j\neq \sigma(i)} x^*_j F_{j\cdot}^\top F_{j\cdot} \right)^{\dagger} F_{\sigma(i)\cdot}^{\top}\nonumber\\
   \le & ~F_{\sigma(i)\cdot} \left( F_{\sigma(i)\cdot}^\top F_{\sigma(i)\cdot}\right)^{\dagger} F_{\sigma(i)\cdot}^{\top}
    =1,\label{thm:Fact-4-eqn2}
\end{align}
where the first inequality is due to Lemma \ref{Nik} and Definition \ref{def:fact}, and the second inequality is due to the  Sherman–Morrison formula for the Moore-Penrose inverse. \eqref{thm:Fact-4-eqn1} and \eqref{thm:Fact-4-eqn2} together imply that
\begin{align*}
    \left(g(Q;\beta; \mathbf{e})\right)_{\sigma(1)}=\cdots = \left(g(Q;\beta; \mathbf{e})\right)_{\sigma(s)}=1.
\end{align*}
Moreover, for $i> s$ such that $x^*_{\sigma(i)}>0$, we must have $\left(g(Q;\beta; \mathbf{e})\right)_{\sigma(i)} = \left(g(Q;\beta; \mathbf{e})\right)_{\sigma(s)}$, otherwise we contradict \eqref{thm:Fact-4-eqn1}, due to the non-increasingness of $\left(g(Q;\beta; \mathbf{e})\right)_{\sigma(i)}$ in $i$.

\item[\ref{thm:fact}.v:] By the generalized gradients characterized for $x$ and gradients characterized for $\psi$ in Theorem \ref{thm:fact}.iii,iv, we only need to prove the continuity of $g(Q;\beta;\Upsilon)$ as defined in \eqref{def:g} with respect to $(x,\Upsilon)$. Because of the invariance of $g(Q;\beta;\Upsilon)$ to $F, Q$ as long as we change $\beta$ accordingly, we can fix $F, Q$; then the conclusion follows from the continuity of eigenvalues in the matrix elements. $\hfill \square$

\end{itemize}
\end{proof}

\section{Algorithms}
In this section, we discuss our algorithms  for determining optimal g-scaling vectors for \ref{BQP} and \ref{linx}, as well as for the selection of \emph{good} g-scaling vectors for \ref{DDFact}.
For \ref{DDFact}, we can only aim for \emph{good}, because of the lack of a convexity result concerning the g-scaling vector for \ref{DDFact}; despite this, results presented in \S \ref{sec:num} demonstrate that, in many cases, the \ref{DDFact} bounds computed with the best g-scaling vectors obtained are the strongest that we have, demonstrating the effectiveness of such algorithms.
% Initially, we establish a sufficient condition for the differentiability of each upper bound ($z_{{\tiny\mbox{BQP}}} (\Upsilon)$, $z_{{\tiny\mbox{linx}}} (\Upsilon)$, or $z_{{\tiny\mbox{DDFact}}} (\Upsilon)$) with respect to $\Upsilon\in \mathbb{R}^n_{++}$.

For notational generality, we consider an upper-bound form for \ref{CMESP}, which encompasses \ref{BQP}, \ref{linx}, and \ref{DDFact} as particular instantiations. Specifically, we define a general upper bound for \ref{CMESP} of the form:
\begin{align}\tag{CMESP-UB} \label{CMESP-UB}
    \begin{array}{rl}
         z(\psi)~:=~ \max ~& ~f(x; \psi) \\
         \text{s.t.} ~& ~ g_i(x) \le 0, ~ \forall i = 1, 2, \cdots, m_1\,;\\
         &~ h_j(x) = 0, ~ \forall j = 1, 2, \cdots, m_2\,,\\
    \end{array}
\end{align}
where  $f: \dom(f) \rightarrow \mathbb{R}$,
$g_i:\dom(g_i) \rightarrow \mathbb{R}$, and $h_j:\dom(h_j)\rightarrow \mathbb{R}$ (with the data  $C,s, A, b$ being absorbed into these functions). We assume that $T$, the set of possible values for the parameter vector $\psi$ to be open. We also assume that $f(x;\psi)$ is convex in $x$ for each $\psi$ and \textit{continuously generalized differentiable} in $x$ and \textit{continuous differentiable} in $\psi$ on its domain. Finally, we assume that \ref{CMESP-UB} is a convex program and that its maximum is attained on the feasible set.
% We assume that the feasible set of \ref{CMESP-UB} is a compact subset of $\mathbb{R}^n$, denoted as $M$. We further assume $M$ to be convex so that \ref{CMESP-UB}  is a convex program.
% We also assume $z(\psi)$ is finite and thus attainable by the compactness of $M$.
%  \jon{We haven't addressed the technical point that the feasible region can contain points that are not in the domain of $f$. So the argument that the max is attained is trickier.} \zc{Can we assume the maximum is attainable or assume that $f(x;\psi)$ goes to infinity when approaching the boundary of the domain?}
We let $M$ denote the feasible set of \ref{CMESP-UB}, we let $M^*(\psi):=\{x\in M : f(x;\psi)= z(\psi)\}$ (the optimal $x$ given $\psi$), and we say that $\psi^*$ is optimal if $z(\psi^*) = \min_{\psi\in T} z(\psi)$.

\begin{remark}
    \ref{linx} and \ref{DDFact} can naturally be viewed as an instantiation of \ref{CMESP-UB} with $\psi:=\log \Upsilon$. For \ref{BQP}, we can view $X\in \mathbb{S}^n$ as a vector in $\mathbb{R}^{n(n+1)/2}$, therefore it can also be regarded as an instantiation of \ref{CMESP-UB} with $\psi:=\log \Upsilon$. The continuous generalized differentiability and continuous differentiability of the objective functions is established in the proofs of Theorem \ref{thm:bqp}, \ref{thm:linx}, and \ref{thm:fact}.
\end{remark}

We first focus on the cases where $f(x;\psi)$ is convex in $\psi$, which encompasses \ref{BQP} and \ref{linx} as particular cases. For such cases, $z(\psi)$ becomes a convex function in $\psi$. Our algorithm relies on the following theorem, tailored from \cite[Theorem 2.4.18]{zalinescu2002convex} to our specific context.
\begin{thm}{\cite[Theorem 2.4.18]{zalinescu2002convex}}\label{alg:cvxdifferential}
    Assume that $f(x;\psi)$ is convex in $\psi$ for every $x\in M$, then the subdifferential of $z(\psi)$ at   $\psi\in T$ is
    \begin{align*}
        \partial z(\psi) = \overline{\text{conv}} \left\{\frac{f(x;\psi)}{\partial\psi}: x\in M^*(\psi)\right\},
    \end{align*}
    where $\overline{\text{conv}}$ denotes the convex  closure. Furthermore, if $M^*(\psi)$ is a singleton, then the unique subgradient becomes the gradient of $z(\psi)$ at $\psi$.
\end{thm}
\begin{remark}
    \cite[Propositions 3.3.7 and 3.6.9]{FLbook} provide sufficient conditions for $M^*(\psi)$ to be a singleton for \ref{BQP} and \ref{linx}, respectively.
\end{remark}

Theorem \ref{alg:cvxdifferential} allows the calculation of the subgradient (or gradient) of $z(\psi)$ by solving \ref{CMESP-UB}. Thus, a standard subgradient algorithm can achieve convergence to an optimal $\psi^*$. However, due to the well-known sluggishness of the subgradient algorithm, we employ a BFGS-type algorithm that utilizes the subgradient (or gradient) to update the Hessian approximation.
% Assuming the smoothness, this algorithm exhibits superlinear convergence to the optimal $\psi^*$.

For cases where $f(x;\psi)$ is not necessarily convex in $\psi$,
% it is usually impossible to find an algorithm that will converge to the optimal $\psi^*$ because $z(\psi)$ is not necessarily convex in $\psi$.
we cannot aim for verified global optimality.
Nevertheless, we still use a BFGS-type algorithm, where we use $ \frac{\partial f(x;\psi)}{\partial\psi}$ for any $x\in M^*(\psi)$ to update the Hessian approximation. Under some smoothness assumption, $\frac{\partial f(x;\psi)}{\partial\psi}$ becomes the differential of $z(\psi)$, and this algorithm will converge to a stable point of $z(\psi)$. The following theorem provides a sufficient condition for the differentiability of $z(\psi)$, tailored from \cite[Proposition 2.1]{oyama2018non} to our specific context.
\begin{thm}{\cite[Proposition 2.1]{oyama2018non}}\label{alg:thm}
     We define a \emph{selection} to be a function mapping from $\psi$ to $x$ selected from $M^*(\psi)$, denoted as $x^*(\psi)$. Given $\bar \psi \in T$, if there is a selection $x^*(\psi)$ continuous at $\bar \psi$, then $z(\psi)$ is differentiable at $\bar \psi$ with
    \begin{align*}
        \frac{\partial z(\bar \psi)}{\partial \psi} = \frac{\partial f(x^*(\bar\psi);\bar \psi)}{\partial \psi}.
    \end{align*}
    In particular, if $M^*(\bar \psi)$ is a singleton, then the unique selection $x^*(\psi)$ is always continuous at $\bar \psi$.
\end{thm}

Additionally, BFGS has been shown to possess good convergence properties under non-smooth settings, e.g., locally Lipschitz and directionally differentiable (see \cite{nonsmoothBFGS}). The following theorem guarantees this property for $z_{{\tiny\mbox{DDFact}}} (\Upsilon)$, tailored from \cite[Theorem 4.1]{fiacco1990sensitivity} to our specific context.
\begin{thm}\label{alg:directdiff}
For any $\psi\in T$, $z(\psi)$ is locally Lipschitz near $\psi$ and directionlly differentiable at $\psi$ in any feasible direction $v$ with formula
\begin{align*}
    \partial z(\psi; v) = \max_{x\in M^*(\psi)} \left(\frac{\partial f(x;\alpha)}{\partial \alpha}\right)^\top v.
\end{align*}
\end{thm}

\begin{remark}
    Theorem \ref{alg:cvxdifferential}, \ref{alg:thm}, and \ref{alg:directdiff} hold based on the continuous differentiability of $f(x; \psi)$ in $\psi$ and the continuity of $f(x; \psi)$ in $x$ for \ref{CMESP-UB}. The continuously generalized differentiability of $f(x; \psi)$ in $x$ ensures good convergence behavior of  algorithms for obtaining $x^*\in M^*(\psi)$.
\end{remark}
% \jon{For the three theorems of this section, can we link these results to what you established about generalized differentiability for the objective of \ref{DDFact}? That is, does the generalized differentiability give us something here (in the context of applying BFGS for $\Upsilon$ optimization for  \ref{DDFact}) that
% we wouldn't get without it?} \zc{The above three theorems hold only when the objective of \ref{CMESP-UB} is differentiable or generalized differentiable in $(x; \psi)$. I did not write this condition explicitly in the theorems as I have assumed that for \ref{CMESP-UB}. Should I specify this condition explicitly in the theorems? \jon{Anything that you assumed for \ref{CMESP-UB} does not have to be repeated. But it is very important that we make clear (in the statements) where we employ generalized differentiability,}}

% \zc{Please have a look at the above context}
% \jon{we need a remark tying the conditions of  Theorem \ref{alg:thm} back to what we can say about smoothness of \ref{DDFact} and what that means for BFGS}

\section{Numerical results}\label{sec:num}

We experimented on benchmark instances of \hyperlink{MESP}{MESP}, using three covariance matrices
 that have been extensively used in the literature, with $n=63,90,124$ (see, e.g., \cite{KLQ,LeeConstrained,AFLW_Using,Anstreicher_BQP_entropy,Kurt_linx}).
 For testing \ref{CMESP}, we included five side constraints $a_i^\top x\leq b_i$, for $i=1,\ldots,5$, in \hyperlink{MESP}{MESP}. As there is no benchmark data for the side constraints, we have generated them randomly.  For each $n$, the  left-hand side of constraint $i$ is given by a uniformly-distributed random vector $a_i$ with integer components between $-2$ and $2$. The right-hand side of the constraints was selected so that, for
every $s$ considered in the experiment, the best known solution
%$x^*(s)$
of the instance of \hyperlink{MESP}{MESP} is violated by
at least one constraint.

% For that, each $b_i$ was selected as the $80$-th percentile of the values $a_i^\top x^*(s)-1$, for all $s$.

% \zc{Here is how I generate $b$:
% \begin{enumerate}
%     \item First, I will generate a matrix A of 5 by n with elemented randomly selected from [-2, 2], integer.
%     \item Second, I choose a range of s with middle values for each n, e.g., when n=124, s in [25, 99]. Then I will generate a constraint-free feasible solution x(s) with the heuristic for each s and in this range and calculate A*x(s). I will divide the s-range into five disjunctive subsets and each subset contains a continuous range of s. For each subset i, we calculate $b_i = \min (a'_i\cdot x(s)-1)$ and use $b_i$ as the right-hand-side of one inequality. The choice of s-range is such that constraints generated are feasible but right for most s, typically larger than the chosen s-range.
% \end{enumerate}}

For each $n$ (which refers always to a particular benchmark covariance matrix), we consider different values of $s$ defining a set of test instances of \hyperlink{MESP}{MESP} and \ref{CMESP}.
%$n=63,90,124$,
% we considered a set of test instances of \hyperlink{MESP}{MESP} and \ref{CMESP} with
% a wide range of $s$.
% $s\in[2,n-1]$, and all instances of \ref{CMESP} with $s$ in the ranges $[3,52]$, $[4,87]$, and $[11,110]$,
%  respectively.
We ran our experiments under Windows, on an Intel Xeon E5-2667 v4 @ 3.20 GHz processor equipped with 8 physical cores (16 virtual cores) and 128 GB of RAM.
We implemented our code in \texttt{Matlab} using the solvers
 \texttt{SDPT3} v. 4.0 for \ref{BQP}, and \texttt{Knitro} v. 12.4
for \ref{linx} and  \ref{DDFact}. {When instantiating the \ref{DDFact} bound, the selection of $F$ is made
as $F := U\Lambda^{1/2}$, where $C=U\Lambda U^\top$ represents a spectral decomposition of $C$ omitting eigenvalues of zero,  so that $U \in \mathbb{R}^{n\times \rank(C)}$ and diagonal matrix $\Lambda \in \mathbb{R}^{\rank(C)\times \rank(C)}$. This choice  gives the number of columns of $F$ equal to the rank of $C$, so that  Remark \ref{rem:fullrank} applies.

In all of our experiments, we
compare g-scaling only with \emph{optimal} o-scaling for
the linx and BQP bounds and no scaling for the factorization bound
(which is invariant under o-scaling). We never work with unscaled linx and BQP bounds,
as these bounds are not generally useful; in branch-and-bound
approaches (see \cite{Kurt_linx}, for example), a good scaling parameter is used on every subproblem.

We  optimized scaling vectors $\Upsilon$ using a BFGS algorithm, and o-scaling parameters $\gamma$ using Newton's method.  In all of our experiments we set \texttt{Knitro} parameters\footnote{see \url{https://www.artelys.com/docs/knitro/2_userGuide.html}, for details} as follows:
\verb;convex;~$=1$ (true),
\verb;gradopt;~$=1$ (we provided exact gradients),
\verb;maxit;~$=1000$.  We set
\verb;opttol;~$=10^{-10}$, aiming to satisfy the  KKT optimality conditions to a very tight tolerance. We set
 \verb;xtol;~$=10^{-15}$ (relative tolerance for lack of progress in the solution point) and
 \verb;feastol;~$=10^{-10}$ (relative tolerance for the feasibility error), aiming for
 the best solutions that we could reasonably find, to be sure that we are seeing the best possible that can be achieved.
 These are not meant to be practical settings for performance, so we do not report running times.
 In our first set of experiments, we set the
 the \texttt{Knitro} parameter
 \verb;algorithm;~$=3$ to use an active-set method. Besides solving the relaxations to get upper bounds for our test instances of
\hyperlink{MESP}{MESP} and \ref{CMESP}, we  compute lower bounds with a heuristic of \cite[Section 4]{LeeConstrained} and then a local search (see \cite[Section 4]{KLQ}).

%\subsection{Experiments with MESP}

In Figure \ref{fig:unconst}, we show the impact of  g-scaling on the  \ref{linx} bound for \hyperlink{MESP}{MESP} on the three benchmark covariance matrices. For  the $n=63$ matrix, we also show the impact of  g-scaling on the \ref{BQP} and complementary \ref{BQP} bounds (recall that the \ref{linx} bound is invariant under complementation). The \ref{DDFact} and complementary \ref{DDFact} bounds are only considered in the experiments for \ref{CMESP}, as the g-scaling methodology was only able to improve these bounds when side constraints were added to \hyperlink{MESP}{MESP}.
The plots on the left in Figure \ref{fig:unconst} present the ``integrality gap decrease ratios'', given by the difference between the integrality gaps using o-scaling and the integrality gaps using g-scaling, divided by the  integrality gaps using o-scaling.
%obtained when we compute the bounds using the g-scaling instead of the o-scaling.
The integrality gaps are given by the difference between the upper bounds computed with the relaxations and
     lower bounds given by heuristic solutions. We see that larger  $n$ leads to larger  maximum ratios. We also see that the g-scaling methodology is effective in reducing all bounds evaluated, especially the linx bound. Even for the most difficult instances, with intermediate values of $s$, we have some improvement on the bounds, which can be  effective in the branch-and-bound context where the bounds would ultimately be applied. The plots on the right in Figure \ref{fig:unconst} present the integrality gaps, and we see that even when the integrality gaps given by the o-scaling are less than 1,  g-scaling can reduce them.

\begin{figure*}
        \centering
        \begin{subfigure}[b]{0.496\textwidth}
            \centering
            \includegraphics[height=5.5cm,width=1.05\textwidth]{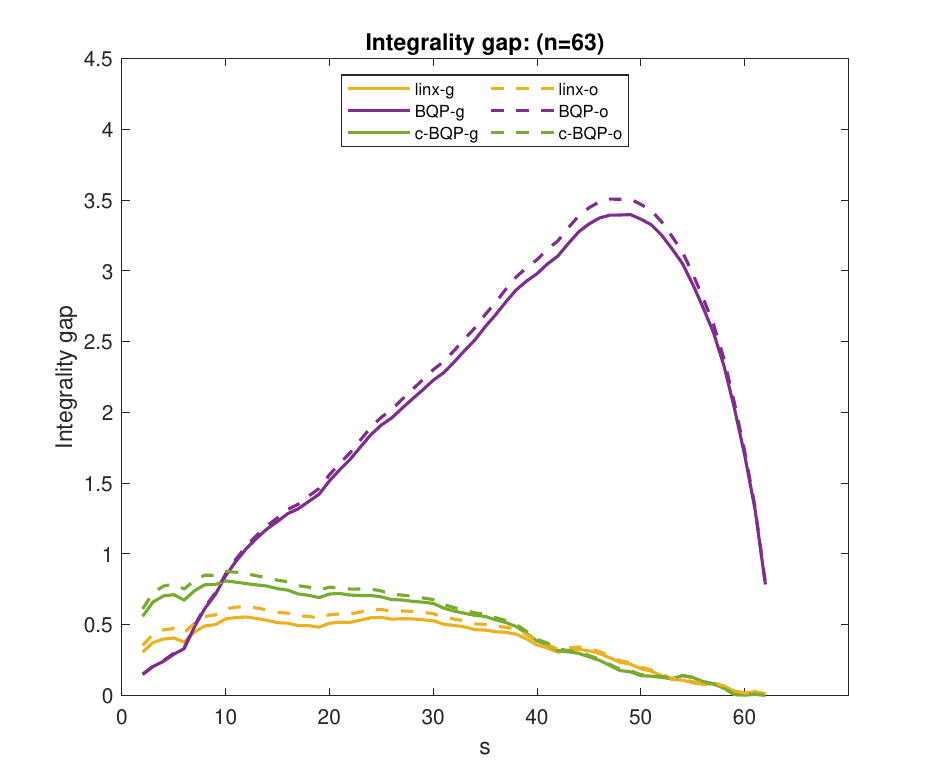}
            %\caption[]%{{\small Network 2}}
            %\label{}
        \end{subfigure}
        \hfill
        \begin{subfigure}[b]{0.496\textwidth}
            \centering
            \includegraphics[height=5.5cm,width=1.05\textwidth]{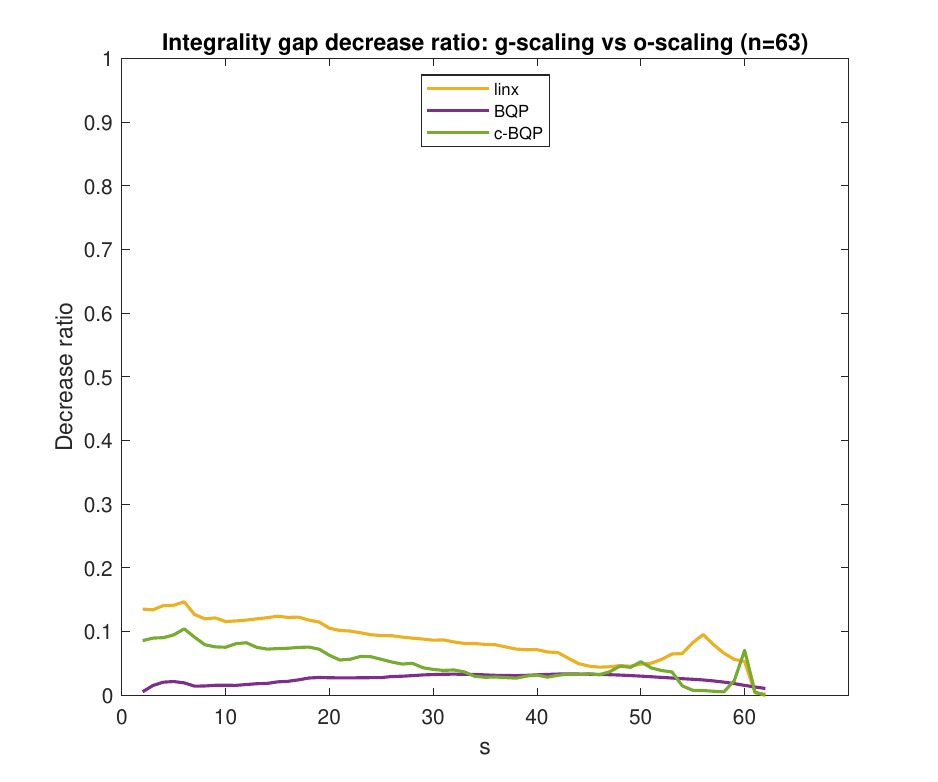}
            %\caption[]%{{\small Network 1}}
            %\label{}
        \end{subfigure}
        \vskip\baselineskip
        \begin{subfigure}[b]{0.496\textwidth}
            \centering
            \includegraphics[height=5.5cm,width=1.05\textwidth]{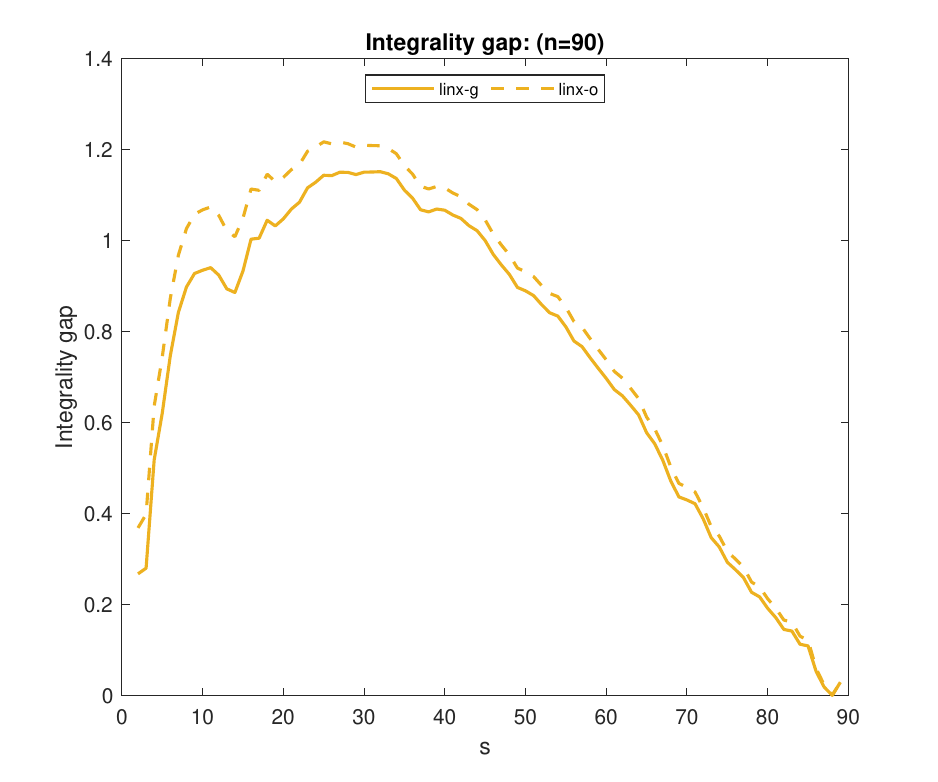}
            %\caption[]%{{\small Network 4}}
            %\label{}
        \end{subfigure}
        \hfill
        \begin{subfigure}[b]{0.496\textwidth}
            \centering
            \includegraphics[height=5.5cm,width=1.05\textwidth]{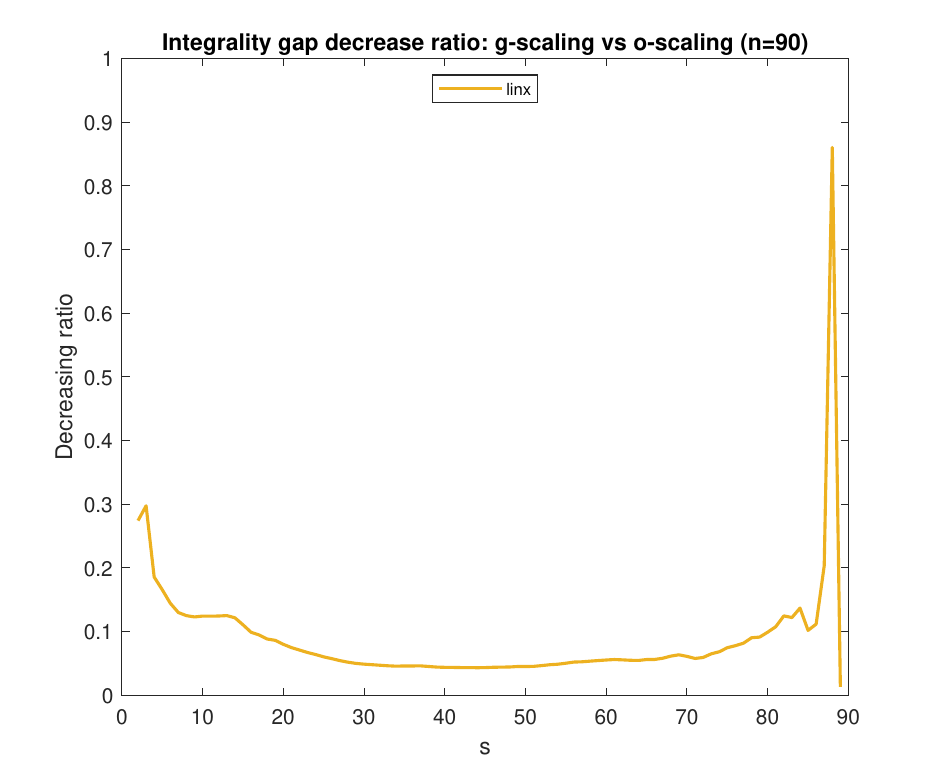}
            %\caption[]%{{\small Network 3}}
            %\label{}
        \end{subfigure}

        \vskip\baselineskip
         \begin{subfigure}[b]{0.496\textwidth}
            \centering
            \includegraphics[height=5.5cm,width=1.05\textwidth]{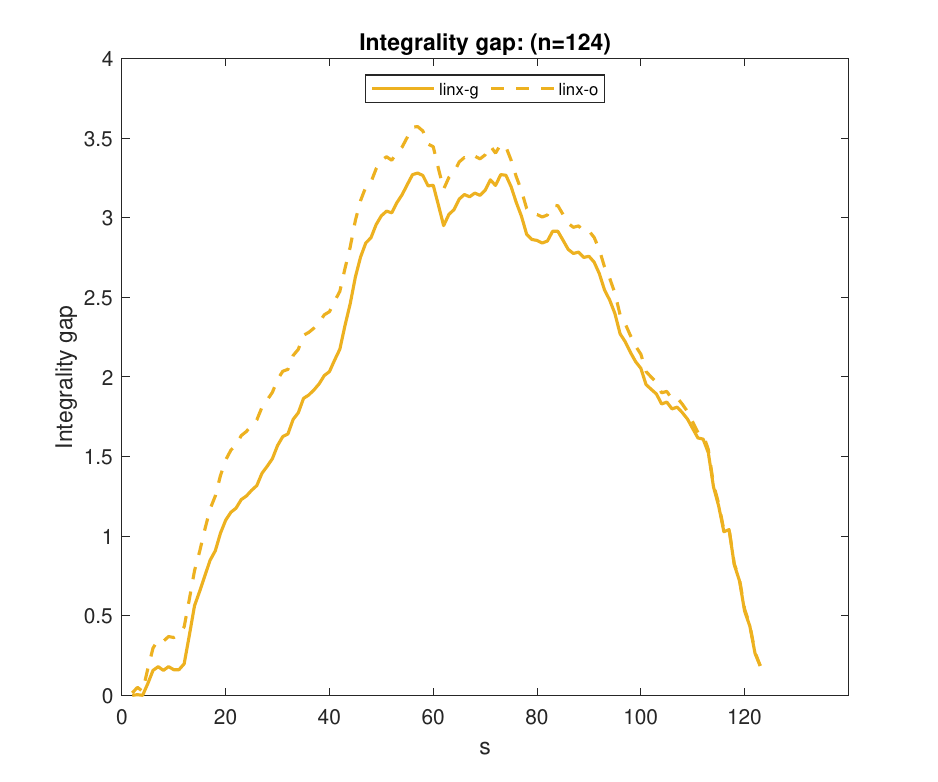}
            %\caption[]%{{\small Network 4}}
            %\label{}
        \end{subfigure}
        \hfill
        \begin{subfigure}[b]{0.496\textwidth}
            \centering
            \includegraphics[height=5.5cm,width=1.05\textwidth]{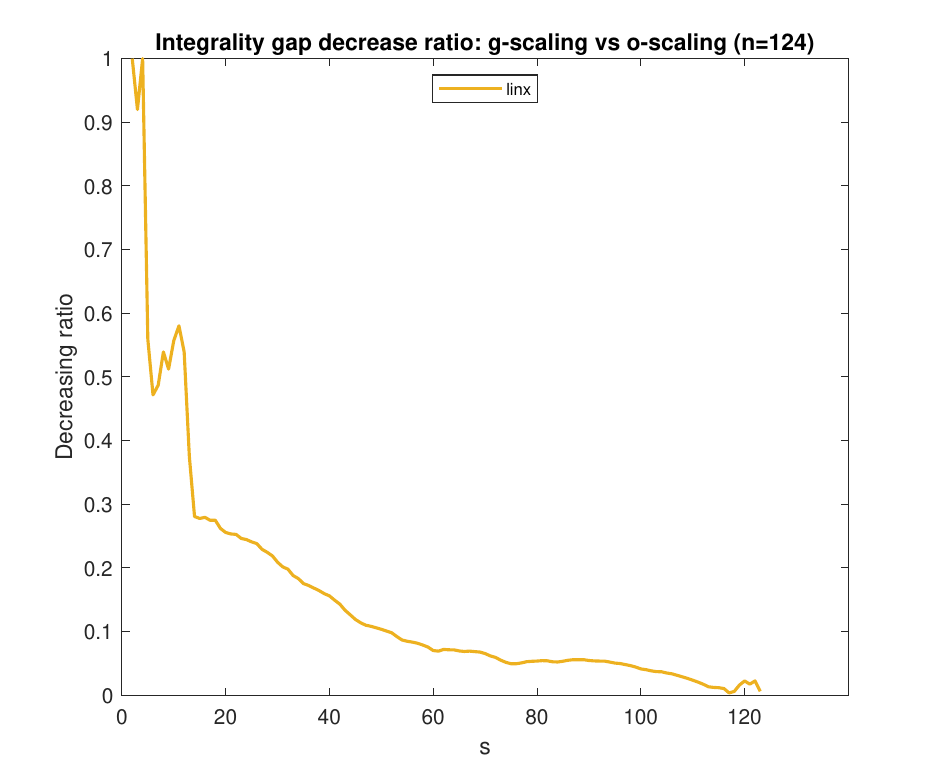}
            %\caption[]%{{\small Network 3}}
            %\label{}
        \end{subfigure}

%\vskip\baselineskip
        % \begin{subfigure}[b]{0.496\textwidth}
        %     \centering
        %     \includegraphics[height=0.5cm,width=1cm]{leg_left.png}
        %     %\caption[]%{{\small Network 3}}
        %     %\label{}
        % \end{subfigure}
        % \hfill
        % \begin{subfigure}[b]{0.496\textwidth}
        %     \centering
        %     \includegraphics[height=0.5cm,width=2.5cm]{leg_right.png}
        %     %\caption[]%{{\small Network 4}}
        %     %\label{}
        % \end{subfigure}

        \caption[ ]
        {\small Comparison between g-scaling and o-scaling for \hyperlink{MESP}{MESP}}
        \label{fig:unconst}
    \end{figure*}

 %%
% \begin{figure}[htbp]
% 	\centering

% 	\scalebox{0.55}{
% 		\includegraphics{data63uratio.pdf}}
% 	\scalebox{0.55}{
% 	\includegraphics{data90uratio.pdf}
% 	}
% 	\scalebox{0.55}{
% 	\includegraphics{data124uratio.pdf}
% 	}
% 	\caption{Comparison between g-scaling and o-scaling for MESP}
% 	\label{fig:unconst}
% \end{figure}

%\subsection{Experiments with CMESP}

In Figures  \ref{fig:intgap} and \ref{fig:ratio}, we show for \ref{CMESP}, similar results  to the ones shown in Figure \ref{fig:unconst}, except that now we also present the effect of g-scaling on  the \ref{DDFact} and the complementary \ref{DDFact} bounds. We see from the integrality gap decrease ratios that when side constraints are added to  \hyperlink{MESP}{MESP}, the g-scaling is, in general, more effective in reducing the gaps given by o-scaling. We also see that,   it is particularly  effective in reducing the \ref{DDFact} and complementary \ref{DDFact} bounds that were not improved by o-scaling. Especially for the $n=124$ matrix, we see a significant reduction on the gaps given by complementary \ref{DDFact} and \ref{DDFact}, for $s$ smaller and greater than  $50$, respectively.

    \begin{figure*}
            \includegraphics[width=0.99\textwidth]{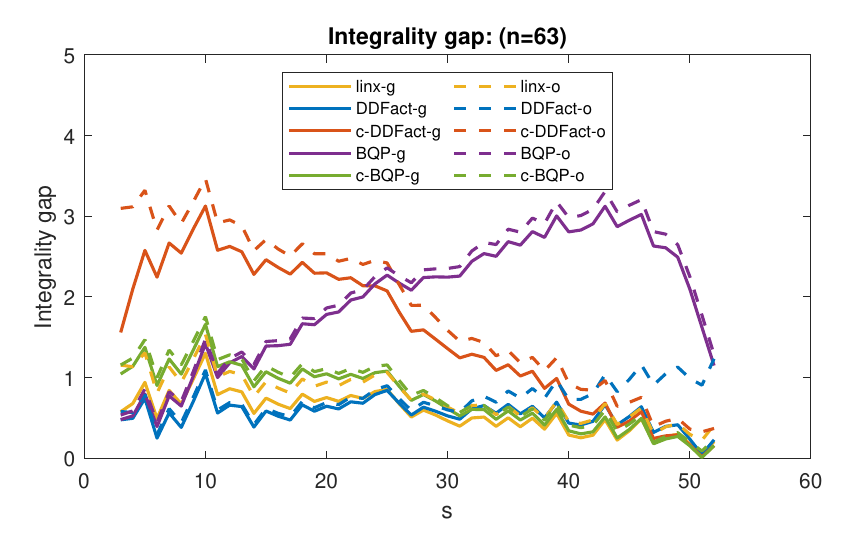}
            \includegraphics[width=0.99\textwidth]{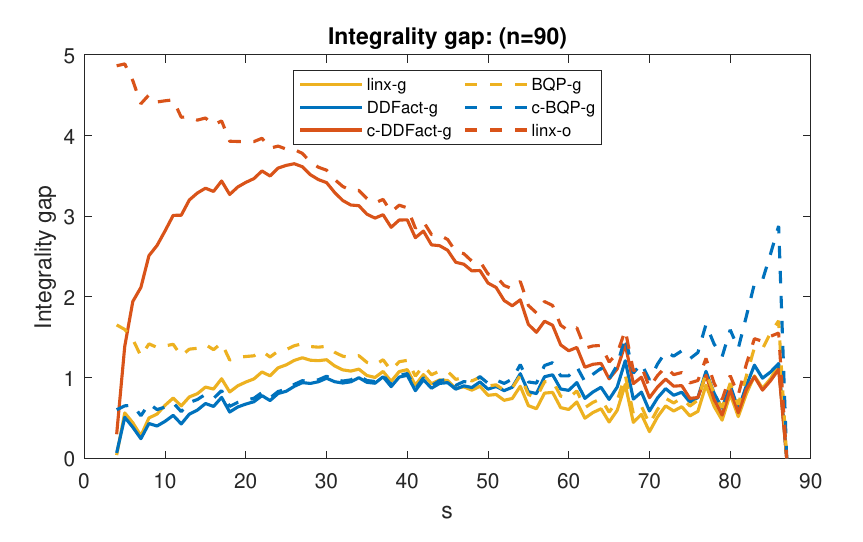}
            \includegraphics[width=0.99\textwidth]{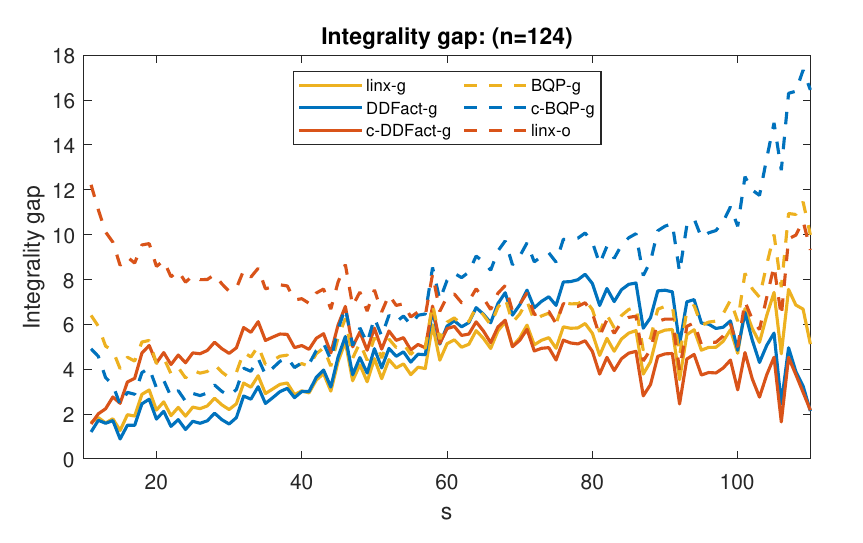}
        \caption[ ]
        {\small Comparison between g-scaling and o-scaling for \ref{CMESP}}
        \label{fig:intgap}
    \end{figure*}
%%%%%%
%%%%%%
\begin{figure*}
            \includegraphics[width=0.99\textwidth]{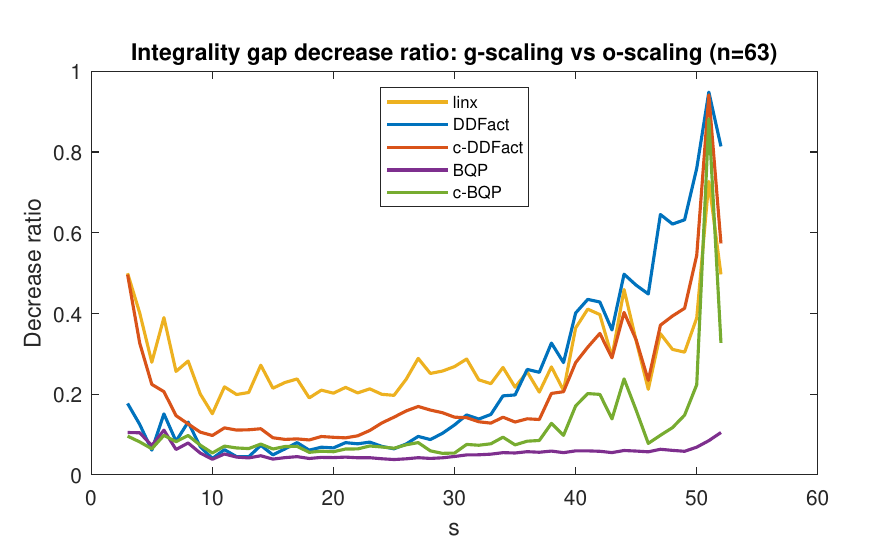}           \includegraphics[width=0.99\textwidth]{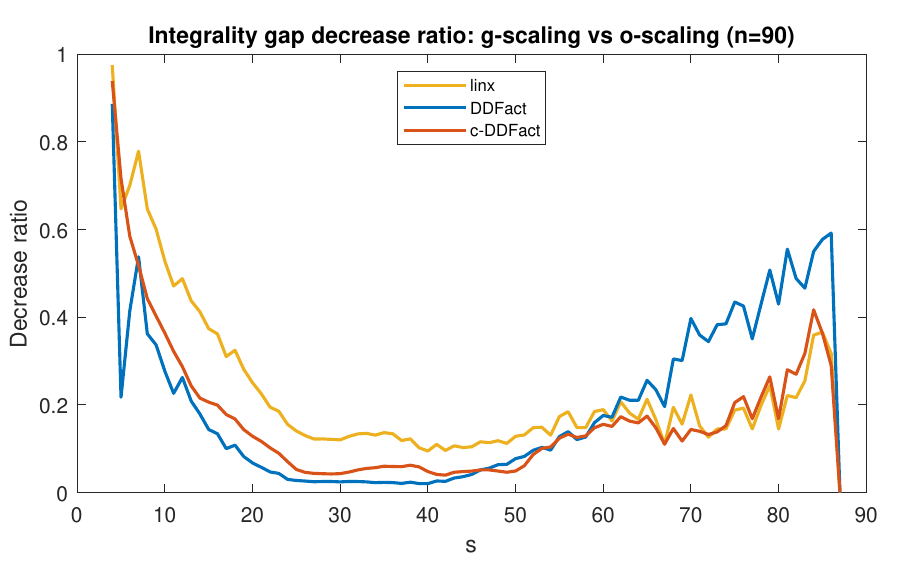}           \includegraphics[width=0.99\textwidth]{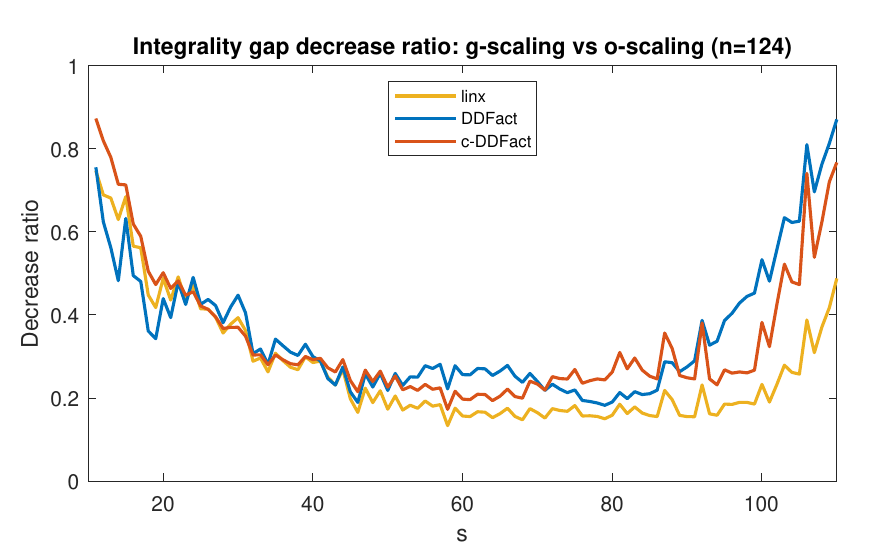}
        \caption[ ]
        {\small Comparison between g-scaling and o-scaling for \ref{CMESP}}
        \label{fig:ratio}
    \end{figure*}

% \begin{figure}[htbp]
% 	\centering

% 	\scalebox{0.55}{
% 		\includegraphics{data63ratio.pdf}}
% 	\scalebox{0.55}{
% 	\includegraphics{data90ratio.pdf}
% 	}
% 	\scalebox{0.55}{
% 	\includegraphics{data124ratio.pdf}
% 	}
% 	\caption{Comparison between g-scaling and o-scaling for CMESP}
% 	\label{fig:const}
% \end{figure}

We also investigated how the improvement of g-scaling  over o-scaling for the linx bound can increase the possibility of fixing variables in \hyperlink{MESP}{MESP} and \ref{CMESP}. The methodology for  fixing variables is based on convex duality and  has been applied  since the first convex relaxation  was proposed for these problems in \cite{AFLW_IPCO}. When a lower bound for each instance
%problem
is available, the dual solution of the  relaxation can potentially be used to fix variables at 0/1 values (see \cite{FLbook}, for example).  This is an important feature in the B{\&}B context. The methodology may be able to fix a number of variables when the relaxation generates a strong bound, and in doing so, it reduces the size of the  successive subproblems and improves the bounds computed for them.

In Table \ref{tab:fix}, for \hyperlink{MESP}{MESP}, we consider  (unscaled) \ref{DDFact} and (unscaled)  complementary \ref{DDFact}, and we show the impact of using g-scaled \ref{linx}, compared to o-scaled \ref{linx}, on an iterative procedure where we solve \ref{linx}, \ref{DDFact}, and complementary \ref{DDFact}, fixing variables at 0/1 whenever possible.  While for \ref{CMESP}, we show the impact of using g-scaled linx, g-scaled \ref{DDFact},  and g-scaled complementary \ref{DDFact}, compared to o-scaled linx,  (unscaled) \ref{DDFact}, and  (unscaled) complementary \ref{DDFact}, on the same iterative procedure where we solve \ref{linx}, \ref{DDFact}, and complementary \ref{DDFact}, fixing variables at 0/1 whenever possible.
% \mf{do we use g-scaling for DDFact and comp-DDFact when considering CMESP? i think we should add a sentence here clarifying what we do in terms of scaling for these two relaxations, for MESP and CMESP} \zc{Please check the red sentences.}
 In both cases, we update the scaling parameters of the scaled bounds at every iteration. For o-scaling, we optimize the scalar  $\gamma$ by applying Newton steps until  the absolute value of the derivative is less than $10^{-10}$. For g-scaling, we were interested here
 in getting closer to a practical computational context.
 So, we optimize the vector   $\Upsilon$ by applying up to 10 BFGS steps, taking $\gamma \mathbf{e}$ as a starting point. In a practical computational context within B\&B, we would have a better starting point (from the parent),
 and we could then probably get away with 2-3 perhaps  BFGS steps.

 We present in the columns of Table \ref{tab:fix}, the following information from left to right: The problem considered, $n$, the range of $s$ considered,
the scaling, the number of instances solved (one for each $s$ considered), the number of instances on which we could fix at least one variable (``inst fix''), the total number of variables fixed on all instances solved (``var fix''), the \%-improvement of g-scaling over o-scaling for the two last statistics. Additionally, to better understand how well our methods works for \hyperlink{MESP}{MESP} as $n$ grows, we also experimented with a covariance matrix of order $n=300$, which is a principal submatrix of the  covariance matrix of order $n=2000$ used as a benchmark in the literature (see \cite{li2020best,chen2023computing}).   First, we see that, except for the number of instances of  \hyperlink{MESP}{MESP} with $n=124$ and $n=300$ on which we could fix variables,  there is always an improvement. The improvement becomes very significant when side constraints are considered.   We note that the number of variables fixed, reported on Table \ref{tab:fix}, refers only to the root nodes of the B{\&}B algorithm and indicates a promising approach to reduce  the B{\&}B enumeration.

% \begin{table}[ht]
% \setlength{\arraycolsep}{10pt}
% \begin{center}
% \begin{tabular}{l|cc|ccc|rr}
% %\hline
% &&&\multicolumn{3}{c|}{Number of}&\multicolumn{2}{|c}{Improvement}\\
% &$n$&scaling&$s$&inst fix&var fix&inst fix&var fix\\
% \hline
% \hyperlink{MESP}{MESP} & 63  & o & 61  & 41 & 1123   &          &          \\
%               &     & g & 61  & 42 & 1140   & 2.44\%   & 1.51\%   \\
%               & 90  & o & 88  & 41 & 1741   &          &          \\
%               &     & g & 88  & 42 & 1790  & 2.44\%   & 2.81\%   \\
%               & 124 & o & 122 & 35 & 3322 &          &          \\
%               &     & g & 122 & 35 & 3353  & 0.00\%   & 0.93\%   \\
%               & 300 & o & 41 & 41 &8382 &\\
%               &    &  g & 41 & 41 & 10753& / &  28.3\%\\
% \hline
% \ref{CMESP}   & 63  & o & 50  & 22 & 371    &          &          \\
%               &     & g & 50  & 28 & 537    & 27.27\%  & 44.74\%  \\
%               & 90  & o & 84  & 26 & 606    &          &          \\
%               &     & g & 84  & 37 & 1048   & 42.31\%  & 72.94\%  \\
%               & 124 & o & 100 & 9  & 197    &          &          \\
%               &     & g & 100 & 33 & 1120  & 266.67\% & 468.53\%
%  %             \hline
% \end{tabular}
% \end{center}
% \medskip
% \caption{Impact of g-scaling on variable fixing}\label{tab:fix}
% \end{table}

\begin{table}[ht]
\setlength{\tabcolsep}{1pt}
\begin{center}
\begin{tabular}{l|ccc|ccc|cc}
%\hline
&&&&\multicolumn{3}{c|}{Number of}&\multicolumn{2}{|c}{Improvement}\\
&$ $& & &$ $&inst &var &inst&var\\
&$n$&s&scal&$s$&fix&fix&fix&fix\\
\hline
\hyperlink{MESP}{MESP}
              & 63  & [2,62] & o & 61  & 41 & 1123   &          &          \\
              &     && g & 61  & 42 & 1140   & 2.44\%   & 1.51\%   \\
              & 90  & [2,89] & o & 88  & 41 & 1741   &          &          \\
              &     && g & 88  & 42 & 1790  & 2.44\%   & 2.81\%   \\
              & 124 & [2,123] & o & 122 & 35 & 3322 &          &          \\
              &     && g & 122 & 35 & 3353  & 0.00\%   & 0.93\%   \\
              & 300 & [80,120] & o & 41 & 41 &8382 &\\
              &     &&  g & 41 & 41 & 10753& 0.00\% &  28.3\%\\
\hline
\ref{CMESP}   & 63  & [3,52] & o & 50  & 22 & 371    &          &          \\
              &     && g & 50  & 28 & 537    & 27.27\%  & 44.74\%  \\
              & 90  & [4,87] & o & 84  & 26 & 606    &          &          \\
              &     && g & 84  & 37 & 1048   & 42.31\%  & 72.94\%  \\
              & 124 & [11,110] & o & 100 & 9  & 197    &          &          \\
              &     && g & 100 & 33 & 1120  & 266.67\% & 468.53\%
 %             \hline
\end{tabular}
\end{center}
%\medskip
\caption{Impact of g-scaling on variable fixing}\label{tab:fix}
\end{table}

%\vskip-10pt

 The experiments with the fixing methodology
 % and the  branch-and-bound algorithm
 show that  g-scaling  can effectively lead to a positive impact on the solution of \hyperlink{MESP}{MESP} and \ref{CMESP}, especially of the latter.

% We further did experiments to show the superiority of algorithms which have iterates at the boundary of $\dom\left(f_{{\tiny\mbox{DDFact}}}; \Upsilon\right)_+$ over interior-point algorithms when computing the factorization bound, thus verifying the indispensability of generalized differentiability for designing an efficient algorithm.
% By doing this, we collect the average running time of four algorithms listed as options of \texttt{Knitro} solver, i.e., interior-point algorithm, interior-point algorithm combined with projected conjugate gradient iteration, active-set algorithm, and sequential quadratic programming (SQP) algorithm. The first two algorithms confine their iterates to the interior of $ \dom\left(f_{{\tiny\mbox{DDFact}}}; \Upsilon\right)_+$ while the latter two can have their iterates at the boundary of $ \dom\left(f_{{\tiny\mbox{DDFact}}}; \Upsilon\right)_+$. In particular, the running time is averaged over $5\le s\le n-5$ for $n=63, 90, 124$ benchmark instances and over $s= 20, 40, 60, 80, 100$ for $n=2000$ benchmark instance.

We carried out further experiments to investigate the relevance of our generalized differentiability for \ref{DDFact}.
For these experiments,
we only worked with \hyperlink{MESP}{MESP},
because we wanted to better expose the non-negativity constraints to the algorithms{,
and we chose (again) factorizations
with $k$ equal to the rank of $C$, taking advantage of Remark \ref{rem:fullrank}.
% \zc{The experiments of Table \ref{table:factexperiments_time}, \ref{table:factexperiments_boundary}, and \ref{table:factexperiments_best} are done for MESP, should we add experiments for \ref{CMESP}?}.
For these experiments, we employed  all four of the \texttt{Knitro} algorithmic options:
Interior/Direct, Interior/Conjugate-Gradient(CG),
Active Set, Sequential Quadratic Programming (SQP), and chose all of the other \texttt{Knitro}
parameters as described for our first experiments.
% \mf{Please inform the stopping criteria for each algorithm and relevant parameters and  tolerances  (feasibility and optimality))}\zc{emphasize we use the same parameters as before.}
The first two algorithms have all of their iterates in the interior of $ \dom\left(f_{{\tiny\mbox{DDFact}}}; \Upsilon\right)_+$\,,
while the latter two can have  iterates at the boundary of $\dom\left(f_{{\tiny\mbox{DDFact}}}; \Upsilon\right)_+$\,.
We collected average converging times of the four algorithms in Table \ref{table:factexperiments_time}.
  In particular, the converging times are averaged over $5\le s\le n-5$ for the $n=63, 90, 124$ benchmark  covariance matrices and over $s= 20, 40, 60, 80, 100$ for the (full) $n=2000$ benchmark  covariance matrix.
  % \jon{for $n=2000$, is this the entire matrix}\zc{yes}
 %  \jon{Are these with $\Upsilon$ fixed (for each run) and how? Or is this un-scaled?} \zc{They are unscaled because we experiment on MESP.}
  To mitigate the impact of some variance in the run time for each instance, we also included in Table \ref{table:factexperiments_best} the percentage of instances $s$ for each $n$ where the convergence time of the algorithm is within $105\%$ of the convergence time of the best-performing algorithm among the four. This criterion implies that the algorithm is considered the best within a tolerance of $5\%$. Additionally, in Table \ref{table:factexperiments_boundary}, we gathered the average iterates that lie on the boundary of $\dom\left(f_{{\tiny\mbox{DDFact}}}; \Upsilon\right)_+$ for each algorithm to exhibit the relevance of generalized differentiabiliy in Definition \ref{def:generaldiff}. We use the rank function of \texttt{MATLAB} to determine the boundary iterates
  by singularity of $F_{{\tiny\mbox{DDFact}}}(x;\Upsilon)$ (equivalently, when $x$ has any zero components; see Remark \ref{rem:fullrank}). In particular, \texttt{MATLAB} asserts a matrix to be singular if the matrix has some singular value smaller than the product of the maximum of dimension lengths and the exponential of the matrix 2-norm.

  Table \ref{table:factexperiments_time} exhibits  that the active-set algorithm consistently achieves the minimum, or near-minimum, average converging time. Table \ref{table:factexperiments_best} shows that the active-set algorithm has the greatest winning percentages except for $n=124$, where the combined winning percentages of the active-set and SQP algorithms still exceed those of the other two algorithms. These outcomes indicate the superiority of algorithms that produce iterates lying on the boundary of $\dom\left(f_{{\tiny\mbox{DDFact}}}; \Upsilon\right)_+$\,, thereby emphasizing the relevance of generalized differentiability in justifying their use.
   Table \ref{table:factexperiments_boundary} reveals that for both the active-set and SQP algorithms, nearly all iterates are on the boundary of $\dom\left(f_{{\tiny\mbox{DDFact}}}; \Upsilon\right)_+$\,. We note that even the interior-point methods display iterates on the boundary of $\dom\left(f{{\tiny\mbox{DDFact}}}; \Upsilon\right)_+$\, within the tolerance. These findings underscore the relevance of generalized differentiability across all algorithms.

% \jon{Zhongzhu, please clearly state what we can conclude from each table. For Table
% \ref{table:factexperiments_boundary}, precisely state how you are calculating these numbers. I think that we agreed to use the matlab rank function, and if so, what is there criterion and  tolerance? }

\begin{table}
    \centering
    \begin{tabular}{c|C{1.7cm}|C{1.7cm}|C{1.7cm}|C{1.7cm}}
        n & Interior & Interior(CG) & Active-set & SQP  \\
         \hline
       63  &  0.09 &	0.42 &	0.11 &	0.18\\
       90  &  0.19 &	0.83 &	0.20 &	0.29\\
       124  &  0.40 &	1.63 &	0.37 &	0.44\\
       2000  &  1292.2 & 2227.39 &	96.30 &	304.60\\
    \end{tabular}
    \caption{Average converging time of each algorithm for solving \ref{DDFact}.}
    \label{table:factexperiments_time}
\end{table}

\begin{table}
    \centering
    \begin{tabular}{c|C{1.7cm}|C{1.7cm}|C{1.7cm}|C{1.7cm}}
        n & Interior & Interior(CG) & Active-set & SQP  \\
         \hline
       63  &  46.3     &    0  &  55.6    &    0\\
       90  &  48.8	& 0	& 57.3 & 1.2\\
       124  &  48.3 & 0 & 44.8  & 13.8\\
       2000  &  0 & 0 &	100.0 &	0\\
    \end{tabular}
    \caption{$\%$ of $s$ on which the algorithm converges within no more than $105\%$ converging time of the best algorithm (i.e., optimal under $5\%$ tolerance).}
    \label{table:factexperiments_best}
\end{table}

\begin{table}
    \centering
    \begin{tabular}{c|C{1.7cm}|C{1.7cm}|C{1.7cm}|C{1.7cm}}
        n & Interior & Interior(CG) & Active-set & SQP  \\
         \hline
       63  &  27.6 &	1.3 &	97.7 &	97.2\\
       90  &  35.5 &	0.8 &	98.7 &	98.6\\
       124  &  61.1 &	22.4 &	93.0 &	93.6\\
       2000  &  69.2 & 29.5 &	100.0 &	100.0\\
    \end{tabular}
    \caption{Average $\%$ of iterates
    with $x$ having any zero components,
    % at the boundary of $\dom\left(f_{{\tiny\mbox{DDFact}}}; \Upsilon\right)_+$,
    which is equivalent to the singularity of $F_{{\tiny\mbox{DDFact}}}(x;\Upsilon)$.}
    \label{table:factexperiments_boundary}
\end{table}

\section{Concluding remarks}\label{sec:conc}
We have seen that g-scaling can lead to improvements in upper bounds and variable fixing  for \hyperlink{MESP}{MESP} and very good improvements for \ref{CMESP}.
In future work, we will implement this in an efficient manner, within a B{\&}B algorithm.
In that context, it is important to efficiently use parent scaling vectors to warm-start
the optimization of scaling vectors for children (see \cite{Kurt_linx}, where this was an important issue for o-scaling in the context of the \ref{linx} bound). An open question that we wish to highlight
is whether g-scaling can help the \ref{DDFact} bound for \hyperlink{MESP}{MESP}.
Theorem \ref{thm:fact}.\emph{iv} is a partial result toward a negative answer.

Finally, we remark that there is room to do g-scaling for other bounds for \ref{CMESP}.
We did not work with g-scaling for the NLP bound. Besides the fact that we do not have a
convexity result  for o-scaling of the NLP bound as a starting point for generalizing the theory,
the o-scaling parameter is entangled with other parameters of the NLP bound which must be selected properly (even for the NLP bound to be a convex optimization problem).
For these reasons, we have left exploration of g-scaling for the NLP bound for future research.
Additionally, we did not attempt to merge the ideas of $g$-scaling with bound ``mixing'' (see \cite{chen_mixing}); this looks like another promising area for investigation.

% Finally, there is another convex-optimization bound, the so-called ``NLP bound'' (see \cite{AFLW_Using}),
% and it appears to be more difficult to get mathematical results on optimizing a g-scaling version of that
% bound; but this is a good direction to explore.

% \subsection{A Subsection Sample}
% Please note that the first paragraph of a section or subsection is
% not indented. The first paragraph that follows a table, figure,
% equation etc. does not need an indent, either.

% Subsequent paragraphs, however, are indented.

% \subsubsection{Sample Heading (Third Level)} Only two levels of
% headings should be numbered. Lower level headings remain unnumbered;
% they are formatted as run-in headings.

% \paragraph{Sample Heading (Fourth Level)}
% The contribution should contain no more than four levels of
% headings. Table~\ref{tab1} gives a summary of all heading levels.

% \begin{table}
% \caption{Table captions should be placed above the
% tables.}\label{tab1}
% \begin{tabular}{|l|l|l|}
% \hline
% Heading level &  Example & Font size and style\\
% \hline
% Title (centered) &  {\Large\bfseries Lecture Notes} & 14 point, bold\\
% 1st-level heading &  {\large\bfseries 1 Introduction} & 12 point, bold\\
% 2nd-level heading & {\bfseries 2.1 Printing Area} & 10 point, bold\\
% 3rd-level heading & {\bfseries Run-in Heading in Bold.} Text follows & 10 point, bold\\
% 4th-level heading & {\itshape Lowest Level Heading.} Text follows & 10 point, italic\\
% \hline
% \end{tabular}
% \end{table}

\section*{Acknowledgements}
We are especially grateful to Kurt Anstreicher for suggesting the possibility of generalizing (ordinary) scaling for the linx bound.
M. Fampa was supported in part by CNPq grants 305444/2019-0 and 434683/2018-3.  J. Lee was supported in part by AFOSR grant FA9550-22-1-0172. This work is partially based upon work supported by the
National Science Foundation under Grant No. DMS-1929284 while
the authors were in residence at the Institute for Computational and Experimental Research in
Mathematics (ICERM) at Providence, RI, during the Discrete Optimization program.

%\section{Appendix}\label{appendix}

\FloatBarrier

% ----------------------------------------------------------------
\bibliographystyle{alpha}

\bibliography{MO}

\begin{thebibliography}{DCMT02}

\bibitem[AFLW96]{AFLW_IPCO}
Kurt~M. Anstreicher, Marcia Fampa, Jon Lee, and Joy Williams.
\newblock Continuous relaxations for constrained maximum-entropy sampling.
\newblock In {\em Integer Programming and Combinatorial Optimization ({V}ancouver, {BC}, 1996)}, volume 1084 of {\em Lecture Notes in Computer Science}, pages 234--248. Springer, Berlin, 1996.

\bibitem[AFLW99]{AFLW_Using}
Kurt~M. Anstreicher, Marcia Fampa, Jon Lee, and Joy Williams.
\newblock Using continuous nonlinear relaxations to solve constrained maximum-entropy sampling problems.
\newblock {\em Mathematical Programming, Series A}, 85(2):221--240, 1999.

\bibitem[AL04]{AnstreicherLee_Masked}
Kurt~M. Anstreicher and Jon Lee.
\newblock A masked spectral bound for maximum-entropy sampling.
\newblock In {\em m{OD}a 7---{A}dvances in Model-Oriented Design and Analysis}, Contrib. Statist., pages 1--12. Physica, Heidelberg, 2004.

\bibitem[Ans18]{Anstreicher_BQP_entropy}
Kurt~M. Anstreicher.
\newblock {Maximum-entropy sampling and the Boolean quadric polytope}.
\newblock {\em Journal of Global Optimization}, 72(4):603--618, 2018.

\bibitem[Ans20]{Kurt_linx}
Kurt~M. Anstreicher.
\newblock Efficient solution of maximum-entropy sampling problems.
\newblock {\em Operations Research}, 68(6):1826--1835, 2020.

\bibitem[ATL20]{Al-ThaniLee1}
Hessa Al-Thani and Jon Lee.
\newblock An {R} package for generating covariance matrices for maximum-entropy sampling from precipitation chemistry data.
\newblock {\em SN Operations Research Forum}, Volume 1:Article 17 (21 pages), 2020.
\newblock \url{https://doi.org/10.1007/s43069-020-0011-z}.

\bibitem[ATL21]{AlThaniLeeTridiag}
Hessa Al-Thani and Jon Lee.
\newblock Tridiagonal maximum-entropy sampling and tridiagonal masks.
\newblock {\em LAGOS 2021 proceedings, Procedia Computer Science}, 195:127--134, 2021.

\bibitem[ATL23]{AlThaniLeeTridiag_journal}
Hessa Al-Thani and Jon Lee.
\newblock Tridiagonal maximum-entropy sampling and tridiagonal masks.
\newblock {\em Discrete Applied Mathematics}, 337:120--138, 2023.

\bibitem[BH90]{DOE3}
George Box and Ian Hau.
\newblock Constrained experimental designs part i: Construction of projection designs, 1990.
\newblock Technical report, \url{http://digital.library.wisc.edu/1793/69117}.

\bibitem[BL07]{BurerLee}
Samuel Burer and Jon Lee.
\newblock Solving maximum-entropy sampling problems using factored masks.
\newblock {\em Mathematical Programming, Series B}, 109(2--3):263--281, 2007.

\bibitem[CF95]{DOE1}
{R.D.} Cook and V.~Fedorov.
\newblock Constrained optimization of experimental design (invited with discussion).
\newblock {\em Statistics}, 26:129--178, 1995.

\bibitem[CFL22]{chen2022masking}
Zhongzhu Chen, Marcia Fampa, and Jon Lee.
\newblock Masking {A}nstreicher’s linx bound for improved entropy bounds.
\newblock {\em Operations Research}, 2022.

\bibitem[CFL23a]{ACDA2023}
Zhongzhu Chen, Marcia Fampa, and Jon Lee.
\newblock Generalized scaling for the constrained maximum-entropy sampling problem.
\newblock In {\em SIAM Conference on Applied and Computational Discrete Algorithms (ACDA23)}, pages 110--118. SIAM, 2023.
\newblock \url{https://doi.org/10.1137/1.9781611977714.10}.

\bibitem[CFL23b]{chen2023computing}
Zhongzhu Chen, Marcia Fampa, and Jon Lee.
\newblock On computing with some convex relaxations for the maximum-entropy sampling problem.
\newblock {\em INFORMS Journal on Computing}, 35(2):368--385, 2023.

\bibitem[CFLL21]{chen_mixing}
Zhongzhu Chen, Marcia Fampa, Am\'elie Lambert, and Jon Lee.
\newblock Mixing convex-optimization bounds for maximum-entropy sampling.
\newblock {\em Mathematical Programming, Series B}, 188:539--568, 2021.

\bibitem[DCMT02]{DOE2}
Duangporn~Jearkpaporn Douglas C.~Montgomery, Elvira N.~Loredo and Murat~Caner Testik.
\newblock Experimental designs for constrained regions.
\newblock {\em Quality Engineering}, 14(4):587--601, 2002.

\bibitem[Fam96]{FampaPhD}
Marcia Helena~Costa Fampa.
\newblock {\em Relaxa\c{c}\~{o}es Contínuas para o Problema da Amostra de M\'{a}xima Entropia Restrito e um Algoritmo de Trajet\'{o}ria Central de Passos Longos para Problemas de Programa\c{c}\~{a}o Semidefinida}.
\newblock {D.Sc.}, Universidade Federal do Rio de Janeiro, November 1996.
\newblock \url{https://www.cos.ufrj.br/uploadfile/1339609553.pdf}.

\bibitem[FI90]{fiacco1990sensitivity}
Anthony~V. Fiacco and Yo~Ishizuka.
\newblock Sensitivity and stability analysis for nonlinear programming.
\newblock {\em Annals of Operations Research}, 27(1):215--235, 1990.

\bibitem[FL22]{FLbook}
Marcia Fampa and Jon Lee.
\newblock {\em Maximum-Entropy Sampling: Algorithms and Application}.
\newblock Springer International Publishing, 2022.

\bibitem[KLQ95]{KLQ}
Chun-Wa Ko, Jon Lee, and Maurice Queyranne.
\newblock An exact algorithm for maximum-entropy sampling.
\newblock {\em Operations Research}, 43(4):684--691, 1995.

\bibitem[Lee98]{LeeConstrained}
Jon Lee.
\newblock Constrained maximum-entropy sampling.
\newblock {\em Operations Research}, 46(5):655--664, 1998.

\bibitem[LO13]{nonsmoothBFGS}
Adrian~S. Lewis and Michael~L. Overton.
\newblock Nonsmooth optimization via quasi-{N}ewton methods.
\newblock {\em Mathematical Programming}, 141:135--163, 2013.

\bibitem[LX23]{li2020best}
Yongchun Li and Weijun Xie.
\newblock Best principal submatrix selection for the maximum entropy sampling problem: Scalable algorithms and performance guarantees.
\newblock {\em Operations Research}, 2023.
\newblock \url{https://doi.org/10.1287/opre.2023.2488}.

\bibitem[Mor66]{moreau1966fonctionnelles}
Jean-Jacques Moreau.
\newblock Fonctionnelles convexes.
\newblock {\em S{\'e}minaire Jean Leray}, 2:1--108, 1966.

\bibitem[Nik15]{nikolov2015randomized}
Aleksandar Nikolov.
\newblock Randomized rounding for the largest simplex problem.
\newblock In {\em Proceedings of the 47th Annual ACM Symposium on Theory of Computing}, pages 861--870, 2015.

\bibitem[OT18]{oyama2018non}
Daisuke Oyama and Tomoyuki Takenawa.
\newblock On the (non-)differentiability of the optimal value function when the optimal solution is unique.
\newblock {\em Journal of Mathematical Economics}, 76:21--32, 2018.

\bibitem[Roc97]{rockafellar1997convex}
R.~Tyrrell Rockafellar.
\newblock {\em Convex Analysis}.
\newblock Princeton Mathematical Series. Princeton University Press, 1997.

\bibitem[Sch11]{Schur}
J.~Schur.
\newblock Bemerkungen zur theorie der beschränkten bilinearformen mit unendlich vielen veränderlichen.
\newblock {\em Journal für die reine und angewandte Mathematik}, 1911(140):1--28, 1911.

\bibitem[Sha48]{Shannon}
Claude~E. Shannon.
\newblock A mathematical theory of communication.
\newblock {\em The Bell System Technical Journal}, 27(3):379--423, 1948.

\bibitem[SW87]{SW}
Michael~C. Shewry and Henry~P. Wynn.
\newblock Maximum entropy sampling.
\newblock {\em Journal of Applied Statistics}, 46:165--170, 1987.

\bibitem[Whi34]{Whitney}
Hassler Whitney.
\newblock Analytic extensions of differentiable functions defined in closed sets.
\newblock {\em Transactions of the American Mathematical Society}, 36(1):63--89, 1934.

\bibitem[Wil98]{WilliamsPhD}
Joy~Denise Williams.
\newblock {\em Spectral Bounds for Entropy Models}.
\newblock {Ph.D.}, University of Kentucky, April 1998.

\bibitem[Zal02]{zalinescu2002convex}
Constantin Zalinescu.
\newblock {\em Convex Analysis in General Vector Spaces}.
\newblock World Scientific, 2002.

\end{thebibliography}

\appendix
\section*{Appendix}\label{sec:App}

We give a small numerical example to demonstrate how a g-scaling bound
can be better than the optimal o-scaling bound. Consider the (randomly generated) positive-definite matrix
\begin{equation*}
    C := \begin{pmatrix*}[r]
         5.69 & -1.34 & -0.93 & -0.61 \\
         -1.34 & 3.49 & 1.36 & 0.45 \\
         -0.93 & 1.36 & 3.71 & 2.25 \\
         -0.61 & 0.45 & 2.25 & 3.13 \\
    \end{pmatrix*},
\end{equation*}
with $n=4$. We work with the \ref{linx} bound, as it is the  easiest for an interested reader to check (using software like \texttt{CVX} or \texttt{YALMIP}, for example).

We took $s=2$, and it easy  to check that
among the $\binom{4}{2}\!=\!6$ feasible solutions,
the optimal one is $\hat{x}=(1,0,1,0)^\top$ with optimal value
$3.0079$. We calculated
the optimal o-scaling parameter $\gamma^*=0.0963$ (which can be verified by perturbation, because of convexity in the scaling parameter), with the corresponding optimal solution $x^*=(1.0000, 0.2423, 0.7577, 0.0000)^\top$. This gives  an o-scaled \ref{linx} bound of $3.0210$ with an integrality gap of $0.0131$.  On the other hand, employing the g-scaling parameter $\Upsilon^* = (0.1215, 0.3306, 0.3358, 0.2839)^\top$, the (optimal) g-scaled \ref{linx}  solution is $x^*=(1.0000, 0.1850, 0.7319, 0.0831)^\top$,
which results in a bound of $3.0146$ and a reduced integrality gap of $0.0067$. So, we
can see that for this example, g-scaling
reduces the gap by about 50\%.

% The eigenvalues of $\gamma^* C$ are $0.0963$, $0.7387$, $2.0309$, and $5.6495$, signifying a non-one spectrum. Therefore, via references \cite[Remark 5]{chen_mixing} and \cite[Theorem 21]{chen_mixing}, we can confirm that the o-scaled \ref{linx} bound, at $\gamma^*$, possesses a unique solution and thus is differentiable at $\gamma^*$. Using the gradient expression from \cite[Theorem 18]{chen_mixing}, we validate $\gamma^*$ as the optimal parameter, yielding an o-scaled \ref{linx} bound of $3.0210$ with an integrality gap of $0.0131$. On the other hand, employing the g-scaling parameter $\Upsilon = (0.1215, 0.3306, 0.3358, 0.2839)^\top$, the optimal g-scaled \ref{linx} bound solution is $x^*=(1.0000, 0.1850, 0.7319, 0.0831)^\top$, computed via KKT conditions, which results in a bound of $3.0146$ and a reduced integrality gap of $0.0067$. This comparison illustrates the substantial enhancements afforded by g-scaling. In fact, this effect is more significant for larger problem dimensions.

\end{document}